\documentclass[10pt]{article}
\usepackage{latexsym,amsfonts,amssymb,amsmath,amsthm}
\usepackage{graphicx}
\usepackage[dvipsnames]{xcolor}
\usepackage{tikz}
\usepackage{placeins}
\usetikzlibrary{arrows.meta,decorations.markings,patterns}
\definecolor{diagramblue}{RGB}{38,76,115}
\definecolor{diagramred}{RGB}{151,58,48}
\tikzset{orbit/.style={line width=.85pt,draw=diagramblue,
  postaction={decorate},decoration={markings,
  mark=at position .52 with {\arrow{Stealth}}}}}
\usepackage{cite}
\usepackage{mathtools}
\usepackage{fancyhdr}
\usepackage[colorlinks=true,hypertexnames=false]{hyperref}
\usepackage{bm}

\newcounter{appsec}

\renewcommand{\theequation}{\thesection.\arabic{equation}}

\usepackage[margin=0.8in]{geometry}
\begin{document}
	\setlength{\baselineskip}{13pt}
	
	\parindent 0.5cm
	\evensidemargin 0cm \oddsidemargin 0cm \topmargin 0cm \textheight
	22cm \textwidth 16cm \footskip 2cm \headsep 0cm
	
	\newtheorem{thm}{Theorem}[section]
	\newtheorem{cor}{Corollary}[section]
	\newtheorem{lem}{Lemma}[section]
	\newtheorem{prop}{Proposition}[section]
	\newtheorem{defn}{Definition}[section]
	\newtheorem{rk}{Remark}[section]
	\newtheorem*{claim}{Claim}
	\newtheorem{nota}{Notation}[section]
	\newtheorem{Ex}{Example}[section]
	\def\nm{\noalign{\medskip}}
	
	\numberwithin{equation}{section}
	
	\def\p{\partial}
	\def\I{\textit}
	\def\R{\mathbb R}
	\def\C{\mathbb C}
	\def\u{\underline}
	\def\l{\lambda}
	\def\a{\alpha}
	\def\O{\Omega}
	\def\e{\epsilon}
	\def\ls{\lambda^*}
	\def\D{\displaystyle}
	\def\wyx{ \frac{w(y,t)}{w(x,t)}}
	\def\imp{\Rightarrow}
	\def\tE{\tilde E}
	\def\tX{\tilde X}
	\def\tH{\tilde H}
	\def\tu{\tilde u}
	\def\d{\mathcal D}
	\def\aa{\mathcal A}
	\def\DH{\mathcal D(\tH)}
	\def\bE{\bar E}
	\def\bH{\bar H}
	\def\M{\mathcal M}
	\renewcommand{\labelenumi}{(\arabic{enumi})}

	\def\disp{\displaystyle}
	\def\undertex#1{$\underline{\hbox{#1}}$}
	\def\card{\mathop{\hbox{card}}}
	\def\sgn{\mathop{\hbox{sgn}}}
	\def\exp{\mathop{\hbox{exp}}}
	\def\OFP{(\Omega,{\cal F},\PP)}
	\newcommand\JM{Mierczy\'nski}
	\newcommand\RR{\ensuremath{\mathbb{R}}}
	\newcommand\CC{\ensuremath{\mathbb{C}}}
	\newcommand\QQ{\ensuremath{\mathbb{Q}}}
	\newcommand\ZZ{\ensuremath{\mathbb{Z}}}
	\newcommand\NN{\ensuremath{\mathbb{N}}}
	\newcommand\PP{\ensuremath{\mathbb{P}}}
	\newcommand\abs[1]{\ensuremath{\lvert#1\rvert}}
	
	\newcommand\normf[1]{\ensuremath{\lVert#1\rVert_{f}}}
	\newcommand\normfRb[1]{\ensuremath{\lVert#1\rVert_{f,R_b}}}
	\newcommand\normfRbone[1]{\ensuremath{\lVert#1\rVert_{f, R_{b_1}}}}
	\newcommand\normfRbtwo[1]{\ensuremath{\lVert#1\rVert_{f,R_{b_2}}}}
	\newcommand\normtwo[1]{\ensuremath{\lVert#1\rVert_{2}}}
	\newcommand\norminfty[1]{\ensuremath{\lVert#1\rVert_{\infty}}}

	\title{Structural stability for the bidirectional cyclic negative feedback systems}
	
	\author{
Xu Cheng\thanks{Partially supported by the Postgraduate Research \& Practice Innovation Program of Jiangsu Province (No. KYCX24\_0620) and the China Scholarship Council (No. 202406840118).}
\quad
Dun Zhou\thanks{Corresponding author. Email: zhoudun@njust.edu.cn. Partially supported by the NSF of China (Nos.12671213,12331006).}\\[1ex]
School of Mathematics and Statistics\\
Nanjing University of Science and Technology\\
Nanjing, Jiangsu 210094, P. R. China
}
	\date{}
	
	\maketitle
	
	\let\thefootnote\relax
\begin{abstract}
We establish the generic Kupka--Smale property for bidirectional cyclic
negative feedback systems on open convex domains in the compact-open
$C^1$ topology. Under an inward-pointing dissipativity condition, we
also prove the generic Morse--Smale property in the relative topology
of the dissipative class. A key ingredient is automatic transversality
between the stable and unstable manifolds of hyperbolic critical
elements, except possibly when both elements are equilibria with the
same odd Morse index. The proofs use an integer-valued Lyapunov
function, a Floquet-type spectral decomposition, and perturbations
that preserve the feedback structure. Under a stronger dissipativity
condition, we further show that the equilibrium-free
Poincar\'{e}--Bendixson property persists under sufficiently small
multiples of bounded $C^1$ perturbations, even when these perturbations
do not preserve the feedback structure.
\end{abstract}

	\section{Introduction}
Cyclic negative feedback is a common mechanism for oscillation in
biological and biochemical systems, with examples arising in circadian
rhythms, neuronal activity, enzymatic reactions, and gene regulation
\cite{MBE,JEF,SHJ,TYS}. Representative models include the Goodwin
oscillator \cite{Good}, the repressilator \cite{MBE,MuHo}, the metabolator
\cite{Fung}, and the Frzilator \cite{Igo}, which describe oscillatory
behavior in a variety of regulatory and signaling processes.
Negative feedback also plays an important role in Boolean network
dynamics and biomolecular control \cite{Sontag1,BGK}. These applications
provide a broader motivation for studying feedback systems, although
not all of them fall within the continuous-time cyclic framework
considered here.
	
	Motivated by cyclic feedback models (see \cite{AS,MBE,Igo,KGM}), we consider continuous-time systems in which each component interacts only with itself and its two cyclic neighbors:
	\begin{equation}\label{Three independent variables}
		\dot{x}_i=f_i\left(x_{i-1},x_i,x_{i+1}\right),\quad i=1,2,\dots,n,
	\end{equation}
	subject to the sign conditions
	
	\begin{equation}\label{generalized-f}
		\begin{split}
			& \delta_{i+1}\frac{\partial f_i}{\partial x_{i+1} }   \ge 0,\, \delta _{i}\frac{\partial f_i}{\partial x_{i-1} }   >0,\quad 
			i=1,\dots,n; \\
			\text{or }
			& \delta _{i+1}\frac{\partial f_i}{\partial x_{i+1} }  >0,\, \delta_{i}\frac{\partial f_i}{\partial x_{i-1} }  \ge 0,\quad 
			i=1,\dots,n;
		\end{split}
	\end{equation}
	Here $n\ge3$, $f=(f_1,\ldots,f_n)\in C^1(\Omega,\mathbb R^n)$,
and $\Omega\subset\mathbb R^n$ is nonempty, open, and convex.
Indices are interpreted cyclically: $x_0=x_n$, $x_{n+1}=x_1$, and
$\delta_{n+1}=\delta_1$, where $\delta_i\in\{-1,1\}$.
At each $x\in\Omega$, at least one of the two collections of
inequalities in \eqref{generalized-f} holds; the choice may depend on $x$.
The system is called a bidirectional negative feedback system if
$\Delta:=\prod_{i=1}^n\delta_i=-1$, and a bidirectional positive
feedback system if $\Delta=1$.

	For $\Delta=-1$, we study transversality and genericity for
system~\eqref{Three independent variables} under the sign conditions
\eqref{generalized-f}. We establish automatic transversality,
generic hyperbolicity of critical elements, and the generic
Kupka--Smale property. Under an inward-pointing dissipativity
condition, we also prove the generic Morse--Smale property in the
relative compact-open $C^1$ topology
(see Theorems~\ref{Automatic Transversality},
\ref{critical elements is generic},
\ref{Generic non-existence of homoindexed connecting orbits}, and
\ref{GenericityofMorse–Smaleproperty}).

	Structural stability concerns the persistence of qualitative dynamics
under perturbations. Following Andronov and Pontryagin \cite{AL}, a vector
field is structurally stable, relative to a specified phase space and
perturbation topology, if every sufficiently close vector field is
topologically orbit equivalent to it, with the orientation of time
preserved. A reparametrization of time along trajectories is allowed.

A vector field is Morse--Smale if its nonwandering set consists of
finitely many hyperbolic equilibria and periodic orbits, and the stable
and unstable manifolds of its critical elements intersect transversely
\cite{SS}. Palis and Smale \cite{JP,JPS} proved structural stability of
Morse--Smale vector fields on compact manifolds. On orientable closed
surfaces, Peixoto's theorem identifies structural stability with the
Morse--Smale property and establishes its $C^1$ genericity \cite{PM}.
In higher dimensions, Morse--Smale vector fields need not be dense.
For example, transverse homoclinic orbits can produce horseshoes and
thereby obstruct the Morse--Smale property \cite{Smale-65,SSS}.

	A vector field is called Kupka--Smale if all its equilibria and periodic orbits are hyperbolic and the stable and unstable manifolds of any two critical elements intersect transversely. The classical Kupka--Smale theorem asserts that this property is generic in the appropriate $C^r$ spaces of vector fields; see \cite{SS-63,KI,PM-66}. A separate argument is needed for the class considered here, because unrestricted perturbations need not preserve the feedback structure.
	
	Structural stability has also been established for classes of infinite-dimensional Morse--Smale systems, including dissipative parabolic equations, under suitable compactness and smoothing assumptions \cite{Oliva2000,84Hale}.

	Although the Morse--Smale property is known for several classes of higher-dimensional and infinite-dimensional systems \cite{Fusco1987,YDZ,MP,85Hen,MX,BP,BG,RGR}, its verification generally requires additional structure. For cyclic feedback systems, integer-valued Lyapunov functions and Poincar\'{e}--Bendixson theorems provide such structure. The results of Margaliot, Wu, and Sontag \cite{MS} and Weiss and Margaliot \cite{EM} are particularly relevant to the bidirectional negative feedback systems considered here.
	
	\subsection{Background on cyclic feedback systems}
	When $\Delta=1$, the system is strongly monotone with respect to a suitable orthant order \cite{Smi95}. In the autonomous case, generic convergence results describe the asymptotic behavior of precompact positive orbits; see Smith and Thieme \cite{ST91}. For the cyclic tridiagonal systems considered in \cite{MP}, an equilibrium-free $\omega$-limit set of a precompact positive orbit is a periodic orbit. Generic Morse--Smale results were established in \cite{FOT,MP}. For time-periodic systems, typical orbits converge to linearly stable periodic orbits under the compactness and regularity assumptions of \cite{Pol92}.

	For $\Delta=-1$, the system is no longer monotone in this sense.
Mallet-Paret and Smith \cite{JMP} proved a Poincar\'{e}--Bendixson theorem
for unidirectional cyclic negative feedback systems. Gedeon \cite{Gedeon}
studied Morse decompositions of their global attractors. Global
stability criteria under additional hypotheses are given in \cite{LPD}.
Elkhader \cite{elkhader1992_sub} extended the Poincar\'{e}--Bendixson
theorem to a broader class under a $C^{n-1}$ regularity assumption.
The strongly $2$-cooperative structure of the present system gives the
equilibrium-free Poincar\'{e}--Bendixson property \cite{MS,EM}.
Gao and Zhou \cite{MD} developed an integer-valued Lyapunov function
for bidirectional cyclic negative feedback systems and used it to
obtain planar embeddings of $\omega$-limit sets of Poincar\'{e}
maps. We recall this function and its properties in Section~2.
For the unidirectional systems of \cite{JMP}, Wang and Zhou \cite{YD}
established transversality results for invariant manifolds of
hyperbolic critical elements. The present paper extends this analysis
to the bidirectional setting and proves genericity within the
corresponding structured class of vector fields.

	A change of signs in the coordinates puts the negative feedback conditions
\eqref{generalized-f} in the normalized form 
	
	\begin{equation}\label{transformation}
		\begin{split} 
			& \frac{\partial f_n}{\partial x_{1} } \leq 0,\, \frac{\partial f_1}{\partial x_{n} }< 0,\, \frac{\partial f_i}{\partial x_{i+1} }\ge 0,\, \frac{\partial f_{i+1}}{\partial x_{i} } >0, \quad
			i=1,\dots,n-1;\\
			\text{or }
			& \frac{\partial f_n}{\partial x_{1} }< 0,\, \frac{\partial f_1}{\partial x_{n} }\leq 0, \, \frac{\partial f_i}{\partial x_{i+1} } > 0,\, \frac{\partial f_{i+1}}{\partial x_{i} } \ge 0, \quad
			i=1,\dots,n-1;
		\end{split}
	\end{equation}
	The coordinate transformation is described in \cite[Section~5]{MD}.
We use this sign convention throughout the paper.

	\subsection{Main results}\label{Main results}
	We first introduce the notation needed to state the main results.
We use $\|\cdot\|$ for the Euclidean norm and its induced operator
norm, and also write $|x|$ for the Euclidean norm of a vector $x$.
For a bounded vector field $g$, set
$\|g\|_\infty:=\sup_{x\in\Omega}\|g(x)\|$.
Let $\mathcal M$ be the set of $n\times n$ matrices of the form
	\begin{equation}\label{special matrix}
		\left[\begin{array}{cccccc}
			a_1 & b_1 & 0 & \cdots & 0 & c_n \\
			c_1 & a_2 & b_2 & \ddots & \ddots & 0 \\
			0 & c_2 & \ddots & \ddots & \ddots & \vdots \\
			\vdots & \ddots & \ddots & \ddots & b_{n-2} & 0 \\
			0 & \ddots & \ddots & c_{n-2} & a_{n-1} & b_{n-1} \\
			b_n & 0 & \cdots & 0 & c_{n-1} & a_n
		\end{array}\right]
	\end{equation}
	for some $a_i, b_i, c_i$ satisfying
	\begin{equation}\label{bid-feedback}
		b_ic_i \ge 0, \quad i=1, \ldots, n ; \quad \prod_{i=1}^n b_i+\prod_{i=1}^n c_i\ne 0.
	\end{equation}
	The normalized negative feedback condition leads to the subset of $\mathcal{M}$ given by
	\begin{align*}
		\mathcal{M}^-=\{A\in\mathcal{M}\mid b_i>0,i=1,\cdots,n-1, b_n<0\}\cup\{A\in\mathcal{M}| c_i>0,i=1,\cdots,n-1, c_n<0\}.
	\end{align*}
	Define
	$$
	\begin{gathered}
		\mathcal C^1_{\mathrm{BF}}=\left\{f\in C^1(\Omega,\mathbb R^n):
\frac{\partial f_i}{\partial x_j}=0\text{ on }\Omega
\text{ whenever }j\notin\{i-1,i,i+1\},\ 1\le i\le n\right\},
	\end{gathered}
	$$
	where indices are interpreted cyclically. Since $\Omega$ is convex, this derivative condition is equivalent to requiring that $f_i$ depend only on $(x_{i-1},x_i,x_{i+1})$. We then define the corresponding class of vector fields by 
	$$
	\mathcal{L}^-:=\left\{f\in \mathcal C^1_{\mathrm{BF}}\mid \forall x\in \Omega,\,D{f} \left ( x \right ) \in \mathcal{M}^- \right\}.
	$$
	Thus the systems satisfying \eqref{Three independent variables} and \eqref{transformation} can be written as
	\begin{equation}\label{negative feedback system}
		\dot{x}=f\left(x\right),\,D{f} \left ( x \right ) \in \mathcal{M}^-.
	\end{equation}

	\begin{defn}\label{HCE}
		\rm{\textbf{(Hyperbolic critical elements)}
An equilibrium $e$ of \eqref{negative feedback system} is hyperbolic
if every eigenvalue of $Df(e)$ has nonzero real part.
Let $p(t)$ be a nonconstant periodic solution with period $\omega$, and
let $S_{Df(p)}(t,s)$ denote the solution operator from time $s$
to time $t$ of the variational equation along $p$. The periodic
orbit is hyperbolic if $1$ is an algebraically simple eigenvalue of $S_{Df(p)}(\omega,0)$ and all other
eigenvalues lie off the unit circle. Hyperbolic equilibria and hyperbolic
periodic orbits are collectively called hyperbolic critical elements.}
\end{defn}

    {Fix $f\in\mathcal L^-$. For $x\in\Omega$, let $\phi_f(t,x)$ denote the
maximal solution of $\dot x=f(x)$ satisfying $\phi_f(0,x)=x$, with
maximal existence interval
$I_f(x)=(\tau_f^-(x),\tau_f^+(x))$. We also write
$S_f(t)x=\phi_f(t,x)$ whenever $t\in I_f(x)$.
For a hyperbolic critical element $\gamma$, set
$d(z,\gamma):=\inf_{y\in\gamma}\|z-y\|$ and define its global stable
and unstable manifolds by
\[
\begin{aligned}
W^s(\gamma)
&=\left\{x\in\Omega:
\tau_f^+(x)=+\infty,\quad
\lim_{t\to+\infty}d(\phi_f(t,x),\gamma)=0\right\},\\
W^u(\gamma)
&=\left\{x\in\Omega:
\tau_f^-(x)=-\infty,\quad
\lim_{t\to-\infty}d(\phi_f(t,x),\gamma)=0\right\}.
\end{aligned}
\]
The global manifolds $W^s(\gamma)$ and $W^u(\gamma)$ are $C^1$
injectively immersed submanifolds of $\Omega$; see
\cite[Chapter~1]{CC}. 

We write $W^s_{\mathrm{loc}}(\gamma)$ and
$W^u_{\mathrm{loc}}(\gamma)$ for the corresponding local
manifolds, described at the beginning of Section~5.

Let $i(\gamma)$ denote the Morse index of $\gamma$; see, for example,
\cite[Chapter~6]{Hirsch} or \cite[Chapter~1]{CCCCI}.
If $\gamma$ is an equilibrium, then $i(\gamma)$ is the number of
eigenvalues of $e^{Df(\gamma)}$ whose modulus is greater than one.
If $\gamma$ is a periodic orbit with minimal period $\omega>0$, then
$i(\gamma)$ is the number of eigenvalues of the monodromy operator
$S_{Df(\gamma)}(\omega,0)$ whose modulus is greater than one.
In both cases, eigenvalues are counted with their algebraic
multiplicities. In particular, the trivial Floquet multiplier $1$
associated with the flow direction is not counted in the Morse index
of a periodic orbit. 

If $\gamma$ is an equilibrium, then $\dim W^s_{\mathrm{loc}}(\gamma)=\dim W^s(\gamma)=n-i(\gamma)$ and $\dim W^u_{\mathrm{loc}}(\gamma)=\dim W^u(\gamma)=i(\gamma)$.
If $\gamma$ is a periodic orbit, then $\dim W^s_{\mathrm{loc}}(\gamma)=\dim W^s(\gamma)=n-i(\gamma)$ and $\dim W^u_{\mathrm{loc}}(\gamma)=\dim W^u(\gamma)=i(\gamma)+1$.
For further details on invariant-manifold theory, see
\cite{RCR,JKHG}, \cite[Theorem~6.1.9]{H1},
\cite[Appendix~C]{MX}, \cite[Theorem~4.1]{CC}, and
\cite[Theorem~14.2 and Remark~14.3]{DR}.}
	\begin{defn}\label{Transitivity}
		\rm{\textbf{(Transversality)}
Two $C^1$ submanifolds $M,N\subset\mathbb R^n$ are transverse at
$a\in M\cap N$ if $T_aM+T_aN=\mathbb R^n$.
They are transverse, written $M\pitchfork N$, if this condition holds
at every point of $M\cap N$. In particular, disjoint submanifolds are
transverse.}
\end{defn}
	
	Our first main result establishes automatic transversality.
	
	\begin{thm}\label{Automatic Transversality}
		\textbf{(Automatic transversality)}
		Let $\gamma ^-$, $\gamma ^+$ be hyperbolic critical elements of \eqref{negative feedback system}. 
		Then 
		$$
		W^u\left(\gamma^{-}\right) \pitchfork W^s\left(\gamma^{+}\right) ,
		$$
		provided that one of the following conditions holds:
		\begin{enumerate}
			\item[{\rm(i)}] $\gamma ^-$ or $\gamma ^+$ is a periodic orbit;
			
			\item[{\rm(ii)}] $\gamma ^-$ and $\gamma ^+$ are equilibria satisfying $i\left(\gamma^{-}\right) \ge 2h \ge 
			i\left(\gamma^{+}\right)$, for some  integer $h\in \left \{ 1,\dots ,\frac{\widetilde n+1 }{2}  \right \}$. Here, if $n$ is odd, $n=\widetilde n$; if 
			$n$ is even, 
			$n=\widetilde n+1$.
		\end{enumerate}
	\end{thm}
	
	The next two theorems establish the generic Kupka--Smale property.
A property is called generic if it holds on a residual subset of the
specified space of vector fields. Throughout the paper, residual means
containing a countable intersection of open dense sets. We use the
compact-open $C^1$ topology, defined precisely in Section~\ref{GTC}.

\begin{thm}\label{critical elements is generic}
		\textbf{(Generic hyperbolicity of critical elements)}
		There exists a generic subset $\mathcal{O}\subset \mathcal{L}^-$ such that, for every $f\in \mathcal{O}$, all critical elements of \eqref{negative feedback system} are hyperbolic.
	\end{thm}
	
	Theorem~\ref{Automatic Transversality} leaves open connections between equilibria with the same odd Morse index. The following theorem shows that such connections are generically absent.
	
	\begin{thm}\label{Generic non-existence of homoindexed connecting orbits}
		\textbf{(Generic Kupka--Smale property)}
		There exists a generic subset $\tilde{\mathcal{O}}\subset\mathcal{O}$ such that, for every $f\in \tilde{\mathcal{O}}$, no nonconstant solution $u(t)$ of \eqref{negative feedback system} connects two equilibria with the same Morse index. Consequently, every $f\in\tilde{\mathcal O}$ is Kupka--Smale and has no homoclinic orbits.
	\end{thm}
	
	{For the generic Morse--Smale result, we introduce the following
dissipative subclass of $\mathcal L^-$:
\begin{equation}
\mathcal L^-_d
:=
\left\{
f\in\mathcal L^- :
\begin{array}{l}
\text{there exist } a\in\Omega \text{ and } R>0
\text{ such that } B(a,R)\Subset\Omega,\\[1mm]
\langle f(x),x-a\rangle<0
\quad\text{for every }x\in\Omega\setminus B(a,R)
\end{array}
\right\}.
\tag{D}
\end{equation}
Here both the center $a$ and the radius $R$ are allowed to depend on
$f$. We equip $\mathcal L^-_d$ with the relative topology
$\iota_d:=\iota|_{\mathcal L^-_d}$, where $\iota$ is the topology on
$\mathcal L^-$ introduced in Section~\ref{GTC}.

A point $x\in\Omega$ is called nonwandering if, for every neighborhood
$U$ of $x$ and every $T>0$, there exist $y\in U$ and $t>T$ such that
$\phi_f(s,y)$ is defined for all $0\leq s\leq t$ and
$\phi_f(t,y)\in U$.}

	\begin{thm}\label{GenericityofMorse–Smaleproperty}
		\textbf{(Generic Morse--Smale property)} There exists a dense $G_\delta$
subset $\mathcal{O}_M$ of $(\mathcal{L}^-_d,\iota_d)$ such that, for every
$f\in\mathcal{O}_M$, system~\eqref{negative feedback system} is Morse--Smale; that is, 
		\begin{itemize}
			\item there are finitely many critical elements, all of which are hyperbolic;
			\item $W^s\left(\sigma_1\right)$ and $W^u\left(\sigma_2\right)$ intersect transversely for any two critical elements $\sigma_1$ and $\sigma_2$;
			\item the nonwandering set equals the union of the critical elements.
		\end{itemize}
	\end{thm}

	\subsection{Remarks}
	We comment on the scope of the results and the main ingredients of the proofs.

	1) {\it Remark on the Morse--Smale property.}
The Morse--Smale property is difficult to verify in high dimensions.
For bidirectional cyclic negative feedback systems, the
transversality results in \cite{JMP,YD} cover only part of the problem.
Theorem~\ref{GenericityofMorse–Smaleproperty} proves that the dissipative
class $\mathcal L_d^-$ contains a dense $G_\delta$ subset of Morse--Smale
vector fields in its relative compact-open $C^1$ topology. The proof uses
complete metrizability of each fixed trapping class; it does not assert
that their union $\mathcal L_d^-$ is a Baire space.
The structural stability theorems for Morse--Smale vector fields on
compact manifolds motivate this result. They do not, by themselves,
establish structural stability on the whole noncompact domain $\Omega$
in the compact-open $C^1$ topology, and no such global conclusion is
claimed here.

	2) {\it Main ingredients of the proofs.}
The proof of automatic transversality uses the spectral decomposition
in Theorem~\ref{Floquet Theory} and two families of nested invariant
cones. Together, they identify complementary subspaces contained in the
tangent spaces to the stable and unstable manifolds.

For genericity, we use the compact-open $C^1$ topology.
Proposition~\ref{Baire} shows that $\mathcal L^-$ is completely
metrizable and hence a Baire space. This choice of topology is
essential: when $n\ge4$, the strong Whitney topology inherited from
$C^1(\Omega,\mathbb R^n)$ is discrete on $\mathcal C^1_{\mathrm{BF}}$;
see Remark~\ref{rk:weak-strong-topologies}.
Generic hyperbolicity of equilibria follows from constant and diagonal
perturbations. For periodic orbits, we first approximate the vector field by a
smooth one whose signed neighboring derivatives are uniformly
positive. The injectivity of projections onto adjacent coordinates
then allows us to localize perturbations while preserving the
feedback structure. Finite-dimensional
parametric transversality gives simple periodic orbits, which are
hyperbolic by Corollary~\ref{simple}. The higher regularity needed for this step is supplied by the
smooth perturbation families; the ambient class remains $C^1$.

To exclude connections between equilibria of the same odd Morse index,
we discretize the flow, following an approach inspired by
\cite{CHM,RGR}. This turns the connection problem into a regular-value
problem on a space of bounded sequences. The surjectivity condition reduces to an integral
pairing with bounded solutions of the adjoint equation
(Lemma~\ref{lem:integral-surjectivity}). The injectivity result in
Proposition~\ref{one to one for connect} allows us to construct the
required perturbations; see Section~\ref{KSD}.

Under (D), we prove genericity directly in the relative topology of
$\mathcal L_d^-$, using the completely metrizable fixed trapping classes
$\mathcal D_{a,R}^-$. We combine a planar classification of limit sets with an
integer-valued rank that decreases along nonconstant connections.
The strict decrease excludes cycles and bounds the length of
connection chains. Together, these properties yield convergence to
individual critical elements, identify the nonwandering set, and
show that there are only finitely many critical elements.

	3) {\it Integer-valued Lyapunov functions.} Such functions are useful in the analysis of tridiagonal competitive--cooperative systems, Cauchy--Riemann equations on $S^1$, and scalar parabolic equations on bounded intervals; see \cite{H.MATANO:1982,Matano,smillie1984,Tl,JBS}. We use the function developed by Gao and Zhou \cite{MD}.
	
	4) {\it Remark on the Poincar\'{e}--Bendixson property.}
System~\eqref{negative feedback system} is strongly $2$-cooperative, so
\cite[Theorem~12]{EM} gives the equilibrium-free Poincar\'{e}--Bendixson
property stated in Proposition~\ref{PB}. The nested invariant cones also allow us to apply the planar
embedding theorem used in Section~\ref{limit set finite}. Under the stronger dissipativity
condition (SD), Theorem~\ref{rob} shows that this equilibrium-free
Poincar\'{e}--Bendixson property persists under sufficiently small
multiples of bounded $C^1$ perturbations, even when the perturbations
do not preserve the feedback structure. This is a robustness statement
about the limit-set alternative; it does not assert topological
equivalence of individual limit sets before and after perturbation.

	5) {\it Remark on the application of our results.}
The class defined by \eqref{Three independent variables} and
\eqref{generalized-f} includes the unidirectional cyclic negative
feedback systems studied in \cite{JMP}. This subclass is obtained by
setting $b_i=0$ for all $i$ in \eqref{bid-feedback} and requiring
the remaining cycle to satisfy the negative feedback sign condition.
Our results for individual systems, including automatic transversality
and the Poincar\'{e}--Bendixson property, apply to this subclass
whenever their hypotheses hold. Analogous genericity results can be obtained in the relative compact-open \(C^1\) topology by adapting the perturbations in Sections \ref{GTC}–\ref{Generic Morse--Smale property111} to preserve the unidirectional structure.

	\vskip 3mm
	The paper is organized as follows. Section~\ref{Integer-valued Lyapunov function} recalls the integer-valued Lyapunov function. Section~\ref{Spectral theory, invariant spaces and nested invariant cones} develops the spectral decomposition and the nested invariant cones. Section~\ref{sec:poincare-bendixson-robustness} treats the Poincar\'{e}--Bendixson property and its robustness, and Section~\ref{Automatic transversality111} proves automatic transversality. Sections~\ref{GTC} and~\ref{GKS} establish generic hyperbolicity and the generic Kupka--Smale property, respectively. Section~\ref{Generic Morse--Smale property111} proves the generic Morse--Smale property in the dissipative class. The appendices collect the transversality and Fredholm results,
together with the uniform local estimates used in Section~\ref{GKS}.
	
	\section{Integer-valued Lyapunov function}\label{Integer-valued Lyapunov function}
	We recall the integer-valued Lyapunov function introduced in
\cite{MD}, together with the properties needed in the sequel.
	
	For $x=(x_1,\ldots,x_n)\in\mathbb R^n$ with
$x_j\ne0$ for every $j\in\{1,\ldots,n\}$, define
\[
N(x)=\operatorname{card}
\{i\in\{1,\ldots,n\}:\delta_i x_i x_{i-1}<0\},
\]
where $\delta_1=-1$, $\delta_i=1$ for $2\le i\le n$, and $x_0=x_n$.
Its possible values are $1,3,\ldots,\widetilde n$, where
$\widetilde n=n$ if $n$ is odd and $\widetilde n=n-1$ if $n$ is even.
For any $x\in\mathbb R^n$, set
\[
\begin{aligned}
N_m(x)&=\min\{N(y):y\in U_x,\ y_j\ne0\text{ for all }j=1,\ldots,n\},\\
N_M(x)&=\max\{N(y):y\in U_x,\ y_j\ne0\text{ for all }j=1,\ldots,n\},
\end{aligned}
\]
where $U_x$ is a sufficiently small neighborhood of $x$.
The function $N$ extends to the open dense set
\[
\mathcal N:=\{x\in\mathbb R^n:N_m(x)=N_M(x)\}
\]
by setting $N(x)=N_m(x)=N_M(x)$ there. The following properties
describe its behavior along nontrivial solutions of the linear equation
\begin{equation}\label{A-E}
\dot x=A(t)x,\qquad A(t)\in\mathcal M^-.
\end{equation}
In particular, the variational equation
\begin{equation}\label{B}
\dot y=Df(x(t))y,\qquad x(0)\in\Omega,
\end{equation}
along a solution $x(t)$ of~\eqref{negative feedback system} has this form.

	\begin{lem}\label{pro of N of L-S}
		Let $A:(\alpha,\beta)\to\mathcal M^-$ be continuous, and let
$x(t)$ be a nontrivial solution of \eqref{A-E}. Then the following
properties hold.
		\begin{enumerate}
			\item[{\rm(i)}] $x(t)\in\mathcal{N}$, except at finitely many times in $(\alpha, \beta)$;
			\item[{\rm(ii)}] $N(x(t))$ is locally constant wherever $x(t)\in\mathcal{N}$;
			\item[{\rm(iii)}] if $x(t_0)\notin\mathcal N$, then there exists
$\varepsilon>0$ with $[t_0-\varepsilon,t_0+\varepsilon]\subset(\alpha,\beta)$
such that, for every $0<s<\varepsilon$,
$x(t_0-s),x(t_0+s)\in\mathcal N$ and
$N(x(t_0+s))<N(x(t_0-s))$;
			\item[{\rm(iv)}] if $x(t)\in\mathcal{N}$, then $(x_i(t),x_ {i+1}(t))\neq(0,0)$, for any $i\in\{1,2,\cdots,n\}$;
			\item[{\rm(v)}] if $(\alpha,\beta)=\mathbb R$, then there exist
$T>0$ and $N_-,N_+\in\{1,3,\ldots,\widetilde n\}$, with $N_-\ge N_+$,
such that $x(t)\in\mathcal N$ for $|t|\ge T$ and
\[
N(x(t))=N_-\quad(t\le -T),\qquad
N(x(t))=N_+\quad(t\ge T).
\]
		\end{enumerate}
	\end{lem}
	
	\begin{proof}
		The local assertion follows from the implication
(i)$\Rightarrow$(ii) in \cite[Proposition~3.3]{MD}, applied on
$(\alpha,\beta)$; see also \cite[Theorem~3.1]{MD}. The strict
decrease and the finite range of $N$ give (i) and (v).
	\end{proof}
\begin{lem}\label{linear system large-constant}
Let $I=(\alpha,\beta)$, let
$A\in C(I,\mathbb R^{n\times n})$ satisfy
$A(t)\in\mathcal M^-$ for every $t\in I$, and let $x$ be a
nontrivial solution of
\[
\dot x(t)=A(t)x(t),\qquad t\in I.
\]
If $t_0\in I$ and $x(t_0)\notin\mathcal N$, then there exists
$\varepsilon>0$, with
$[t_0-\varepsilon,t_0+\varepsilon]\subset I$, such that, for every
$0<s<\varepsilon$,
\[
x(t_0-s),\,x(t_0+s)\in\mathcal N
\]
and
\[
N(x(t_0+s))
=
N_m(x(t_0))
<
N_M(x(t_0))
=
N(x(t_0-s)).
\]
\end{lem}
	
	\begin{proof}
		See \cite[Proposition 3.3]{MD}.
	\end{proof}
	
	\begin{cor}\label{subtract}
Let $x^1,x^2:I\to\Omega$ be distinct solutions of
\eqref{negative feedback system} on an open interval $I$, and put
$y(t)=x^1(t)-x^2(t)$. If $t_0\in I$ and $y(t_0)\notin\mathcal N$,
then there exists $\varepsilon>0$ with
$[t_0-\varepsilon,t_0+\varepsilon]\subset I$ such that, for every
$0<s<\varepsilon$,
\[
y(t_0-s),y(t_0+s)\in\mathcal N
\]
and
\[
N(y(t_0+s))=N_m(y(t_0))<N_M(y(t_0))=N(y(t_0-s)).
\]
\end{cor}

\begin{proof}
Uniqueness implies $y(t)\ne0$ for all $t\in I$. By convexity of
$\Omega$ and the mean-value formula,
\[
\dot y(t)=B(t)y(t),\qquad
B(t)=\int_0^1 Df\bigl(x^2(t)+\theta(x^1(t)-x^2(t))\bigr)\,d\theta.
\]
The matrix function $B$ is continuous, has the required cyclic
tridiagonal form, and has nonnegative signed neighboring entries.
For each fixed $t$, choose a point on the segment. At that point,
all entries in one of the two signed cycles are strictly positive.
By continuity in $\theta$, their integrals are strictly positive. Thus
$B(t)\in\mathcal M^-$ for every $t\in I$.
Lemma~\ref{linear system large-constant} now applies to $y$.
\end{proof}

	\begin{cor}\label{periodic solution of system}
		Let $p(t)=(p_1(t),\ldots,p_n(t))$ be a nonconstant periodic solution of \eqref{negative feedback system} with minimal period $\omega$. For each $s\in\{1,\ldots,n\}$, with $p_{n+1}=p_1$, the map
		$$t\in \left(0,\omega\right]\mapsto  \left(p_s\left(t\right),p_{s+1}\left(t\right)\right)$$
		is injective.
	\end{cor}
	
	\begin{proof}
		Suppose, to the contrary, that there are distinct $t_1,t_2\in(0,\omega]$ such that $$\left(p_s\left(t_1\right)-p_s\left(t_2\right),p_{s+1}\left(t_1\right)-p_{s+1}\left(t_2\right)\right)= 0.$$ 
		Let $y\left(t\right)=p\left(t+t_1-t_2\right)-p\left(t\right)$. The minimality of $\omega$ implies that $y$ is nontrivial.
By Lemma~\ref{pro of N of L-S}(iv), $y(t_2)\notin\mathcal N$.
Corollary~\ref{subtract} gives a strict drop at $t=t_2$.
Since $N$ is nonincreasing along $y$, this contradicts the
$\omega$-periodicity of $y$.

	\end{proof}

	\section{Spectral theory, invariant spaces and nested invariant cones}\label{Spectral theory, invariant spaces and nested invariant cones}
Let \(S_A(t,s)\) denote the solution operator of (\ref{A-E}), so that \(S_A(t,s)x_0\) is the value at time \(t\) of the solution satisfying \(x(s)=x_0\). Throughout the paper, the terminal time is written first; \(t>s\) corresponds to forward evolution and \(t<s\) to backward evolution. For each fixed pair \(s<t\), the following theorem describes the spectral decomposition of \(S_A(t,s)\).

	\begin{thm}\label{Floquet Theory}
		Let $A:(\alpha,\beta)\to\mathcal M^-$ be continuous. For each
fixed pair $s,t\in(\alpha,\beta)$ with $s<t$, there are
$S_A(t,s)$-invariant subspaces $W_i$, $1\le i\le(\widetilde n+1)/2$,
such that
		\begin{align*}
			\operatorname{dim} W_i & =2, \quad i=1, \ldots,\frac{\widetilde n-1 }{2}  \\
			\operatorname{dim} W_{\frac{\widetilde n+1 }{2}} & = \begin{cases} 1 & \text { if } n=\widetilde n \text { is odd } \\
				2 & \text { if } n=\widetilde n+1 \text { is even }
			\end{cases}
		\end{align*}
		and
		$$\mathbb{R}^{n} =W_1 \oplus W_2 \oplus \cdots \oplus W_{\frac{\widetilde n+1 }{2}}.$$
		If $x\in W_i \setminus  \left\{0\right\}$, then $x\in \mathcal{N}$ and $N\left(x\right)=2i-1$, for $i=1, \ldots,\frac{\widetilde n+1 }{2}$. If $x \in W_i 
		\oplus W_{i+1}\oplus\cdots\oplus W_k$ and $x\ne 0$, then $N_m(x) \geqslant 2 i-1$ and $N_M(x) \leqslant 2 k-1$. Moreover, if $\nu _{i}$ and $\mu _{i} $ are 
		the 
		minimum and maximum moduli of the eigenvalues of the restriction of $S_{A} \left (t,s \right )$ to $W_i$, then
		\begin{equation*}
			\mu_1 \geq \nu_1>\mu_2 \geq \nu_2>\cdots>\mu_{\frac{\widetilde n+1}{2}}\geq \nu_{\frac{\widetilde n+1}{2}}.
		\end{equation*}
	\end{thm}
	
	The proof uses cones of finite rank. We recall the definition
from \cite{FOA,KL,S}.
	
	\begin{defn}
		{\rm 
			Let $k\in \mathbb{N} \setminus \left \{ 0 \right \}$. A closed subset $K\subset \mathbb{R}^{n} $ is called a
			cone of rank $k$ if for any $x\in K$ and $\lambda \in \mathbb{R}$, one has $\lambda x\in K$; moreover, $\max \left\{\operatorname{dim} V \mid 
			V\right.$ is a subspace of $\mathbb{R}^n$ and $\left.V \subset K\right\}=k$.		
		}
	\end{defn}
	
	For $h\in\{1,\ldots,(\widetilde n+1)/2\}$, define
	\begin{itemize}
		\item $\mathcal{K}_h=\{0\} \cup\left\{x \in \mathbb{R}^n \mid N_M(x) \leqslant 2 h-1\right\}$,
		\item $\mathcal{K}^h=\{0\} \cup\left\{x \in \mathbb{R}^n \mid N_m(x)>2 h-1\right\}$.
	\end{itemize}
	In particular, we set $\mathcal{K}_0=\{0\}$ and $\mathcal{K}^0=\mathbb{R}^n$. The sets $\mathcal{K}_h\setminus\{0\}$ and 
	$\mathcal{K}^h\setminus\{0\}$ are open, and $\mathcal{K}_h\cap\mathcal{K}^h=\{0\}$. We denote by $\overline{\mathcal K_h}$ (resp. 
	$\overline{\mathcal K^h}$) the closure of $\mathcal{K}_h$ (resp. $\mathcal{K}^h$), and by $\operatorname{Int}\overline{\mathcal K_h}$ the interior of $\overline{\mathcal K_h}$. 
	These definitions give $\overline{\mathcal K_h}=\overline{\left(\mathcal{K}_h \backslash\{0\}\right)}$ and $\operatorname{Int}\overline{\mathcal K_{h}}=\mathcal{K}_h 
	\backslash\{0\}$ {for $1\le h<(\widetilde n+1)/2$}. 
	
	The next proposition follows from Lemma~\ref{pro of N of L-S}.
	\begin{prop}\label{solution operator prop}
		Let $S_A\left(t,s\right)$ be the solution operator of \eqref{A-E}. Then  
		$$S_A\left ( t,s\right ) \left ( \overline{\mathcal K_h}\setminus \left \{ 0 \right \} \right )\subset \operatorname{Int} \overline{\mathcal K_h} \quad 
		\left(S_A\left ( t,s \right ) \left ( \overline{\mathcal K^h}\setminus \left \{ 0 \right \} \right )\subset \operatorname{Int} 
		\overline{\mathcal K^h}\right)$$
		for $t>s$ (respectively, $t<s$).
	\end{prop}
	
	The next proposition identifies the ranks of these cones.
	
	\begin{prop}\label{cone}
		For $h\in\{1,\ldots,(\widetilde n-1)/2\}$, the maximal dimension of a linear subspace contained in $\overline{\mathcal K_h}$ is $2h$; that is,
		$$
		d_h=\max \left\{\operatorname{dim} V \mid V \subset \overline{\mathcal K_h}\right\}=2h.
		$$
	\end{prop}
	
	\begin{proof}
		The argument of \cite[Proposition~2.7]{YD} applies to the cones
defined above.
	\end{proof}
	
	We also use the following generalized Perron--Frobenius theorem
\cite[Theorem~1]{FOA}.

	\begin{lem}\label{Perron–Frobenius}
		Let $K\subset \mathbb{R}^{n} $  be a solid cone (i.e., $\operatorname{Int}K\ne \emptyset$) of rank $d$, and let $L$ be a linear operator on $\mathbb{R}^{n}$ satisfying $L(K \setminus \{0\}) \subset \operatorname{Int} K$. Then there exist unique subspaces $V_1$, $V_2$ such that
		\begin{enumerate}
			\item[{\rm(i)}] $\mathbb{R}^n =V_1 \oplus  V_2,\quad \operatorname{dim} V_1=d$;
			\item[{\rm(ii)}] $L V_j \subset V_j, j=1,2$;
			\item[{\rm(iii)}] $V_1 \subset\{0\} \cup \operatorname{Int} K, V_2 \cap K=\{0\}$.
		\end{enumerate}
		Moreover, if $\sigma _1\left ( L \right ) $ and $\sigma _2\left ( L \right ) $ are the spectra of $L$ restricted to $V_1$ and $V_2$, then 
		$$
		\lambda \in \sigma_1(L), \mu \in \sigma_2(L) \Rightarrow|\lambda|>|\mu|.
		$$
	\end{lem}
	
\begin{proof}[Proof of Theorem \ref{Floquet Theory}]
Fix $s<t$ in $(\alpha,\beta)$, and set $L=S_A(t,s)$ and
$m=(\widetilde n+1)/2$. By Propositions~\ref{solution operator prop}
and~\ref{cone}, Lemma~\ref{Perron–Frobenius} applies to
$L$ and $\overline{\mathcal K_h}$ for $1\le h<m$.
It yields invariant decompositions $\mathbb R^n=E_h\oplus F_h$ with
\[
\dim E_h=2h,\qquad
E_h\setminus\{0\}\subset\mathcal K_h,\qquad
F_h\cap\overline{\mathcal K_h}=\{0\},
\]
and every eigenvalue on $E_h$ has modulus strictly larger than
every eigenvalue on $F_h$. Thus $E_h$ is the real spectral subspace
corresponding to the $2h$ largest eigenvalue moduli, counted with
algebraic multiplicity, and $F_h$ is its complementary spectral
subspace. Consequently,
$E_h\subset E_{h+1}$ and $F_{h+1}\subset F_h$.
Set $E_0=F_m=\{0\}$, $E_m=F_0=\mathbb R^n$, and
\[
W_i=E_i\cap F_{i-1},\qquad 1\le i\le m.
\]
The spectral decompositions give
$E_h=W_1\oplus\cdots\oplus W_h$ and
$F_h=W_{h+1}\oplus\cdots\oplus W_m$.
They also give the asserted dimensions and the strict spectral
ordering; compare \cite[Theorem~2]{FOT}.

By the definitions of $N_m$ and $N_M$, and the density of vectors
with no zero coordinates,
\[
\overline{\mathcal K_h}
=\{0\}\cup\{v\ne0:N_m(v)\le 2h-1\}.
\]
Hence, for $0\ne v\in W_i\oplus\cdots\oplus W_k
\subset F_{i-1}\cap E_k$, we have
$N_m(v)\ge2i-1$ and $N_M(v)\le2k-1$.
At $i=1$ or $k=m$, the respective bound follows from the range
$1\le N_m(v)\le N_M(v)\le\widetilde n$.
Taking $i=k$ gives $N_m(v)=N_M(v)=2i-1$, so that
$v\in\mathcal N$ and $N(v)=2i-1$, as required.
\end{proof}

Figure~\ref{fig:spectral-cones} summarizes the spectral ordering and
the complementary subspaces obtained in the proof. The separation of
these subspaces by the cones will also be used in the transversality
arguments below.

\begin{figure}[htbp]
\centering
\begin{tikzpicture}[x=1cm,y=1cm,font=\small]
  \draw[-{Stealth},line width=.7pt] (.25,1.15)--(13.75,1.15)
    node[midway,above=3pt] {decreasing eigenvalue modulus of the fixed operator $L=S_A(t,s)$};
  \filldraw[fill=diagramblue!6,draw=diagramblue!50,rounded corners=3pt]
    (0,-1.18) rectangle (6.65,.65);
  \filldraw[fill=diagramred!6,draw=diagramred!50,rounded corners=3pt]
    (7.35,-1.18) rectangle (14,.65);
  \foreach \x/\w/\r in {1.05/1/1,5.6/h/{2h-1}} {
    \node[draw=diagramblue,fill=white,minimum width=1.6cm,minimum height=.65cm] at (\x,.05) {$W_{\w}$};
    \node at (\x,-.72) {$N=\r$};
  }
  \node at (3.32,.05) {$\cdots$};
  \node[draw=diagramred,fill=white,minimum width=1.6cm,minimum height=.65cm] at (8.4,.05) {$W_{h+1}$};
  \node[draw=diagramred,fill=white,minimum width=1.6cm,minimum height=.65cm] at (12.95,.05) {$W_m$};
  \node at (10.68,.05) {$\cdots$};
  \node at (8.4,-.72) {$N=2h+1$};
  \node at (12.95,-.72) {$N=2m-1$};
  \draw[densely dashed,black!55] (7,.65)--(7,-1.18);
  \node[text=diagramblue] at (3.32,-1.65) {$E_h=W_1\oplus\cdots\oplus W_h$};
  \node[text=diagramred] at (10.68,-1.65) {$F_h=W_{h+1}\oplus\cdots\oplus W_m$};
  \node at (3.32,-2.2) {$\dim E_h=2h$};
  \node at (10.68,-2.2) {$\dim F_h=n-2h$};
  \node at (3.32,-2.85) {$E_h\setminus\{0\}\subset\operatorname{Int}\overline{\mathcal K_h}$};
  \node at (10.68,-2.85) {$F_h\cap\overline{\mathcal K_h}=\{0\}$};
\end{tikzpicture}
\caption{Spectral blocks and cone separation for $m=(\widetilde n+1)/2$
and $1\le h<m$. The labels $N=2i-1$ apply to nonzero vectors in
$W_i$. Successive blocks are separated in modulus:
$\mu_i\ge\nu_i>\mu_{i+1}$. The horizontal arrangement records
this ordering, allowing complex eigenvalues and repeated moduli
within each block. Each block is two-dimensional,
except that $W_m$ is one-dimensional when $n$ is odd. At an end cut,
the corresponding group may consist of a single block.
The shading groups the blocks in the decomposition
$\mathbb R^n=E_h\oplus F_h$.}
\label{fig:spectral-cones}
\end{figure}
\FloatBarrier

	\begin{defn}\label{simple critical elements}
		{\rm 
			Let $\gamma$ be a periodic orbit of the system  $\dot{x} = f\left ( x \right 
			) $ with period $\omega$. If $1$ is a simple eigenvalue of $S_{Df \left(\gamma\right)}\left(\omega,0\right)$, then $\gamma$  
			is called a simple periodic orbit. 
		}
	\end{defn}

	The spectral decomposition also restricts the Floquet multipliers on the unit circle.
	
	\begin{cor}\label{simple}
		Let $\gamma$ be a periodic orbit of (\ref{negative feedback system}) with  period $\omega$. Then $1$ is the only eigenvalue of $S_{D{f}\left ( 
			\gamma  \right ) } \left ( \omega,0  \right ) $ on the unit circle. Consequently, a periodic orbit is simple if and only if it is hyperbolic.
	\end{cor}
   \begin{proof}
{Let $M=S_{Df(\gamma)}(\omega,0)$.
Since $\dot\gamma(t)$ is a nontrivial solution of the variational
equation along $\gamma$ and is $\omega$-periodic, we have
$M\dot\gamma(0)=\dot\gamma(0)$. Hence, $1$ is an eigenvalue of $M$.
By Theorem~\ref{Floquet Theory}, there exists a unique $h$ such that
$1\in\sigma(M|_{W_h})$, and every eigenvalue of $M$ with modulus one
must belong to $\sigma(M|_{W_h})$.

If $\dim W_h=1$, the conclusion is immediate. 

Suppose that $\dim W_h=2$. We first show that
$\det(M|_{W_h})>0$. Choose a basis $\{\xi_1,\xi_2\}$ of $W_h$ and set
$\xi_i(t)=S_{Df(\gamma)}(t,0)\xi_i$. For every
$\xi\in W_h\setminus\{0\}$, Theorem~\ref{Floquet Theory} gives
$N(\xi)=N(M\xi)=2h-1$. Since $N$ is nonincreasing along nontrivial
solutions of the variational equation and drops strictly whenever
the solution leaves $\mathcal N$, Lemma~\ref{pro of N of L-S} implies that
$S_{Df(\gamma)}(t,0)\xi\in\mathcal N$ and
$N(S_{Df(\gamma)}(t,0)\xi)=2h-1$ for all $t\in[0,\omega]$.
Fix $s\in\{1,\ldots,n\}$, with $n+1$ identified with $1$, and define
\[
D(t)=
\det
\begin{pmatrix}
(\xi_1(t))_s & (\xi_2(t))_s\\
(\xi_1(t))_{s+1} & (\xi_2(t))_{s+1}
\end{pmatrix}.
\]
Then $D(t)\neq0$ for every $t\in[0,\omega]$. Indeed, if
$D(t_0)=0$, there would exist $(a,b)\neq(0,0)$ such that the nonzero
vector $a\xi_1(t_0)+b\xi_2(t_0)$ has its $s$-th and $(s+1)$-st
components both equal to zero, contradicting Lemma~2.1(iv).
Thus $D$ has a constant sign on $[0,\omega]$. Let $C$ be the matrix
of $M|_{W_h}$ with respect to the basis
$\{\xi_1,\xi_2\}$. Since
$[\xi_1(\omega)\ \xi_2(\omega)]=[\xi_1\ \xi_2]C$, projecting onto
the $s$-th and $(s+1)$-st coordinates gives
$D(\omega)=D(0)\det C$. Hence
$\det C=D(\omega)/D(0)>0$, and therefore
$\det(M|_{W_h})>0$.

Let $\lambda$ denote the other eigenvalue of the real
two-dimensional operator $M|_{W_h}$, counted with algebraic
multiplicity. Then
$1\cdot\lambda=\det(M|_{W_h})>0$, so $\lambda>0$. Consequently,
if $|\lambda|=1$, then necessarily $\lambda=1$. Thus $1$ is the
unique eigenvalue of $M$ with modulus one.

Therefore, if $\gamma$ is simple, then $1$ is simple and is the
unique eigenvalue of $M$ on the unit circle, so $\gamma$ is
hyperbolic. The converse follows directly from Definition~\ref{HCE}.}
\end{proof}

	The next consequence of Theorem~\ref{Floquet Theory} is used to construct perturbations in Section~\ref{KSD}.
	
	\begin{prop}\label{one to one for connect}
		Let $e^\pm$ be hyperbolic equilibria of \eqref{negative feedback system} with $i(e^+)=i(e^-)=2h-1$, where $h\in\{1,\ldots,(\widetilde n+1)/2\}$. Let $y(t)=(y_1(t),\ldots,y_n(t))$ be a nonconstant solution connecting $e^-$ to $e^+$. The equilibria $e^-$ and $e^+$ are allowed to coincide. For each $s\in\{1,\ldots,n\}$, with $y_{n+1}=y_1$, the map
		$$t\in \left(-\infty,\infty\right)\mapsto  \left(y_s\left(t\right),y_{s+1}\left(t\right)\right)$$
		is injective.
	\end{prop}
	
	\begin{proof}
{Suppose, for contradiction, that there exist $t_1,t_2\in\mathbb R$,
$t_1\neq t_2$, such that
\[
 (y_s(t_1)-y_s(t_2),\,y_{s+1}(t_1)-y_{s+1}(t_2))=(0,0).
\]
Set $\delta=t_1-t_2$ and $z(t)=y(t+\delta)-y(t)$. The two
solutions $y(t+\delta)$ and $y(t)$ are distinct; otherwise $y$ would
be periodic, which is incompatible with its convergence to equilibria
as $t\to\pm\infty$, unless $y$ is constant. Thus, $z(t)\neq0$ for all
$t$. Since $z_s(t_2)=z_{s+1}(t_2)=0$, Lemma~\ref{pro of N of L-S}(iv) implies
$z(t_2)\notin\mathcal N$. Hence, by Corollary~\ref{subtract}, $N(z(t))$
strictly decreases at $t=t_2$.

Let $E_-^u=T_{e^-}W^u(e^-)$ and $E_+^s=T_{e^+}W^s(e^+)$. Applying
Theorem~\ref{Floquet Theory} to $e^{Df(e^-)}$ and $e^{Df(e^+)}$, respectively, and
using $i(e^-)=i(e^+)=2h-1$, we obtain
\begin{equation}\label{EUS}
    E_-^u\subset W_1^-\oplus\cdots\oplus W_h^-\subset\mathcal K_h,
 \qquad
 E_+^s\subset W_h^+\oplus\cdots\oplus
 W_{\frac{\widetilde n+1}{2}}^+\subset\mathcal K^{h-1}.
\end{equation}
Indeed, if the first inclusion failed, then some $W_j^-$ with
$j\ge h+1$ would contain a nonzero unstable spectral direction. The
spectral ordering in Theorem~\ref{Floquet Theory} would imply $\nu_h^->1$, and hence
$W_1^-\oplus\cdots\oplus W_h^-\subset E_-^u$. Since $j\ge h+1$,
the first $h$ spaces are two-dimensional, which gives
$\dim E_-^u\ge2h$, contradicting $\dim E_-^u=2h-1$. The second
inclusion follows in the same way. The remaining inclusions in
\eqref{EUS} follow directly from Theorem~\ref{Floquet Theory}.

To obtain the corresponding secant estimates for the local
manifolds, write, near $e^-$,
\[
 W_{\rm loc}^u(e^-)
 =\{e^-+u+g^-(u):u\in E_-^u,\ \|u\|<r\},
\]
where $g^-(0)=0$ and $Dg^-(0)=0$. Since
$E_-^u\cap S^{n-1}$ is compact and contained in the open set
$\mathcal K_h\setminus\{0\}$, all directions sufficiently close to
$E_-^u\cap S^{n-1}$ remain in $\mathcal K_h$. Moreover, after
decreasing $r$ if necessary, the continuity of $Dg^-$ gives
$\|g^-(u)-g^-(v)\|\le\varepsilon\|u-v\|$, with $\varepsilon>0$
arbitrarily small. Hence, for $u\neq v$, the direction of
$(u-v)+(g^-(u)-g^-(v))$ is uniformly close to that of
$u-v\in E_-^u$. Therefore, the difference of any two distinct points
of $W_{\rm loc}^u(e^-)$ sufficiently close to $e^-$ belongs to
$\mathcal K_h$.

Applying the same argument to $W_{\rm loc}^s(e^+)$ and
$\mathcal K^{h-1}$, and using $y(t)\to e^-$ as $t\to-\infty$ and
$y(t)\to e^+$ as $t\to+\infty$, we obtain, for all sufficiently large
$T$, $ z(-T)\in\mathcal K_h$ and $z(T)\in\mathcal K^{h-1}$.
Moreover, $z$ is a nontrivial solution of an equation of the form
\eqref{A-E}, since $f(y(t+\delta))-f(y(t))$ has the integral representation
along the segment joining $y(t)$ and $y(t+\delta)$. Hence
Lemma~\ref{pro of N of L-S}(v) gives $z(\pm T)\in\mathcal N$ for all sufficiently large
$T$. By the definitions of $\mathcal K_h$ and $\mathcal K^{h-1}$,
\begin{equation}\label{NT}
     N(z(-T))\le2h-1\le N(z(T)).
\end{equation}

Finally, choose $T$ sufficiently large that $-T<t_2<T$. Since
$N(z(t))$ is nonincreasing and strictly decreases at $t=t_2$, one has
$N(z(-T))>N(z(T))$, contradicting (\ref{NT}). Therefore,
$t\mapsto(y_s(t),y_{s+1}(t))$ is injective.}
\end{proof}

	\section{Poincar\'e--Bendixson property and robustness}\label{sec:poincare-bendixson-robustness}

{We first recall the Poincar\'e--Bendixson property for
system~\eqref{negative feedback system}. Under a stronger
dissipativity condition, we then prove that this property persists
for $f+\varepsilon g$ when $g$ is a bounded $C^1$ vector field
and $|\varepsilon|$ is sufficiently small.

\begin{prop}[Poincar\'e--Bendixson property]\label{PB}
Let $f\in\mathcal L^-$, and let $u(t)=S_f(t)u_0$ be a forward solution
of~\eqref{negative feedback system} defined for all $t\geq0$. Assume that $\omega(u_0)$ is a
nonempty compact subset of $\Omega$. If $\omega(u_0)$ contains no
equilibrium, then it is a nonconstant periodic orbit.
\end{prop}

\begin{proof}
System~\eqref{negative feedback system} is strongly $2$-cooperative by
\cite[Corollary~4]{EM}. The conclusion follows directly from
\cite[Theorem~12]{EM}.
\end{proof}

\begin{rk}
{\rm If the closure of the positive orbit is a compact subset of $\Omega$,
then the compactness assumption in Proposition~\ref{PB} is automatically
satisfied. In particular, when $\Omega=\mathbb R^n$, boundedness of the
positive orbit is sufficient. For a general open set $\Omega$, however,
boundedness alone does not exclude accumulation at $\partial\Omega$.}
\end{rk}

The following lemma verifies the hypotheses of the abstract
perturbation theorem. We will also use it in Section~8.

\begin{lem}\label{Terescak hypotheses}
Let $f\in\mathcal L^-$, and let $U\Subset\Omega$ be a bounded open
convex set that is positively invariant under $S_f(t)$. Suppose that
the semiflow on $U$ is forward complete. For every compact invariant
set $K_0\subset U$, the hypotheses of
\cite[Corollaries~2.3--2.4]{Tl} hold with
\[
X=\mathbb R^n,\qquad
\Pi=\{v\in\mathbb R^n:v_1=v_2=0\}.
\]
\end{lem}

\begin{proof}
Write $S(t)=S_f(t)$ and $m=(\widetilde n+1)/2$. Set
$C_0=\{0\}$, $C_j=\overline{\mathcal K_j}$ for $1\le j<m$,
and $C_m=\mathbb R^n$. These cones are nested, and $C_j$ has rank
$2j$ for $j<m$. The definition of a $k$-cone in \cite[p.~7]{Tl}
also requires a complementary subspace of codimension $k$ that
meets the cone only at zero. To verify this condition, apply
Theorem~\ref{Floquet Theory} to $e^{A_*}$ for any fixed
$A_*\in\mathcal M^-$. Its spectral subspaces give
\[
E_j=W_1\oplus\cdots\oplus W_j\subset C_j,\qquad
F_j=W_{j+1}\oplus\cdots\oplus W_m,
\]
with $\dim E_j=\operatorname{codim}F_j=2j$ and
$F_j\cap C_j=\{0\}$.

For $x,y\in U$, define
\[
B_{x,y}(s)=\int_0^1
Df\bigl(S(s)y+\theta(S(s)x-S(s)y)\bigr)\,d\theta,
\qquad s\ge0.
\]
Convexity and positive invariance of $U$ make this expression
well defined. The signed off-diagonal entries are nonnegative.
At any point of the segment, all signed entries in at least one
of the two cyclic orientations are strictly positive. By continuity,
those entries remain positive on a subinterval of positive length.
Their averages are therefore positive, so
$B_{x,y}(s)\in\mathcal M^-$.
Let $T^t_{x,y}$ be the solution operator from time $0$ to time $t$
of $\dot v=B_{x,y}(s)v$. Then
\begin{equation}\label{two point identity}
T^t_{x,y}(x-y)=S(t)x-S(t)y.
\end{equation}
The operators $T^t_{x,y}$ are invertible and depend continuously on
$(t,x,y)$. They are compact because the space is finite-dimensional.
Proposition~\ref{solution operator prop} gives the strong cone
invariance required in \cite[Corollary~2.3(ii)]{Tl}; condition (iii)
there follows from invertibility and $C_m=\mathbb R^n$.

For the projection condition, suppose that $v\in C_j$ and
$0\ne w=T^t_{x,y}v\in C_j\setminus C_{j-1}$, where $t>0$.
Strong cone invariance and the description of $C_{j-1}$ imply
\[
N_M(w)\le2j-1,\qquad N_m(w)>2j-3.
\]
For $j=1$ or $j=m$, the corresponding endpoint inequality follows
from $1\le N_m(w)\le N_M(w)\le\widetilde n$.
Since these quantities are odd integers, $N_m(w)=N_M(w)=2j-1$.
Thus $w\in\mathcal N$, and Lemma~\ref{pro of N of L-S}(iv) gives
$w\notin\Pi$. This is precisely condition (iv) of the cited corollary.

Finally, $f$ is bounded on $\overline U$, so each positive trajectory
in $U$ is uniformly Lipschitz in time. Moreover,
$(t,x)\mapsto\partial_tS(t)x=f(S(t)x)$ is continuous. Hence the
semiflow is continuous and regularizing in the sense of
\cite[p.~9]{Tl}. All the hypotheses of Corollaries~2.3--2.4 in that
reference are satisfied, with $\operatorname{codim}\Pi=2$.
\end{proof}

To establish robustness, assume that $f\in\mathcal L^-$ satisfies
the following strong dissipativity condition: there exist $a\in\Omega$ and
$M,\delta>0$ such that $B(a,M)\Subset\Omega$ and
\begin{equation}
\langle f(x),x-a\rangle\leq-\delta |x-a|,
\qquad x\in\Omega,\quad |x-a|\geq M.
\tag{SD}
\end{equation}
Let $g\in C^1(\Omega,\mathbb R^n)$ be bounded and consider
\begin{equation}\label{robf}
\dot x=f(x)+\varepsilon g(x).
\end{equation}

\begin{thm}\label{rob}
There exists $\varepsilon_0>0$ such that, for every
$|\varepsilon|<\varepsilon_0$, each forward solution of~(\ref{robf}) that is
defined in $\Omega$ for all $t\geq0$ has a precompact positive orbit in
$\Omega$. Moreover, if its $\omega$-limit set contains no equilibrium
of~(\ref{robf}), then the $\omega$-limit set is a nonconstant periodic orbit.
\end{thm}

\begin{proof}
Choose $M_1>M$ such that
$D_1:=\overline{B(a,M_1)}\Subset\Omega$, and put
$D:=\overline{B(a,M)}$. Restrict $\varepsilon$ so that
$|\varepsilon|\|g\|_\infty<\delta/2$. Along a solution with
$|x(t)-a|\ge M$, condition (SD) gives
\[
\frac{d}{dt}|x(t)-a|
=\frac{\langle f(x(t))+\varepsilon g(x(t)),x(t)-a\rangle}{|x(t)-a|}
\le-\frac{\delta}{2}.
\]
Thus $D$ and $D_1$ are positively invariant, and the semiflow on
$B(a,M_1)$ is forward complete. Every forward solution defined in
$\Omega$ for all $t\ge0$ enters $D$ in finite time. Its initial
finite orbit segment and its remaining orbit in $D$ together have
compact closure in $\Omega$.

Write $S_\varepsilon(t)$ for the semiflow of $f+\varepsilon g$ on
$U:=B(a,M_1)$. The sets $S_0(t)D$ are nonempty, compact, and nested.
Consequently,
\[
\mathcal A:=\bigcap_{t\ge0}S_0(t)D
\]
is a nonempty compact invariant set. It attracts $D$ uniformly:
for every neighborhood $V$ of $\mathcal A$, one has
$S_0(t)D\subset V$ for all sufficiently large $t$. This follows
directly from nestedness and compactness.

By Lemma~\ref{Terescak hypotheses}, we may apply
\cite[Corollary~2.4]{Tl} with $K_0=\mathcal A$. It provides an
open neighborhood $V\subset U$ of $\mathcal A$ and a number
$\eta>0$ such that every regularizing semiflow $S'(t)$ satisfying
\begin{equation}\label{uniform semiflow perturbation}
\sup_{1/2\le t\le1}
\|S'(t)-S_0(t)\|_{C^1(V)}<\eta
\end{equation}
has the stated Poincar\'{e}--Bendixson property for every positive
orbit whose closure is contained in $V$. Both $V$ and $\eta$ are
chosen for the fixed unperturbed set $\mathcal A$.

On $D_1$, the fields $f+\varepsilon g$ converge to $f$ in $C^1$.
Continuous dependence of solutions and variational equations gives,
for every fixed $T>0$,
\[
\sup_{0\le t\le T}
\|S_\varepsilon(t)-S_0(t)\|_{C^1(D_1)}\longrightarrow0
\qquad(\varepsilon\to0).
\]
The perturbed semiflows are regularizing, since their trajectories
remain in $D_1$ and
\[
\partial_tS_\varepsilon(t)y
=(f+\varepsilon g)(S_\varepsilon(t)y)
\]
depends continuously on $(t,y)$. In particular,
\eqref{uniform semiflow perturbation} holds for sufficiently small
$|\varepsilon|$.

Choose an open set $W$ such that
$\mathcal A\subset W\Subset V$. Uniform attraction gives a time
$T_*>0$ with $S_0(T_*)D\subset W$. By continuous dependence,
after decreasing the perturbation bound if necessary,
$S_\varepsilon(T_*)D\subset W$. Since $D$ is positively invariant
for every such perturbation, for all $t\ge T_*$ we have
\begin{equation}\label{uniform tail localization}
S_\varepsilon(t)D
=S_\varepsilon(T_*)S_\varepsilon(t-T_*)D
\subset S_\varepsilon(T_*)D\subset\overline W\subset V.
\end{equation}
Thus the tails of all positive orbits starting in $D$ lie in the
same compact subset of $V$ after time $T_*$.

Choose $\varepsilon_0>0$ so that all the preceding conditions hold
for $|\varepsilon|<\varepsilon_0$. This choice is independent of
the initial point. Every forward solution under consideration
enters $D$, so \cite[Corollary~2.4]{Tl}, applied after a time shift,
classifies its $\omega$-limit set as a single periodic orbit or a
set of equilibria and their connecting orbits. If the limit set
contains no equilibrium, only the first alternative is possible,
and the periodic orbit is nonconstant.
\end{proof}

\begin{rk}
{\rm The perturbed vector field $f+\varepsilon g$ is not assumed to belong
to $\mathcal L^-$. Hence Proposition~\ref{PB} cannot, in general, be applied
directly to the perturbed system; the robustness result used above is
needed for the second conclusion of Theorem~\ref{rob}.}
\end{rk}	}

The localization in \eqref{uniform tail localization} is illustrated
in Figure~\ref{fig:localization-projection}(a). The planar embedding
used below has a complementary geometric interpretation: on an
embedded set $K$, the inverse Lipschitz estimate gives
$|x-y|\le L|P(x-y)|$ for $x,y\in K$. Thus nonzero secants stay
uniformly away from $\ker P$; see panel (b). The existence of
such an estimate follows from \cite[Corollary~5.4 and Lemma~6.2]{Tl}
in their respective settings.

\begin{figure}[htbp]
\centering
\begin{tikzpicture}[font=\small,>=Stealth]
\begin{scope}
  \fill[black!3] (0,0) ellipse (2.35 and 1.65);
  \draw[black!65] (0,0) ellipse (2.35 and 1.65);
  \node at (-1.85,1.13) {$D_1$};
  \draw[diagramblue,fill=diagramblue!6] plot[smooth cycle,tension=.7]
    coordinates {(-1.45,-.15) (-.65,1.03) (.9,.95) (1.53,.02)
    (.8,-1.03) (-.65,-.94)};
  \node[diagramblue] at (1.06,.6) {$V$};
  \draw[diagramblue,dashed,fill=diagramblue!10]
    (0,0) ellipse (1.04 and .55);
  \node[font=\scriptsize,align=center] at (0,-.05)
    {$\overline W$};
  \draw[diagramred,line width=1.5pt] (-.55,.27)
    .. controls (-.15,.5) and (.13,.09) .. (.54,.31);
  \node[diagramred] at (0,.66) {$\mathcal A$};
  \draw[->,thick,diagramblue] (-1.9,-.45)
    .. controls (-1.55,-.5) and (-1.31,-.3) .. (-1.03,-.12);
  \node[font=\scriptsize] at (0,-1.3) {all late tails remain in $V$};
  \node at (0,-2.05) {(a) Uniform localization};
\end{scope}
\begin{scope}[xshift=7.05cm]
  \fill[diagramred!9] (0,0)--(-.63,1.65)--(.63,1.65)--cycle;
  \fill[diagramred!9] (0,0)--(-.63,-1.65)--(.63,-1.65)--cycle;
  \draw[diagramred!65,dashed] (-.63,-1.65)--(.63,1.65);
  \draw[diagramred!65,dashed] (.63,-1.65)--(-.63,1.65);
  \draw[->,black!70] (0,-1.72)--(0,1.92)
    node[above] {$\Pi=\ker P$};
  \draw[->,black!70] (-2.05,0)--(2.35,0);
  \node[below,font=\scriptsize] at (-1.32,0) {projection space};
  \draw[->,diagramblue,line width=1pt] (0,0)--(1.72,.96)
    node[above right] {$v=x-y$};
  \draw[dashed,black!65] (1.72,.96)--(1.72,0);
  \draw[->,diagramblue,line width=1pt] (0,0)--(1.72,0);
  \node[below,diagramblue] at (1.3,0) {$Pv$};
  \node[font=\scriptsize,align=center,diagramred] at (-1.53,1.35)
    {excluded\\directions};
  \draw[diagramred!65] (-.95,1.3)--(-.35,1.3);
  \node[fill=white,inner sep=2pt] at (0,-1.28) {$|v|\le L|Pv|$};
  \node at (0,-2.05) {(b) Inverse Lipschitz estimate};
\end{scope}
\end{tikzpicture}
\caption{Two ingredients of planar reduction. In (a), the invariant
set $\mathcal A$ determines neighborhoods
$W\Subset V\subset U\subset D_1$, and the perturbed images of $D$
remain in $\overline W$ after
time $T_*$. Panel (b) is a two-dimensional section of the secant
estimate for a bi-Lipschitz projection; the actual image space has
dimension two. The drawings indicate inclusions and directions,
not the geometry of a particular attractor.}
\label{fig:localization-projection}
\end{figure}
\FloatBarrier

	\section{Automatic transversality}\label{Automatic transversality111}

We begin with the local invariant manifolds used in the proof.
For a hyperbolic critical element $\gamma$, choose a sufficiently small open neighborhood $U_\gamma$ of $\gamma$,
with $U_\gamma\Subset\Omega$, adapted to the local stable and unstable
manifold theorem. Define
\[
\begin{aligned}
W^s_{\mathrm{loc}}(\gamma)
&=\left\{x\in W^s(\gamma)\cap U_\gamma:
\phi_f(t,x)\in U_\gamma\text{ for all }t\ge0\right\},\\
W^u_{\mathrm{loc}}(\gamma)
&=\left\{x\in W^u(\gamma)\cap U_\gamma:
\phi_f(t,x)\in U_\gamma\text{ for all }t\le0\right\}.
\end{aligned}
\]
These local manifolds are $C^1$ embedded submanifolds of $\Omega$
without boundary. Their global extensions are given by
\[
\begin{aligned}
W^s(\gamma)
&=\bigcup_{t\ge0}
\left\{\phi_f(-t,y):
y\in W^s_{\mathrm{loc}}(\gamma),\ [-t,0]\subset I_f(y)\right\},\\
W^u(\gamma)
&=\bigcup_{t\ge0}
\left\{\phi_f(t,y):
y\in W^u_{\mathrm{loc}}(\gamma),\ [0,t]\subset I_f(y)\right\}.
\end{aligned}
\]
The conditions on $I_f(y)$ ensure that every flow image in these
unions is defined.

	To prove Theorem~\ref{Automatic Transversality}, let $\gamma^-$ and
$\gamma^+$ be hyperbolic critical elements, let
$q\in W^u(\gamma^-)\cap W^s(\gamma^+)$, and let $c(t)$ be the
solution with $c(0)=q$. Then $\alpha(q)=\gamma^-$ and
$\omega(q)=\gamma^+$. If $q\in\gamma^-\cup\gamma^+$, the two critical
elements coincide and transversality at $q$ follows from hyperbolicity.
We therefore assume that $q\notin\gamma^-\cup\gamma^+$.
The orbit $c(\mathbb R)$ is then a nonconstant connecting orbit,
and $\dot c(t)$ is a nontrivial solution of
$\dot y=Df(c(t))y$. The endpoints $\gamma^-$ and $\gamma^+$ may coincide.
By Lemma~\ref{pro of N of L-S}(v), there exist $\bar t>0$ and
$h^-\ge h^+$ in $\{1,\ldots,(\widetilde n+1)/2\}$ such that
\[
N(\dot c(t))=2h^--1\quad(t<-\bar t),\qquad
N(\dot c(t))=2h^+-1\quad(t>\bar t).
\]

	The next two propositions relate these limiting values to the
dimensions and cone inclusions of the tangent spaces.
	\begin{prop}\label{Hyperbolic critical elements prop}
		Let $\gamma^-$ and $\gamma^+$ be hyperbolic critical elements
of \eqref{negative feedback system}, and let $c(t)$ be a
nonconstant connecting orbit as above. Let
$h^-,h^+\in\{1,\ldots,(\widetilde n+1)/2\}$ satisfy
$\lim_{t\to-\infty}N(\dot c(t))=2h^--1$ and
$\lim_{t\to+\infty}N(\dot c(t))=2h^+-1$.
		\begin{enumerate}
			\item[{\rm(i)}] If $\gamma^-$ is a hyperbolic periodic orbit, then $N(\dot{\gamma}^-(t))=2 h^{-}-1$, $\forall t \in \left ( -\infty ,+\infty  
			\right 
			)$. Moreover,
			\begin{itemize}
				\item if $h^->h^+$, then $T_{q} W^u\left(\gamma^{-}\right)$ contains a subspace $\tilde{\Sigma}_0^{-}$, such that
				$$
				\operatorname{dim}\tilde{\Sigma}_0^{-}=2 h^{-}-2 \quad \text { and } \quad \tilde{\Sigma}_0^{-} \subset \mathcal{K}_{h^{-}-1};
				$$
				\item if $\gamma^+$ is periodic and $h^{-}=h^{+}=h<\frac{\widetilde n+1}{2}$, then
				$$
				\operatorname{dim} T_{q} W^u\left(\gamma^{-}\right)=2 h \quad \text { and } \quad T_{q} W^u\left(\gamma^{-}\right) \subset 
				\mathcal{K}_h \text {; }
				$$
				\item if $\gamma^+$ is periodic and $h^{-}=h^{+}=\frac{\widetilde n+1}{2}$, then $T_{q} W^u\left(\gamma^{-}\right)=\mathbb{R}^n$.
			\end{itemize}
			
			\item[{\rm(ii)}] If $\gamma^+$ is a hyperbolic periodic orbit, then $N(\dot{\gamma}^+(t))=2 h^{+}-1$, $\forall t \in \left ( -\infty ,+\infty  
			\right 
			)$. Moreover,
			\begin{itemize}
				\item if $h^->h^+$, then $T_{q} W^s\left(\gamma^{+}\right)$ contains a subspace $\tilde{\Sigma}_0^{+}$, such that
				$$
				\operatorname{dim} \tilde{\Sigma}_0^{+}=n- 2 h^{-}+2 \quad \text { and } \quad \tilde{\Sigma}_0^{+} \subset \mathcal{K}^{h^{-}-1};
				$$
				\item if $\gamma^-$ is periodic and $h^{-}=h^{+}=h>1$, then
				$$
				\operatorname{dim} T_{q} W^s\left(\gamma^{+}\right)=n- 2 h+2 \quad \text { and } \quad T_{q} W^s\left(\gamma^{+}\right) \subset 
				\mathcal{K}^{h-1} \text {; }
				$$
				\item if $\gamma^-$ is periodic and $h^{-}=h^{+}=1$, then $T_{q} W^s\left(\gamma^{+}\right)=\mathbb{R}^n$.
			\end{itemize}
			\item[{\rm(iii)}] If $\gamma^-$ is a hyperbolic equilibrium, then $i\left(\gamma^-\right)\ge \max\left\{2h^--1,1\right\} $ .
			\item[{\rm(iv)}] If $\gamma^+$ is a hyperbolic equilibrium, then $i\left(\gamma^+\right)\le 2h^+-1 $ .
		\end{enumerate}
	\end{prop}
	
	\begin{proof}
We give the details for the periodic endpoints, following the
secant argument of \cite[pp.~873--875]{YD}. The difference of any
two solutions satisfies a linear equation with coefficients in
$\mathcal M^-$, by the argument in Corollary~\ref{subtract}.

For each periodic endpoint $\gamma^\pm$, let $p^\pm(t)$ be a
periodic parametrization.
Periodicity and strict decrease of $N$ imply that
$\dot p^\pm(t)\in\mathcal N$ and that its $N$-value is constant.
Asymptotic phase and continuity of $f$ then identify this value as
$2h^\pm-1$. Thus the trivial multiplier belongs to the spectral
block $W_{h^\pm}^\pm$. All preceding blocks are unstable, and all
subsequent blocks are stable. When $h^->h^+$, the subspaces
\[
W_1^-\oplus\cdots\oplus W_{h^--1}^-,\qquad
W_{h^-}^+\oplus\cdots\oplus W_m^+,
\quad m=(\widetilde n+1)/2,
\]
at the respective periodic endpoints have the dimensions and cone
inclusions in (i) and (ii). The $C^1$ regularity of the local
invariant manifolds gives nearby tangent subspaces with the
same inclusions.
Transport to $q$ preserves the lower-cone inclusion forward in
time and the upper-cone inclusion backward in time, by
Proposition~\ref{solution operator prop}. Each assertion uses only
the indicated periodic endpoint.

Now suppose that both endpoints are periodic and $h^-=h^+=h$.
Put $L=\gamma^-\cup c(\mathbb R)\cup\gamma^+$. We claim that
\begin{equation}\label{equal rank secants}
y-x\in\mathcal N,\qquad N(y-x)=2h-1
\quad\text{for all distinct }x,y\in L.
\end{equation}
On either periodic orbit, differences of distinct solutions are
periodic, so their $N$-values are constant in time. The space of
ordered distinct pairs on a circle is connected. Continuity of
$N$ and the limiting tangent direction as the two points approach
one another therefore identify this constant as $2h-1$.

For a time shift $\tau>0$ that is not an integer multiple of
either endpoint period, asymptotic phase gives this same value
for $c(t+\tau)-c(t)$ at both ends of the connection. Monotonicity
and strict decrease imply that it belongs to $\mathcal N$ with
value $2h-1$ for every $t$. Approximation by such shifts gives
the same conclusion for all $\tau>0$: a singular limiting
difference would have a strict drop at two nearby times, where
continuity of $N$ contradicts the approximating constant values.
Finally, approximate any pair in $L$ by pairs on $c(\mathbb R)$.
If their nonzero limiting difference were singular, apply the
same strict-drop argument to the two complete trajectories
through the limiting pair. This proves
\eqref{equal rank secants}.

Choose the phase of $p^-$ so that, for all sufficiently negative
$t$, the point $c(t)$ lies in the local strong unstable fiber
through $p^-(t)$. Along a sequence $t_k\to-\infty$ of integer
multiples of its period, these fibers are based at $p^-(0)$.
A subsequence of
\[
w_k=\frac{c(t_k)-p^-(0)}{|c(t_k)-p^-(0)|}
\]
converges to a unit vector $w$ tangent to that fiber. In particular,
$w$ is an unstable direction and is independent of $\dot p^-(0)$.
The normalized differences
\[
\frac{c(t_k+t)-p^-(t)}{|c(t_k)-p^-(0)|}
\]
converge on compact time intervals to the solution $v(t)$ of
$\dot v=Df(p^-(t))v$ with $v(0)=w$. By
\eqref{equal rank secants}, every approximating difference has
$N$-value $2h-1$. At each regular time of $v$, continuity gives
the same value; the strict-drop property excludes singular
times. Hence $N(v(t))=2h-1$ for all $t\in\mathbb R$.
Applying the spectral ordering in Theorem~\ref{Floquet Theory}
to positive and negative powers of the monodromy operator
shows that $w\in W_h^-$. Indeed, the extreme nonzero spectral
components dominate in the two time directions, and their
$N$-values must both equal $2h-1$.

Thus $W_h^-$ contains both the tangent direction and a nonzero
unstable direction. At the positive endpoint, choose the phase
of $p^+$ so that $c(t)$ lies in the local strong stable fiber
through $p^+(t)$ for all sufficiently positive $t$. The same
argument shows that $W_h^+$ contains a nonzero stable direction
in addition to the tangent direction. If $n$ is odd and $h=m$,
then $W_m^+$ is one-dimensional and consists only of the flow
direction. It cannot contain this nonzero strong stable
direction, so this case is impossible.

In all remaining cases, both blocks are two-dimensional, and
their nontrivial multipliers lie outside and inside the unit
circle, respectively. The full tangent spaces at the periodic
orbits are therefore
\[
T_{p^-(0)}W^u(\gamma^-)=W_1^-\oplus\cdots\oplus W_h^-,
\qquad
T_{p^+(0)}W^s(\gamma^+)=W_h^+\oplus\cdots\oplus W_m^+.
\]
Their dimensions are $2h$ and $n-2h+2$. Compactness of the
periodic orbits and $C^1$ tangency give the corresponding cone
inclusions nearby, which propagate to $q$ as above.
The remaining endpoint cases use
$\mathcal K^0=\mathbb R^n$ and $\mathcal K_m=\mathbb R^n$.
This proves (i) and (ii).

It remains to prove (iii) and (iv). We prove (iv); the same
argument on the local unstable manifold proves (iii).

		Write $\gamma^+=e^+$. Apply Theorem~\ref{Floquet Theory} to
$S_{Df(e^+)}(1,0)=e^{Df(e^+)}$, and let
$k=\min\{i:\nu_i<1\}$, $k\ge1$. Then $T_{e^{+}} W^s(e^{+}) \subset W_k \oplus \cdots \oplus 
		W_\frac{\widetilde n+1 }{2} \subset \mathcal{K}^{k-1}$. {Let $E^s=T_{e^+}W^s(e^+)$.
Since $E^s\cap S^{n-1}$ is compact and
$\mathcal{K}^{k-1}\setminus\{0\}$ is open, there exists $\eta>0$ such that
every unit vector whose distance from $E^s\cap S^{n-1}$ is less
than $\eta$ belongs to $\mathcal{K}^{k-1}$. 
Near $e^+$, write the local stable manifold $W_{loc}^s(e^+)$
as $e^++u+g(u)$, where $u\in E^s$, $g(0)=0$, and $Dg(0)=0$.
At $x=e^++u+g(u)$, one has
\[
T_xW^s(e^+)=\{v+Dg(u)v:\ v\in E^s\}.
\]
By the continuity of $Dg$, after shrinking the neighborhood if
necessary, we may assume that $\|Dg(u)\|<\varepsilon$, where
$\varepsilon>0$ is chosen so that
$2\varepsilon/(1-\varepsilon)<\eta$. For
$w=v+Dg(u)v$ with $v\neq0$,
\[
\left\|
\frac{w}{\|w\|}-\frac{v}{\|v\|}
\right\|
\le \frac{2\varepsilon}{1-\varepsilon}<\eta.
\]
Hence $w/\|w\|\in\mathcal K^{k-1}$. Since
$N_m(\lambda w)=N_m(w)$ for every $\lambda\neq0$, it follows that
$w\in \mathcal{K}^{k-1}$. Therefore,
$T_xW^s(e^+)\setminus\{0\}\subset \mathcal{K}^{k-1}$ for every
$x\in W_{loc}^s(e^+)$ sufficiently close to $e^+$.

Since $c(t)\to e^+$ as $t\to+\infty$, we can choose $t^+\gg1$ such
that
\[
T_{c(t^+)}W^s(e^+)\setminus\{0\}\subset \mathcal{K}^{k-1}.
\]}Since $\dot{c}\left(t^{+}\right) \in T_{c\left(t^{+}\right)} W^s(e^+) \subset 
		\mathcal{K}^{k-1} $, and 
		$N\left(\dot{c}\left(t^{+}\right)\right)=2 h^{+}-1$, we have $h^+ \ge k$ and $\nu_{h^+}<1$. Therefore, $i\left(\gamma^+\right)\le 2h^+-1. $ 
	\end{proof}

	\begin{prop}\label{equilibrium point prop}
		Let $e$ be a hyperbolic equilibrium of
\eqref{negative feedback system}, and let $a\in\Omega$ and
$h\in\{1,\ldots,(\widetilde n-1)/2\}$. Then
		\begin{enumerate}
			\item[{\rm(i)}] if $a\in W^u\left(e\right)$, $i\left(e\right)\ge 2h $, then there exists a subspace $\tilde{\Sigma}_0^{-} \subset T_{a} W^u(e)$ 
			such that $\operatorname{dim} \tilde{\Sigma}_0^{-}=2h $ and  $\tilde{\Sigma}_0^{-} \subset \mathcal{K}_{h}$.
			
			\item[{\rm(ii)}] if $a\in W^s\left(e\right)$, $i\left(e\right)\le 2h $, then there exists a subspace $\tilde{\Sigma}_0^{+} \subset T_{a} W^s(e)$ 
			such that $\operatorname{dim} \tilde{\Sigma}_0^{+}=n-2h$ and  $\tilde{\Sigma}_0^{+} \subset \mathcal{K}^{h}$.
		\end{enumerate}
	\end{prop}
	
	\begin{proof}
		We prove (ii); the proof of (i) is analogous. Let ${\Sigma}^{+}=W_{h+1} \oplus \cdots \oplus W_\frac{\widetilde n+1 }{2}$. 
		Theorem~\ref{Floquet Theory} and $i(e)\le2h$ give $\operatorname{dim}{\Sigma}^{+}=n-2h$, ${\Sigma}^{+}\subset T_e 
		W^s(e)$ and ${\Sigma}^{+}\setminus\{0\}\subset \mathcal{K}^{h}$. 
		Let $z\left(t\right)$ be a solution of  (\ref{negative feedback system}) with $z\left(0\right)=a$. {Since $\Sigma^+\setminus\{0\}\subset \mathcal{K}^h$ and
$\mathcal{K}^h\setminus\{0\}$ is open, we first show that this inclusion
persists for a suitable subspace of $T_xW^s(e)$ when $x$ is
sufficiently close to $e$. Let $E^s=T_eW^s(e)$.
Near $e$, the $C^1$ manifold $W_{\rm loc}^s(e)$ can be represented as
$e+u+g(u)$, where $u\in E^s$, $g(0)=0$, and $Dg(0)=0$.
Hence, at $x=e+u+g(u)$,
\[
T_xW^s(e)=\{v+Dg(u)v:\ v\in E^s\}.
\]
Since $\Sigma^+\cap S^{n-1}$ is compact and contained in the open
set $\mathcal{K}^h\setminus\{0\}$, and since $Dg(u)\to0$ as $u\to0$, there
exists a neighborhood $U$ of $e$ such that, for every
$x=e+u+g(u)\in W_{\rm loc}^s(e)\cap U$, the subspace
\[
\widetilde\Sigma^+(x)
:=
\{v+Dg(u)v:\ v\in\Sigma^+\}
\]
satisfies
$\dim\widetilde\Sigma^+(x)=n-2h$ and
$\widetilde\Sigma^+(x)\setminus\{0\}\subset \mathcal{K}^h$.

Since $z(t)\to e$ as $t\to+\infty$, we may choose $t^+\gg1$ such
that $z(t^+)\in U$. Set
$\widetilde\Sigma^+=\widetilde\Sigma^+(z(t^+))$. Then
\[
\widetilde\Sigma^+\subset T_{z(t^+)}W^s(e),\qquad
\dim\widetilde\Sigma^+=n-2h,\qquad
\widetilde\Sigma^+\setminus\{0\}\subset \mathcal{K}^h .
\]

Define
\[
\widetilde\Sigma_0^+
=
S_{Df(z(t))}(0,t^+)\widetilde\Sigma^+ .
\]
Since $W^s(e)$ is invariant and the finite-time linearized evolution
operator is an isomorphism,
$\widetilde\Sigma_0^+\subset T_aW^s(e)$ and
$\dim\widetilde\Sigma_0^+=n-2h$. Moreover, by Proposition~\ref{solution operator prop},
\[
S_{Df(z(t))}(0,t^+)
\bigl(\overline{\mathcal K^h}\setminus\{0\}\bigr)
\subset
\operatorname{Int}\overline{\mathcal K^h}
=
\mathcal{K}^h\setminus\{0\}.
\]
Hence, $\widetilde{\Sigma}_0^+\setminus\{0\}\subset \mathcal{K}^h$.}
	\end{proof}
	
	\vskip 2mm
	
	We can now combine these tangent-space estimates to prove
automatic transversality.
	
	\begin{proof}[Proof of Theorem \ref{Automatic Transversality}]
		If the two manifolds are disjoint, transversality is automatic.
Otherwise, fix $q\in W^u(\gamma^-)\cap W^s(\gamma^+)$ and put
$m=(\widetilde n+1)/2$. If $q$ belongs to a critical element,
transversality follows from hyperbolicity. We may therefore assume
that $q$ lies on a nonconstant connection and distinguish the
following four cases.
		\begin{enumerate}
\item[{\rm(i)}] Both endpoints are periodic. If $h^->h^+$,
Proposition~\ref{Hyperbolic critical elements prop} supplies subspaces
of dimensions $2h^--2$ and $n-2h^-+2$ in
$\mathcal K_{h^--1}$ and $\mathcal K^{h^--1}$, respectively.
Their intersection is zero, so they span $\mathbb R^n$.
If $h^-=h^+=h<m$, the unstable tangent space has dimension $2h$
and lies in $\mathcal K_h$. The stable tangent space contains
an $(n-2h)$-dimensional subspace in $\mathcal K^h$, obtained by
transporting $W_{h+1}^+\oplus\cdots\oplus W_m^+$ backward
along the connection. These subspaces again span $\mathbb R^n$.
For $h=m$, either the unstable manifold is open or an incoming
connection is impossible, as in the preceding proposition.

\item[{\rm(ii)}] The negative endpoint is an equilibrium and the
positive endpoint is periodic. When $h^->h^+$, use the stable
subspace in Proposition~\ref{Hyperbolic critical elements prop}(ii)
and the unstable subspace in Proposition~\ref{equilibrium point prop}
with index $h^--1$. Their dimensions add to $n$ and their cones
intersect only at zero.

Suppose $h^-=h^+=h$. Write $e=\gamma^-$ and let $W_h^+$
be the block containing the trivial multiplier of $\gamma^+$.
If it is one-dimensional, there is no incoming connection.
Otherwise let $\lambda_+$ be its other multiplier. It is real,
positive, and different from $1$, by Corollary~\ref{simple}
and hyperbolicity. If $\lambda_+<1$, the stable tangent space
at the periodic orbit contains $W_h^+\oplus\cdots\oplus W_m^+$.
Its transported subspace lies in $\mathcal K^{h-1}$ and has
dimension $n-2h+2$. Since $i(e)\ge2h-1$, an unstable subspace
of dimension $2h-2$ in $\mathcal K_{h-1}$ gives transversality.
For $h=1$, the stable manifold is open.

If $\lambda_+>1$, it suffices to show that $i(e)\ge2h$.
Indeed, one may then use an unstable subspace of dimension $2h$
in $\mathcal K_h$ and the transported stable subspace
$W_{h+1}^+\oplus\cdots\oplus W_m^+$ in $\mathcal K^h$.
If $2h=n$, the unstable manifold is open. The only remaining
index allowed by Proposition~\ref{Hyperbolic critical elements prop}(iii)
is $i(e)=2h-1$. Let $\tau$ be the period of $\gamma^+$ and set
$z(t)=c(t+\tau)-c(t)\ne0$. Near $e$, the local unstable
manifold is a $C^1$ graph tangent to a subspace of
$W_1^-\oplus\cdots\oplus W_h^-$. Its secants lie in
$\mathcal K_h$, as in the proof of Proposition~\ref{one to one for connect}.
Thus $z(t)\in\mathcal K_h$ for sufficiently negative $t$.
Choose the asymptotic phase $p^+(t)$ of $c(t)$ at $\gamma^+$.
Since $p^+(t+\tau)=p^+(t)$, the two points lie in the same local
strong stable fiber through $p^+(t)$ for sufficiently positive
$t$. Its tangent space at $p^+(t)$ is the sum
$W_{h+1}^+\oplus\cdots\oplus W_m^+$ of spectral subspaces
at that phase. Uniform $C^1$ tangency along the periodic orbit gives
$z(t)\in\mathcal K^h$. These two conclusions contradict
monotonicity of $N$. Hence $i(e)\ge2h$, as required.

\item[{\rm(iii)}] The negative endpoint is periodic and the positive
endpoint is an equilibrium $e$. For $h^->h^+$, combine the
unstable subspace in Proposition~\ref{Hyperbolic critical elements prop}(i)
with Proposition~\ref{equilibrium point prop}, using
$i(e)\le2h^+-1\le2h^--2$.

For $h^-=h^+=h$, a one-dimensional neutral block occurs only
at the last block in odd dimension, when the unstable manifold
is open. Otherwise denote the other multiplier in $W_h^-$ by
$\lambda_-$. If $\lambda_->1$, the unstable tangent space has
dimension $2h$ and lies in $\mathcal K_h$ after transport to
$q$. Since $i(e)\le2h-1$, Proposition~\ref{equilibrium point prop}
supplies a complementary stable subspace in $\mathcal K^h$;
the case $2h=n$ is immediate.

If $\lambda_-<1$, we claim that $i(e)\le2h-2$.
Otherwise $i(e)=2h-1$. Taking $\tau$ to be the period of
$\gamma^-$, choose the asymptotic phase $p^-(t)$ of $c(t)$ at
$\gamma^-$. Since $p^-(t+\tau)=p^-(t)$, the points $c(t+\tau)$
and $c(t)$ belong to the same local strong unstable fiber
through $p^-(t)$ for sufficiently negative $t$. Hence their
secant $z(t)=c(t+\tau)-c(t)$ lies in $\mathcal K_{h-1}$.
Near $e$, secants of its local stable manifold lie in
$\mathcal K^{h-1}$.
This again contradicts monotonicity of $N$. Thus
$i(e)\le2h-2$, and subspaces of dimensions $2h-2$ and
$n-2h+2$ in $\mathcal K_{h-1}$ and $\mathcal K^{h-1}$ give
transversality. If $h=1$ and $\lambda_-<1$, the periodic orbit
has no unstable fiber, so no outgoing connection exists.

			\item[{\rm(iv)}] 
            {$\gamma^-$ and $\gamma^+$ are hyperbolic equilibria with
\[
i(\gamma^-)\ge 2h\ge i(\gamma^+),
\qquad
h\in\left\{1,\ldots,\frac{\widetilde n+1}{2}\right\}.
\]

First assume that
$h\le (\widetilde n-1)/2$. By Proposition~\ref{equilibrium point prop}, there exist subspaces
$\widetilde{\Sigma}_0^-\subset T_qW^u(\gamma^-)$ and
$\widetilde{\Sigma}_0^+\subset T_qW^s(\gamma^+)$ such that
\[
\dim\widetilde{\Sigma}_0^-=2h,\qquad
\dim\widetilde{\Sigma}_0^+=n-2h,
\]
and
$\widetilde{\Sigma}_0^-\subset \mathcal{K}_h$ and
$\widetilde{\Sigma}_0^+\subset \mathcal{K}^h$.
Since $\mathcal{K}_h\cap \mathcal{K}^h=\{0\}$, we have
$\widetilde{\Sigma}_0^-\cap\widetilde{\Sigma}_0^+=\{0\}$.
Therefore,
\[
\mathbb R^n
=
\widetilde{\Sigma}_0^-\oplus\widetilde{\Sigma}_0^+
\subset
T_qW^u(\gamma^-)+T_qW^s(\gamma^+),
\]
which proves transversality.

It remains to consider
$h=(\widetilde n+1)/2$. If $n$ is odd, then
$2h=n+1$, so the condition $i(\gamma^-)\ge2h$ is impossible since
$i(\gamma^-)\le n$. If $n$ is even, then $2h=n$, and hence
$i(\gamma^-)=n$. Thus,
$T_qW^u(\gamma^-)=\mathbb R^n$, and transversality follows
immediately.}
		\end{enumerate}
		This completes the proof of Theorem \ref{Automatic Transversality}.
	\end{proof}

\section{Generic hyperbolicity of critical elements}\label{GTC}
	We prove Theorem~\ref{critical elements is generic} by treating equilibria
and periodic orbits separately. We begin with the topology on the space
of vector fields. For $\alpha>0$, set
\[
K_\alpha
:=
\left\{
x\in\Omega:
\|x\|\le\alpha,\ 
\operatorname{dist}(x,\Omega^c)\ge\alpha^{-1}
\right\},
\]
where $\operatorname{dist}(x,\varnothing)=+\infty$.
For $m\ge1$, define
$p_m(f):=\sup_{x\in K_m}
\bigl(\|f(x)\|+\|Df(x)\|\bigr)$,
with $p_m(f)=0$ when $K_m=\varnothing$.
Every compact subset of $\Omega$ is contained in some $K_m$.

For $f\in C^1(\Omega,\mathbb R^n)$, the sets
\[
V_{f,m,\varepsilon}
:=
\left\{
g\in C^1(\Omega,\mathbb R^n):
p_m(g-f)<\varepsilon
\right\},
\qquad m\ge1,\quad \varepsilon>0,
\]
form a neighborhood basis at $f$. The resulting topology, denoted
by $\bar\iota$, is the weak $C^1$ topology, or equivalently the
compact-open $C^1$ topology;
see \cite[Chapter~2, Section~1]{Hirsch}.
Since every compact subset of $\Omega$ is contained in some $K_m$, $f_k\to f$ with respect to $\bar\iota$ if and only if
$p_m(f_k-f)\to0$ for every fixed $m\ge1$, or equivalently, if and
only if $f_k\to f$ and $Df_k\to Df$ uniformly on every compact
subset of $\Omega$.

A compatible metric is
\[
d(f,g)
=
\sum_{m=1}^{\infty}
2^{-m}\frac{p_m(f-g)}{1+p_m(f-g)}.
\]
We equip all subspaces below with the corresponding relative
topology. In particular, $\iota:=\bar\iota|_{\mathcal L^-}$,
while $\iota_d=\iota|_{\mathcal L_d^-}$, as specified in
Section~\ref{Main results}.

	Following the sign convention in \eqref{transformation}, for
$f\in \mathcal C^1_{\mathrm{BF}}$ and $x\in\Omega$, define
$\beta_i(f,x)=\partial_{x_{i+1}}f_i(x)$ and
$\gamma_i(f,x)=\partial_{x_i}f_{i+1}(x)$ for $1\le i<n$, while
$\beta_n(f,x)=-\partial_{x_1}f_n(x)$ and
$\gamma_n(f,x)=-\partial_{x_n}f_1(x)$.

We also set
\[
\vartheta_f(x)
=
\prod_{i=1}^n\beta_i(f,x)
+
\prod_{i=1}^n\gamma_i(f,x).
\]

\begin{prop}\label{Baire}
    The space $(\mathcal L^-,\iota)$, where
$\iota=\bar\iota|_{\mathcal L^-}$, is completely metrizable and hence is a
Baire space. Moreover, $\mathcal L^-$ is second countable.
\end{prop}

\begin{proof}
We first observe that $\mathcal C^1_{\mathrm{BF}}$ is a closed
linear subspace of $C^1(\Omega,\mathbb R^n)$. Indeed, suppose that
$f^{(k)}\in \mathcal C^1_{\mathrm{BF}}$ converges to
$f\in C^1(\Omega,\mathbb R^n)$ with respect to $\bar\iota$. Then
$p_m(f^{(k)}-f)\to0$ for every $m$, and hence
$Df^{(k)}\to Df$ uniformly on $K_m$.
For each $i$, since $f_i^{(k)}$ depends only on
$x_{i-1},x_i,x_{i+1}$, we have
$\partial_{x_j}f_i^{(k)}=0$ whenever
$j\notin\{i-1,i,i+1\}$, with the indices understood cyclically.
Passing to the limit gives $\partial_{x_j}f_i=0$ on $K_m$.
Since $\Omega=\bigcup_{m\ge1}K_m$, the same holds throughout $\Omega$.
As $\Omega$ is convex, $f_i$ depends only on
$x_{i-1},x_i,x_{i+1}$. Thus $f\in \mathcal C^1_{\mathrm{BF}}$, which proves that this subspace is closed. 

Since
$C^1(\Omega,\mathbb R^n)$ endowed with $\bar\iota$ is a Fr\'echet
space, $\mathcal C^1_{\mathrm{BF}}$ is also a Fr\'echet space
and therefore completely metrizable. Define
\[
\mathcal A
=
\left\{
f\in\mathcal C^1_{\mathrm{BF}}:
\beta_i(f,x)\ge0,\ 
\gamma_i(f,x)\ge0
\text{ for all }x\in\Omega,\ 1\le i\le n
\right\}.
\]
If $f^{(k)}\in\mathcal A$ converges to
$f\in \mathcal C^1_{\mathrm{BF}}$ with respect to
$\bar\iota$, then the first derivatives converge uniformly on every
$K_m$. Hence the inequalities defining $\mathcal A$ pass to the limit.
Therefore, $\mathcal A$ is closed in
$\mathcal C^1_{\mathrm{BF}}$, and thus is completely
metrizable.

By the definition of $\mathcal M^-$,
\[
\mathcal L^-
=
\{f\in\mathcal A:
\vartheta_f(x)>0\text{ for every }x\in\Omega\}.
\]
Since $\Omega=\bigcup_{m\ge1}K_m$, after omitting any initial empty
$K_m$'s if necessary,
\[
\mathcal L^-
=
\bigcap_{m=1}^{\infty}
\left\{
f\in\mathcal A:
\min_{x\in K_m}\vartheta_f(x)>0
\right\}.
\]
For each fixed $m$, the map
$f\mapsto\min_{x\in K_m}\vartheta_f(x)$ is continuous with respect to
$\bar\iota$. Indeed, $p_m(f^{(k)}-f)\to0$ implies uniform convergence on
$K_m$ of all the functions $\beta_i(f^{(k)},\cdot)$ and
$\gamma_i(f^{(k)},\cdot)$, and hence uniform convergence of
$\vartheta_{f^{(k)}}$ to $\vartheta_f$. Therefore, each set in the
intersection above is open in $\mathcal A$.

Thus, $\mathcal L^-$ is a $G_\delta$ subset of the completely metrizable
space $\mathcal A$. By \cite[Theorem 24.12]{Willard1970}, $\mathcal L^-$
is completely metrizable and hence is a Baire space.

Finally, $C^1(\Omega,\mathbb R^n)$ endowed with $\bar\iota$ is second
countable. Therefore, its subspace $\mathcal L^-$ is also second
countable.
\end{proof}

\begin{rk}[{Weak and strong $C^1$ topologies}]
\label{rk:weak-strong-topologies}
\rm The compact-open topology $\bar\iota$ differs from the strong
Whitney $C^1$ topology. On $C^1(\Omega,\mathbb R^n)$, the latter
has a neighborhood basis at $f$ consisting of the sets
\[
\mathcal W(f,\varepsilon)
:=
\left\{
g\in C^1(\Omega,\mathbb R^n):
\begin{array}{l}
\|g(x)-f(x)\|+\|Dg(x)-Df(x)\|<\varepsilon(x)\\
\text{for every }x\in\Omega
\end{array}
\right\},
\]
where $\varepsilon:\Omega\to(0,\infty)$ is continuous;
see \cite[Chapter~2, Sections~1 and~4]{Hirsch}.
In particular, $\varepsilon$ need not have a positive lower bound
on $\Omega$. When $n\ge4$, the strong Whitney topology inherited from
$C^1(\Omega,\mathbb R^n)$ is discrete on $\mathcal C^1_{\mathrm{BF}}$,
meaning that every singleton is open in the relative topology.
To prove this, fix $f\in \mathcal C^1_{\mathrm{BF}}$ and set
\[
\varepsilon_0(x)
:=
\frac{\min\{1,\operatorname{dist}(x,\Omega^c)\}}
{1+\|x\|},
\qquad x\in\Omega.
\]
Our convention for the distance to the empty set makes $\varepsilon_0$
positive and continuous also when $\Omega=\mathbb R^n$. 

We claim that
$\mathcal W(f,\varepsilon_0)\cap \mathcal C^1_{\mathrm{BF}}=\{f\}$.
Let $g$ belong to this intersection and put $h=g-f$.
Fix $x^0\in\Omega$ and $i\in\{1,\ldots,n\}$.
Since $n\ge4$, there exists $j\notin\{i-1,i,i+1\}$, with indices
understood cyclically. Let $e_j$ be the $j$-th coordinate vector.
As $\Omega$ is open and convex,
$I:=\{t\in\mathbb R:x^0+te_j\in\Omega\}$ is an open interval
containing zero. Write $b:=\sup I\in(0,+\infty]$.
The structural restriction on $h_i$ gives
$h_i(x^0+te_j)=h_i(x^0)$ for every $t\in I$. Hence
\[
|h_i(x^0)|
=
|h_i(x^0+te_j)|
\le \|h(x^0+te_j)\|
<
\varepsilon_0(x^0+te_j),
\qquad t\in I.
\]
If $b=+\infty$, the right-hand side tends to zero as
$t\to+\infty$. If $b<+\infty$, then
$x^0+be_j\in\partial\Omega$, so it again tends to zero as
$t\to b$. Thus $h_i(x^0)=0$.
Since $x^0$ and $i$ were arbitrary, $g=f$, proving the claim.

Consequently, the induced strong Whitney topology is discrete on
both $\mathcal C^1_{\mathrm{BF}}$ and $\mathcal L^-$.
A discrete space is Baire, but its only dense subset is the whole
space. Genericity in this topology would therefore require the
property to hold for every vector field.
This observation concerns the topology inherited from the ambient
$n$-variable mapping space; it does not concern a product of strong
Whitney topologies imposed separately on the three-variable
component spaces.
\end{rk}

Theorem~\ref{critical elements is generic} follows from the next two propositions.
	\begin{prop}\label{G-H-P}
		For any $m\in \mathbb{N}\setminus \left \{ 0 \right \}$, the set 
		$$
		\mathcal{E}_m^{h}=\left\{f \in \mathcal{L}^- \mid \text {each equilibrium } e\in K_m \text { of } \dot{x} = f\left ( x \right ) \text { is 
			hyperbolic}\right\}
		$$ 
		is an open and dense subset of $\mathcal{L}^-$. Consequently, the set $\mathcal{E}^{h} =\bigcap_{m=1}^{\infty }\mathcal{E}_m^{h}$ is a generic subset of 
		$\mathcal{L}^-$.
	\end{prop}
	
	\begin{prop}\label{G-H-O-M}
		For any $m\in \mathbb{N}\setminus \left \{ 0 \right \}$, the set 
		\begin{align*}
			\mathcal O_m^h=\bigl\{f\in\mathcal E_m^h:\;&
\text{every periodic orbit of }\dot x=f(x)\text{ contained in }K_m\\
&\text{with minimal period at most }m\text{ is hyperbolic}\bigr\}
\end{align*}
		is an open and dense subset of $\mathcal{L}^-$. Consequently, the set $\mathcal{O}^{h} =\bigcap_{m=1}^{\infty }\mathcal{O}_m^h$ is a generic subset of 
		$\mathcal{L}^-$.
	\end{prop}
	
	\subsection{Proof of Proposition \ref{G-H-P}}
	
	We first establish generic simplicity of equilibria.
	
	\begin{defn}\label{simple critical elements 2}
		{\rm 
			An equilibrium $e$ of \eqref{negative feedback system} is called simple if $Df(e)$ is invertible.
		}
	\end{defn}
	
	\begin{lem}\label{simple points generic}
		For any $m\in \mathbb{N}\setminus \left \{ 0 \right \}$, the set 
		$$
		\mathcal{E}_m^{\mathrm{simple}}=\left\{f \in \mathcal{L}^- \mid \text {each equilibrium } e\in K_m \text { of } \dot{x} = f\left ( x \right ) 
		\text { is simple}\right\}
		$$ 
		is an open and dense subset of $\mathcal{L}^-$. Consequently, the set $\mathcal{E}^{\mathrm{simple}} =\bigcap_{m=1}^{\infty }\mathcal{E}_m^{\mathrm{simple}}$ is a generic subset 
		of $\mathcal{L}^-$.
	\end{lem}
	
	Assuming Lemma~\ref{simple points generic}, we first prove
Proposition~\ref{G-H-P}. We then return to the proof of the lemma.

	\begin{proof}[Proof of Proposition \ref{G-H-P}]
		It suffices to prove that each $\mathcal E_m^h$ is open and dense in $\mathcal L^-$. The Baire property then implies that their intersection is dense and residual.
		
		We first prove that $\mathcal{E}_m^h$ is open in $\mathcal L^-$. Suppose, to the
contrary, that there exist a sequence $\{f_k\}_{k\in\mathbb N}\subset \mathcal L^-\setminus \mathcal{E}_m^h$
and $f\in \mathcal{E}_m^h$ such that $f_k\to f$ with respect to $\iota$. For each $k$, choose a nonhyperbolic equilibrium $e_k\in K_m$ of
$f_k$. Thus, $f_k(e_k)=0$, and there exists $\lambda_k\in \sigma(Df_k(e_k))$
such that $\operatorname{Re}\lambda_k=0$. By compactness of $K_m$, we may pass to a subsequence such that
$e_k\to e\in K_m$. The convergence $f_k\to f$ in $\iota$ gives
$p_m(f_k-f)\to0$, so
\[
\|f(e)\|\le \|f(e)-f(e_k)\|
+\sup_{x\in K_m}\|f(x)-f_k(x)\|
\to 0.
\]
Therefore, $f(e)=0$. Moreover,
\[
\begin{aligned}
\|Df_k(e_k)-Df(e)\|
\le
\sup_{x\in K_m}\|Df_k(x)-Df(x)\|
+\|Df(e_k)-Df(e)\|
\to 0.
\end{aligned}
\]
Set $A_k=Df_k(e_k)$ and $A=Df(e)$. Then $A_k\to A$, so the sequence $\{A_k\}_{k\in\mathbb N}$ is bounded. Since $\lambda_k$ is an
eigenvalue of $A_k$, $|\lambda_k|\le \|A_k\|$,
where $\|\cdot\|$ denotes the induced operator norm. Thus,
$\{\lambda_k\}_{k\in\mathbb N}$ is bounded in $\mathbb C$. Passing to
a further subsequence, we may assume that
$\lambda_k\to\lambda\in\mathbb C$. For every $k$, $\det(\lambda_k I-A_k)=0$.
Passing to the limit in this identity gives $\det(\lambda I-A)=0$.
Hence, $\lambda\in\sigma(A)=\sigma(Df(e))$.
Moreover, $\operatorname{Re}\lambda
=
\lim_{k\to\infty}\operatorname{Re}\lambda_k
=
0.$ Thus, $e\in K_m$ is a nonhyperbolic equilibrium of $f$, contradicting
$f\in \mathcal{E}_m^h$. Therefore, $\mathcal{E}_m^h$ is open in $\mathcal L^-$.
		
		We next prove that $\mathcal E_m^h$ is dense in $\mathcal L^-$. Let
$f_0\in\mathcal L^-$ and let $\mathcal V$ be an arbitrary open neighborhood
of $f_0$. By Lemma~\ref{simple points generic}, we may choose
$f\in\mathcal V\cap\mathcal E_m^{\mathrm{simple}}$. There are only finitely many equilibria of $f$ in $K_m$; denote them by
$e_1,\ldots,e_N$. Indeed, they are isolated, and the zero set of $f$ in
$K_m$ is compact. If $N=0$, then $f\in\mathcal E_m^h$, and there is nothing to prove.

Assume $N\ge1$. Choose pairwise disjoint neighborhoods
$U_1,\ldots,U_N$ of $e_1,\ldots,e_N$, respectively. By the implicit
function theorem, after shrinking $U_1,\ldots,U_N$ if necessary, there
exists a neighborhood $\mathcal W\subset\mathcal V$ of $f$ such that
every $g\in\mathcal W$ has a unique simple equilibrium in each $U_i$.
Moreover, since $f\neq0$ on the compact set
$K_m\setminus\bigcup_{i=1}^N U_i$, shrinking $\mathcal W$ further if
necessary ensures that these are all the equilibria of $g$ in $K_m$.

Set $g_0=f$. Suppose that $g_{j-1}\in\mathcal W$ has been constructed,
and let $e_i^{\,j-1}$ denote its unique equilibrium in $U_i$. Assume that
$e_1^{\,j-1},\ldots,e_{j-1}^{\,j-1}$ are hyperbolic. For
$\alpha\in\mathbb R$, define
\[
g_{j,\alpha}(x)
=
g_{j-1}(x)+\alpha\bigl(x-e_j^{\,j-1}\bigr).
\]
Then $e_j^{\,j-1}$ remains an equilibrium and
\[
Dg_{j,\alpha}(e_j^{\,j-1})
=
Dg_{j-1}(e_j^{\,j-1})+\alpha I.
\]
Hence $e_j^{\,j-1}$ is hyperbolic provided
\[
\alpha\notin
\left\{
-\operatorname{Re}\lambda:
\lambda\in\sigma\bigl(Dg_{j-1}(e_j^{\,j-1})\bigr)
\right\}.
\]
Since the excluded set is finite, such $\alpha$ can be chosen
arbitrarily close to zero. By the openness of hyperbolicity, we may
choose $\alpha_j$ sufficiently small so that
$g_j:=g_{j,\alpha_j}\in\mathcal W$, $e_j^{\,j-1}$ is hyperbolic, and
the equilibria of $g_j$ in $U_1,\ldots,U_{j-1}$ remain hyperbolic.

The perturbation $\alpha_j(x-e_j^{\,j-1})$ changes only the diagonal
entries of the Jacobian. Hence
$\beta_i(g_j,x)=\beta_i(g_{j-1},x)$ and
$\gamma_i(g_j,x)=\gamma_i(g_{j-1},x)$ for all $x\in\Omega$ and
$1\le i\le n$. Therefore, $g_j\in\mathcal L^-$.

Repeating this construction for $j=1,\ldots,N$ gives
$g_N\in\mathcal W\subset\mathcal V$ such that every equilibrium of
$g_N$ in $K_m$ is hyperbolic. Thus,
$g_N\in\mathcal E_m^h$, and $\mathcal E_m^h$ is dense in
$\mathcal L^-$.
	\end{proof}

	\begin{proof}[Proof of Lemma \ref{simple points generic}]
		It suffices to prove openness and density of $\mathcal E_m^{\mathrm{simple}}$ for each $m$. The Baire property then gives the assertion about their intersection.
		
		{We first prove that $\mathcal{E}_m^{\mathrm{simple}}$ is open in $\mathcal L^-$. 
The argument is the same as for $\mathcal E_m^h$. Suppose, to the 
contrary, that there exist a sequence $\{f_k\}_{k\in\mathbb N}
\subset
\mathcal L^-\setminus \mathcal{E}_m^{\mathrm{simple}}$
and $f\in \mathcal{E}_m^{\mathrm{simple}}$ such that $f_k\to f$ with respect to
$\iota$. For each $k$, let $e_k\in K_m$ be a non-simple equilibrium of
$f_k$. By the compactness of $K_m$, after passing to a subsequence, we
may assume that $e_k\to e\in K_m$. Since $p_m(f_k-f)\to0$, the same
argument as above gives $f(e)=0$ and $Df_k(e_k)\to Df(e)$.
Since $\det Df_k(e_k)=0$ for every $k$, the continuity of the determinant
yields $\det Df(e)=0$.
Thus, $e$ is a non-simple equilibrium of $f$ in $K_m$, contradicting
$f\in \mathcal{E}_m^{\mathrm{simple}}$. Hence
$\mathcal{E}_m^{\mathrm{simple}}$ is open in $\mathcal L^-$.}
		
		To prove density, fix $f\in\mathcal L^-$ and define
\[
\Psi:\Omega\times(-1,1)^n\to\mathbb R^n,
\qquad \Psi(x,\lambda)=f(x)+\lambda.
\]
We apply Theorem~\ref{SS-2} with $\zeta=0$. Conditions (i), (ii),
and (iv) hold, and
\[
D\Psi(x,\lambda)(u,v)=Df(x)u+v
\]
is surjective, which verifies (iii). It follows that
\[
\Theta:=\{\lambda\in(-1,1)^n:
0\text{ is a regular value of }\Psi(\cdot,\lambda)\}
\]
is dense in $(-1,1)^n$. For $\lambda\in\Theta$, every equilibrium
of $\dot x=f(x)+\lambda$ is simple. Choose $\lambda_k\in\Theta$
with $\lambda_k\to0$. Constant perturbations leave the Jacobian
unchanged, so $f+\lambda_k\in\mathcal L^-$, while
$p_m(\lambda_k)\le\|\lambda_k\|\to0$ for every $m$.
Thus $f+\lambda_k\to f$ in $\iota$, proving density.
\end{proof}

	\subsection{Proof of Proposition \ref{G-H-O-M}}
	
	By Corollary~\ref{simple}, $\mathcal{O}_m^h=\mathcal{O}_m^{\mathrm{simple}}$, where
	\begin{align*}
		\mathcal O_m^{\mathrm{simple}}=\bigl\{f\in\mathcal E_m^h:\;&
\text{every periodic orbit of }\dot x=f(x)\text{ contained in }K_m\\
&\text{with minimal period at most }m\text{ is simple}\bigr\}.
\end{align*}
	It therefore suffices to prove the following proposition.
	\begin{prop}\label{G-H-O}
		For each $m\in \mathbb{N}\setminus \left \{ 0 \right \}$, the set $\mathcal{O}_m^{\mathrm{simple}}$ is a dense open subset of $\mathcal{L}^-$. In particular, 
		the set $\mathcal{O}^{h} =\bigcap_{m=1}^{\infty }\mathcal{O}_m^{\mathrm{simple}}$ is a generic subset of $\mathcal{L}^-$.
	\end{prop}

	 Define
	$$
	\mathcal C^\infty_{\mathrm{BF}}=\left\{f\in \mathcal C^1_{\mathrm{BF}}\mid f\in C^\infty\left(\Omega,\mathbb{R}^n\right)\right\}.
	$$
		For $\alpha>0$, we use $\mathcal E_\alpha^h$ and
$\mathcal E_\alpha^{\mathrm{simple}}$ for the sets defined by the same
conditions as above, with $K_m$ replaced by $K_\alpha$. The preceding
openness arguments apply without change. The proof uses three auxiliary lemmas.
	
	\begin{lem}\label{neighborhood }
    For $g\in\mathcal L^-$ and $\alpha>0$, denote by
\[
\mathcal E(g,\alpha)
=
\{e\in K_\alpha:g(e)=0\}
\]
the set of equilibria of $\dot x=g(x)$ in $K_\alpha$, and, for $r>0$, set
\[
\mathcal E(g,\alpha,r)
=
\bigcup_{e\in \mathcal E(g,\alpha)} B(e,r).
\]
		The following statements hold.
		\begin{enumerate}
			\item[{\rm(i)}] If $f\in\mathcal E_m^{\mathrm{simple}}$, then
there exists $\delta>0$ such that
$f\in\mathcal E_{m+\delta}^{\mathrm{simple}}$. The same assertion
holds with $\mathcal E^h$ in place of $\mathcal E^{\mathrm{simple}}$.
			
			\item[{\rm(ii)}] Let $f\in \mathcal{E}^{h}_{m}$ and $\delta >0$ be as in (i). Then there 
			exist $\varepsilon >0$ and an open neighborhood  $\mathcal{U}_1\subset\mathcal{E}^{h}_{m+\delta}$ of $f$  such that, for every $g\in \mathcal{U}_1$, every nonconstant periodic orbit $p(t)$ of $\dot x=g(x)$ satisfying $p(\mathbb{R})\subset K_{m+\delta}$ has minimal period greater than $\varepsilon$.
			
			\item[{\rm(iii)}] Let $f\in\mathcal E_m^h$, and let $\delta>0$ be as in (i).
For every $T>0$, there exist $r>0$ and a neighborhood
$\mathcal U_2$ of $f$ in $\mathcal E_{m+\delta}^h$ such that,
for every $g\in\mathcal U_2$,
\begin{enumerate}
\item[(a)]
$\mathcal E(g,m+\delta)\subset\mathcal E(f,m+\delta,r)$;

\item[(b)]
if $p(t)$ is a nonconstant periodic solution of $\dot x=g(x)$,
with minimal period $P\le T$ and
$p(\mathbb R)\subset K_{m+\delta}$, then $p(\mathbb R)\cap\mathcal E(f,m+\delta,r)=\emptyset$.
\end{enumerate}
		\end{enumerate}
	\end{lem}
	
	\begin{lem}\label{dense in L-}
		For every $f\in\mathcal L^-$ and every neighborhood $\mathcal V$ of $f$ in $(\mathcal L^-,\iota)$, there exist $g\in\mathcal V\cap\mathcal L^-\cap \mathcal C^\infty_{\mathrm{BF}}$ and $\kappa>0$ such that $\beta_i(g,x)\ge \kappa$ and $\gamma_i(g,x)\ge \kappa$ for all $x\in\Omega$ and $1\le i\le n$. 
	\end{lem}

\begin{lem}\label{surjective DF}
Let
$f\in \mathcal C^1_{\mathrm{BF}}\cap\mathcal L^-$,
and assume that there exists $\kappa_0>0$ such that
$\beta_i(f,x)\ge\kappa_0$ and $\gamma_i(f,x)\ge\kappa_0$
for all $x\in\Omega$ and $1\le i\le n$.
Suppose that $S_f(T)x_0=x_0$ for some $x_0\in\Omega$ and $T>0$,
with $f(x_0)\neq0$. Then there are finitely many perturbation directions
$\{h^\nu\}_{\nu\in \{1,\dots,N\}}\subset
\mathcal C^\infty_{\mathrm{BF}}$, with $h^\nu$ and $Dh^\nu$ bounded
on $\Omega$, such that, upon setting
\[
f_a:=f+\sum_{\nu=1}^N a_\nu h^\nu,\quad
F(t,y,a)
=
S_{f_a}(t)y-y,
\]
the map
\[
DF(T,x_0,0):
\mathbb R\times\mathbb R^n\times\mathbb R^N
\longrightarrow\mathbb R^n
\]
is surjective. Moreover, there is a neighborhood $B$ of $0$ in
$\mathbb R^N$ such that $f_a\in\mathcal L^-$ for every $a\in B$.
\end{lem}
	
	We first deduce Proposition~\ref{G-H-O} from these lemmas and then prove the lemmas.
	
	\begin{proof}[Proof of Proposition \ref{G-H-O}]
        We first prove that $\mathcal O_m^{\mathrm{simple}}$ is open in $\mathcal L^-$. Suppose, to the contrary, that there exist a sequence
$\{f_k\}_{k\in\mathbb N}\subset\mathcal L^-\setminus\mathcal O_m^{\mathrm{simple}}$
and $f\in\mathcal O_m^{\mathrm{simple}}$ such that $f_k$ converges to $f$ with respect to $\iota$. Since $\mathcal E_m^h$ is open, we may assume that $f_k\in\mathcal E_m^h$ for all $k$.

For each $k$, let $p_k(t)$ be a nonconstant periodic solution of
$\dot x=f_k(x)$ lying on a non-simple periodic orbit, with
$p_k(\mathbb R)\subset K_m$ and minimal period $P_k\le m$. By Lemma~\ref{neighborhood }(ii)--(iii), the periods $P_k$ are uniformly bounded away from zero, and these periodic orbits remain a uniform positive distance from the equilibria of $f$. Hence, after passing to a subsequence,
$P_k\to P_*\in(0,m]$ and $p_k(0)\to x_*\in K_m$, with $f(x_*)\neq0$.
Continuous dependence of the flow gives $S_f(P_*)x_*=x_*$ and
$S_f([0,P_*])x_*\subset K_m$. Thus $x_*$ lies on a nonconstant periodic orbit of $f$.

Set $M_k=DS_{f_k}(P_k)(p_k(0))$ and
$M_*=DS_f(P_*)(x_*)$. Continuous dependence for the flow and its variational equations gives $\|M_k-M_*\|\to0$. Let $\chi_k(\lambda)=\det(\lambda I-M_k)$ and
$\chi_*(\lambda)=\det(\lambda I-M_*)$. Since the orbit of $p_k$ is non-simple, $1$ is an eigenvalue of $M_k$ of algebraic multiplicity at least two, so $\chi_k(1)=\chi_k'(1)=0$. The coefficients of the characteristic polynomial depend continuously on the matrix entries; hence
\begin{equation}\label{chi}
    \chi_*(1)=\chi_*'(1)=0.
\end{equation}

Let $P_0$ be the minimal period of the limiting periodic orbit. Then
$P_*=rP_0$ for some $r\in\mathbb N\setminus\{0\}$, and $P_0\le m$.
Since $f\in\mathcal O_m^{\mathrm{simple}}$, this periodic orbit is hyperbolic. Thus, for
$M_0=DS_f(P_0)(x_*)$, the eigenvalue $1$ is algebraically simple, and
Corollary~\ref{simple} shows that $M_0$ has no other eigenvalue on the unit circle.

Since $M_*=M_0^r$, if $\lambda\in\sigma(M_0)$ satisfies $\lambda^r=1$, then $|\lambda|=1$, and hence $\lambda=1$ by Corollary~\ref{simple}. Therefore $1$ remains algebraically simple for $M_*=M_0^r$, contradicting (\ref{chi}). Hence $\mathcal O_m^{\mathrm{simple}}$ is open in $\mathcal L^-$.

We next prove that $\mathcal O_m^{\mathrm{simple}}$ is dense in $\mathcal L^-$. Let
$f\in\mathcal L^-$, and let $\mathcal V$ be an arbitrary open neighborhood of
$f$ in $(\mathcal L^-,\iota)$. By the density of $\mathcal E_m^h$, choose
$f_1\in\mathcal V\cap\mathcal E_m^h$. Since $\mathcal E_m^h$ is open,
there exists a neighborhood $\mathcal V_1$ of $f_1$ such that
$\mathcal V_1\subset\mathcal V\cap\mathcal E_m^h$. By Lemma~\ref{dense in L-}, we may
choose
$f_*\in\mathcal V_1\cap\mathcal L^-\cap\mathcal C^\infty_{\mathrm{BF}}$
and $\kappa_0>0$ such that
$\beta_i(f_*,x),\gamma_i(f_*,x)\ge\kappa_0$ for all $x\in\Omega$ and
$1\le i\le n$.

By Lemma~\ref{neighborhood }(i), choose $\delta>0$ such that
$f_*\in\mathcal E_{m+\delta}^h$. Consider
\[
\mathcal Z
=
\left\{
(T,x):
0<T\le m,\;
S_{f_*}(T)x=x,\;
f_*(x)\neq0,\;
S_{f_*}(t)x\in K_m
\text{ for all }t\in[0,T]
\right\}.
\]

The set $\mathcal Z$ is compact. By
Lemma~\ref{neighborhood }(ii), there exists $\varepsilon_0>0$ such that
every nonconstant periodic orbit under consideration has minimal period
greater than $\varepsilon_0$. Hence, with
$\varepsilon=\varepsilon_0/2$, every $(T,x)\in\mathcal Z$ satisfies
$T\in[\varepsilon,m]$, and thus
$\mathcal Z\subset[\varepsilon,m]\times K_m$. To see that $\mathcal Z$
is closed, let $(T_k,x_k)\in\mathcal Z$ with $T_k\to T_*$ and
$x_k\to x_*$. Continuous dependence of the flow gives
$S_{f_*}(T_*)x_*=x_*$, while the closedness of $K_m$ implies that
$S_{f_*}(t)x_*\in K_m$ for all $t\in[0,T_*]$. Moreover,
Lemma~\ref{neighborhood }(iii) prevents these periodic orbits from
approaching an equilibrium of $f_*$, so $f_*(x_*)\neq0$. Hence
$(T_*,x_*)\in\mathcal Z$, proving that $\mathcal Z$ is compact.

If $\mathcal Z=\varnothing$, then $f_*\in\mathcal O_m^{\mathrm{simple}}$, and there is
nothing to prove. Assume therefore that $\mathcal Z\neq\varnothing$.

Fix $z=(T,x)\in\mathcal Z$. By Lemma~\ref{surjective DF}, there exist finitely many
perturbation directions
$h_z^1,\ldots,h_z^{N_z}\in\mathcal C^\infty_{\mathrm{BF}}$ such that the
corresponding periodic-point map has surjective full derivative at
$(T,x,0)$. Since the derivative depends continuously on $(t,y,a)$ and
surjectivity is an open condition for finite-dimensional linear maps,
there exists a neighborhood $U_z$ of $z$ such that the same perturbation
directions yield surjectivity on $U_z\times\{0\}$.

The family $\{U_z:z\in\mathcal Z\}$ covers the compact set $\mathcal Z$.
Choose a finite subcover $U_{z_1},\ldots,U_{z_q}$, and collect all the
associated perturbation directions into a single finite family
$h^1,\ldots,h^N$. Define
\[
f_a
=
f_*+\sum_{\nu=1}^N a_\nu h^\nu,
\qquad
F(t,x,a)
=
S_{f_a}(t)x-x.
\]
For every $z\in\mathcal Z$, the derivative $DF(z,0)$ is surjective.
Indeed, $z$ belongs to some $U_{z_i}$, and the restriction of
$DF(z,0)$ to the time and state variables, together with the parameter
directions associated with $z_i$, is already surjective. Adding the
remaining parameter directions preserves surjectivity. Continuity of $DF$
and compactness of $\mathcal Z$ give an open neighborhood $\mathcal U$
of $\mathcal Z$ and a neighborhood $B$ of $0$ in $\mathbb R^N$ such that
$DF(t,x,a)$ is surjective for every
$(t,x,a)\in\mathcal U\times B$.

By Lemma~\ref{surjective DF}, each $h^\nu$ and $Dh^\nu$ is bounded on $\Omega$, and the
uniform inequalities
$\beta_i(f_*,x),\gamma_i(f_*,x)\ge\kappa_0$ hold on $\Omega$. Shrinking
$B$ if necessary, we may therefore assume that $f_a\in\mathcal L^-$
for every $a\in B$. Since $f_a$ converges to $f_*$ with respect to
$\iota$ as $a\to0$, after a further reduction of $B$ we may also assume
that
$f_a\in\mathcal V_1\subset\mathcal V\cap\mathcal E_m^h$ for every
$a\in B$.

We claim that, after shrinking $B$ once more, every nonconstant periodic
orbit of $f_a$ contained in $K_m$ and having minimal period at most $m$
determines a periodic point pair in $\mathcal U$. Otherwise, there
would exist $a_k\to0$ and periodic point pairs
$(P_k,x_k)\notin\mathcal U$, where $P_k\le m$ is the minimal period and
the corresponding orbit is contained in $K_m$. By
Lemma~\ref{neighborhood }(ii)--(iii), after shrinking $B$ so that all
$f_a$ lie in the corresponding neighborhoods of $f_*$, the periods
$P_k$ are uniformly bounded away from zero and the periodic orbits remain a uniform positive distance from the
equilibria of $f_*$. Passing to a subsequence,
$P_k\to P_*\in(0,m]$ and $x_k\to x_*\in K_m$. Continuous dependence on
the vector field, time, and initial condition gives
$S_{f_*}(P_*)x_*=x_*$. The limiting orbit remains in $K_m$, and
Lemma~\ref{neighborhood }(iii) gives $f_*(x_*)\neq0$. Hence
$(P_*,x_*)\in\mathcal Z\subset\mathcal U$, contradicting
$(P_k,x_k)\notin\mathcal U$ for all $k$.

We apply Theorem~\ref{SS-2} to $F$ on $\mathcal U\times B$. Since $f_*$
and $h^1,\ldots,h^N$ are smooth, $F$ has the required regularity, and
its full derivative is surjective on $\mathcal U\times B$. Hence there
exists $a\in B$, arbitrarily close to $0$, such that $0$ is a regular
value of
$F_a(t,x):=F(t,x,a)$.

Fix such an $a$, and let $x$ belong to a periodic orbit of $f_a$
contained in $K_m$, with minimal period $P\le m$. Set
$M=DS_{f_a}(P)(x)$ and $v=f_a(x)\neq0$. Since $(P,x)\in\mathcal U$,
the regular-value property implies that
\[
(\tau,u)\mapsto \tau v+(M-I)u
\]
is surjective. Hence
\begin{equation}\label{range M}
\operatorname{Ran}(M-I)+\operatorname{span}\{v\}
=
\mathbb R^n.
\end{equation}
The flow identity
$DS_{f_a}(P)(x)f_a(x)=f_a(S_{f_a}(P)x)$ gives $(M-I)v=0$. Therefore
$\operatorname{rank}(M-I)\le n-1$, whereas (\ref{range M}) gives
$\operatorname{rank}(M-I)\ge n-1$. Thus
$\operatorname{rank}(M-I)=n-1$ and
$\ker(M-I)=\operatorname{span}\{v\}$.

If $1$ had algebraic multiplicity greater than one, then, since its
geometric multiplicity is one, the Jordan form of $M$ would contain a
single Jordan block for the eigenvalue $1$, of size at least two. Hence
there would exist $u\in\mathbb R^n$ and $c\neq0$ such that
$(M-I)u=cv$. Thus $v\in\operatorname{Ran}(M-I)$, and
(\ref{range M}) would imply
$\operatorname{Ran}(M-I)=\mathbb R^n$, contradicting
$\operatorname{rank}(M-I)=n-1$. Therefore $1$ is algebraically simple.
By Corollary~\ref{simple}, the periodic orbit is hyperbolic.

Thus every periodic orbit of $f_a$ contained in $K_m$ with minimal
period at most $m$ is hyperbolic. Since $f_a\in\mathcal E_m^h$, we
conclude that $f_a\in\mathcal O_m^{\mathrm{simple}}$. Moreover, $f_a\in\mathcal V$.
Hence $\mathcal V\cap\mathcal O_m^{\mathrm{simple}}\neq\varnothing$, proving that
$\mathcal O_m^{\mathrm{simple}}$ is dense in $\mathcal L^-$.
	\end{proof}
	
	We now prove the three auxiliary lemmas.
	
	\begin{proof}[Proof of Lemma \ref{neighborhood }]
(i) We prove the simple case; the hyperbolic case is identical.
Let $e_1,\ldots,e_N$ be the equilibria of $f$ in $K_m$.
Their simplicity and the compactness of the zero set in $K_m$
imply that there are only finitely many.
Choose pairwise disjoint neighborhoods $U_1,\ldots,U_N$ such that
$e_j$ is the only equilibrium of $f$ in $U_j$. If the conclusion were false, there would be
$\delta_k\to 0$ and equilibria
\[
z_k\in K_{m+\delta_k}\setminus\bigcup_{j=1}^N U_j.
\]
For sufficiently large $k$, $z_k\in K_{m+1}$. Passing to a subsequence,
$z_k\to z$. Since
$\|z_k\|\le m+\delta_k$ and
$\operatorname{dist}(z_k,\Omega^c)\ge (m+\delta_k)^{-1}$,
we have $z\in K_m$. Moreover, $f(z_k)=0$ implies $f(z)=0$.
Hence $z=e_j$ for some $j$, and therefore $z_k\in U_j$ for all
sufficiently large $k$, a contradiction. Thus
$f\in\mathcal E_{m+\delta}^{\mathrm{simple}}$ for some $\delta>0$.
The hyperbolic case follows in the same way.

(ii) If $K_{m+\delta}=\varnothing$, both (ii) and (iii) are immediate.
We therefore assume that $K:=K_{m+\delta}\ne\varnothing$ and set
$C=\operatorname{co}(K)$.
Since $\Omega$ is convex, $C$ is a compact subset of $\Omega$.
Choose $M$ such that $C\subset K_M$, and set
$L:=\sup_{x\in C}\|Df(x)\|+1$.
Shrinking a neighborhood $\mathcal U_1$ of $f$ in
$\mathcal E_{m+\delta}^h$, we may assume that
$p_M(g-f)<1$ for every $g\in\mathcal U_1$.
Hence $\sup_{x\in C}\|Dg(x)\|\le L$, and the mean value formula gives
\[
\|g(x)-g(y)\|\le L\|x-y\|,
\qquad x,y\in K.
\]

Let $p(t)$ be a nonconstant periodic solution of $\dot x=g(x)$,
where $g\in\mathcal U_1$, $p(\mathbb R)\subset K$, and $P$ is its
minimal period. Set $D=\operatorname{diam}p(\mathbb R)>0$.
Since $p(P)=p(0)$, we have
$P^{-1}\int_0^P g(p(s))\,ds=0$, and hence
$\|p'(t)\|\le LD$ for every $t$. Choose $0\le t_1<t_2<P$ such that
$D=\|p(t_2)-p(t_1)\|$. Using $p(t_1+P)=p(t_1)$,
\[
\begin{aligned}
D\le
\min\left\{
\int_{t_1}^{t_2}\|p'(t)\|\,dt,\,
\int_{t_2}^{t_1+P}\|p'(t)\|\,dt
\right\}\le
\min\{t_2-t_1,\;P-(t_2-t_1)\}
\sup_t\|p'(t)\| \le \frac{P}{2}LD.
\end{aligned}
\]
Thus $P\ge2/L$ if $L>0$; if $L=0$, no nonconstant periodic solution
exists. Taking, for instance, $\varepsilon=(1+L)^{-1}$ proves (ii).

(iii) Let
$\mathcal E(f,m+\delta)=\{e_1,\ldots,e_N\}$.
Since these equilibria are hyperbolic, the implicit function theorem,
together with compactness of $K_{m+\delta}$, implies that for every
sufficiently small $r>0$, there exists a neighborhood of $f$ such that $\mathcal E(g,m+\delta)
\subset
\mathcal E(f,m+\delta,r)$.
It remains to prove assertion (b).

Let $\varepsilon>0$ be as in (ii). If $T\le\varepsilon$, assertion (b)
is immediate. Assume $T>\varepsilon$. Suppose, to the contrary, that
(b) fails. Then there exist a sequence $\{g_k\}_{k\in\mathbb N}$
converging to $f$ with respect to $\iota$ and nonconstant periodic
solutions $p_k(t)$ of $\dot x=g_k(x)$, with minimal periods
$P_k\le T$ and $p_k(\mathbb R)\subset K_{m+\delta}$, such that, after
translating $p_k$ in time and passing to a subsequence,
$p_k(0)\to e$ for some $e\in\mathcal E(f,m+\delta)$.
By (ii), $P_k>\varepsilon$ for all sufficiently large $k$; hence we may
further assume that $P_k\to P_*\in[\varepsilon,T]$.

Let
$\eta_k=\sup_{x\in K}\|g_k(x)-f(x)\|$.
Since $g_k\to f$, $\eta_k\to0$. Using $f(e)=0$ and the Lipschitz bound
for $f$ on $C$, Gronwall's inequality yields
\[
\sup_{0\le t\le T}\|p_k(t)-e\|
\le
\left(
\|p_k(0)-e\|+T\eta_k
\right)e^{LT}
\to 0.
\]
Together with the $C^1$-convergence of $g_k$ to $f$, this gives $\sup_{0\le t\le T}
\|Dg_k(p_k(t))-Df(e)\|
\to 0$.

Let $Y_k(t)=DS_{g_k}(t)(p_k(0))$. Then $Y_k(0)=I$ and $\dot Y_k(t)=Dg_k(p_k(t))Y_k(t)$.
Subtracting its integral equation from that of $e^{tDf(e)}$ and
applying Gronwall's inequality gives $\sup_{0\le t\le T}
\|Y_k(t)-e^{tDf(e)}\|
\to 0.$
Since $P_k\to P_*$, it follows that
\[
DS_{g_k}(P_k)(p_k(0))
\to e^{P_*Df(e)}.
\]

Since $p_k$ is nonconstant, $g_k(p_k(0))\neq0$. Define $v_k=
{g_k(p_k(0))}/
{\|g_k(p_k(0))\|}$.
Differentiating the flow identity
$S_g(t)S_g(s)x=S_g(t+s)x$ with respect to $s$ at $s=0$ yields $DS_g(t)(x)g(x)=g(S_g(t)x)$.
Since $S_{g_k}(P_k)p_k(0)=p_k(0)$, we obtain
$DS_{g_k}(P_k)(p_k(0))v_k=v_k$.
Passing to a subsequence, $v_k\to v$ with $\|v\|=1$, and hence $e^{P_*Df(e)}v=v$.
By the spectral mapping theorem, there exists
$\lambda\in\sigma(Df(e))$ such that $e^{P_*\lambda}=1$.
Therefore $\operatorname{Re}\lambda=0$, contradicting the
hyperbolicity of $e$.

Thus there exist $\rho>0$ and a neighborhood $\mathcal U$ of $f$ such
that every periodic solution in (b) is disjoint from
$\mathcal E(f,m+\delta,\rho)$. Choose $0<r<\rho$, and shrink
$\mathcal U$ if necessary so that assertion (a) also holds with this
$r$. This proves (iii).
\end{proof}
	
	\begin{proof}[Proof of Lemma \ref{dense in L-}]
		{Fix $f\in\mathcal L^-$ and a neighborhood $\mathcal V$ of $f$.
There exist $m\in\mathbb N$ and $\delta>0$ such that
$p_m(h-f)<\delta$ implies $h\in\mathcal V$.

Choose a constant matrix $B$ whose allowed neighboring entries have the
signs prescribed in (\ref{transformation}) and absolute value one, while all other
off-diagonal entries vanish, and set $Q(x)=Bx$. Then
$Q\in \mathcal C^\infty_{\mathrm{BF}}$ and
$\beta_i(Q,x)=\gamma_i(Q,x)=1$ for all $x\in\Omega$ and $1\le i\le n$.
For $\eta>0$, set $f^\eta=f+\eta Q$. Since $f\in\mathcal L^-$,
\[
\beta_i(f^\eta,x)\ge\eta,\qquad
\gamma_i(f^\eta,x)\ge\eta .
\]
Moreover, $p_m(f^\eta-f)=\eta p_m(Q)$. Choose $\eta>0$ sufficiently
small so that $p_m(f^\eta-f)<\delta/2$.

For $1\le i\le n$, let
$\pi_i(x)=(x_{i-1},x_i,x_{i+1})$ and $\Omega_i=\pi_i(\Omega)$.
Since $f^\eta\in \mathcal C^1_{\mathrm{BF}}$, there exists
$F_i^\eta\in C^1(\Omega_i)$ such that
$f_i^\eta=F_i^\eta\circ\pi_i$.

Choose a countable locally finite smooth partition of unity
$\{\theta_\nu\}$ on $\Omega_i$ such that
$\operatorname{supp}\theta_\nu\Subset\Omega_i$. Let $\varepsilon>0$, to be specified below.
Mollification on compact subsets gives, for each $\nu$, a smooth function $F_{i,\nu}$, defined on a neighborhood of
$\operatorname{supp}\theta_\nu$, such that
\[
\|F_{i,\nu}-F_i^\eta\|_{C^1(\operatorname{supp}\theta_\nu)}
<
\frac{\varepsilon 2^{-\nu-2}}
{1+\|\theta_\nu\|_{C^1}} .
\]
Set $G_i=\sum_\nu\theta_\nu F_{i,\nu}$. The sum is locally finite, so
$G_i\in C^\infty(\Omega_i)$. Since $\sum_\nu\theta_\nu=1$,
\[
G_i-F_i^\eta
=
\sum_\nu\theta_\nu(F_{i,\nu}-F_i^\eta),
\]
and
\[
D(G_i-F_i^\eta)
=
\sum_\nu D\theta_\nu(F_{i,\nu}-F_i^\eta)
+
\sum_\nu\theta_\nu(DF_{i,\nu}-DF_i^\eta).
\]
Since $\sum_{\nu\ge1}2^{-\nu}=1$, these identities give
\begin{equation}\label{supsup}
\sup_{\Omega_i}|G_i-F_i^\eta|
\le\frac{\varepsilon}{4},
\qquad
\sup_{\Omega_i}\|DG_i-DF_i^\eta\|
\le\frac{\varepsilon}{2}.
\end{equation}

Define $g_i=G_i\circ\pi_i$. Then
$g\in \mathcal C^\infty_{\mathrm{BF}}$, $|\beta_i(g,x)-\beta_i(f^\eta,x)|<\varepsilon$ and $|\gamma_i(g,x)-\gamma_i(f^\eta,x)|<\varepsilon$. 
Moreover, $p_m(g-f^\eta)\le 2\sqrt n\,\varepsilon $.
Choose
\[
0<\varepsilon<
\min\left\{\frac{\eta}{2},
\frac{\delta}{4\sqrt n}\right\}.
\]
Then
$\beta_i(g,x),\gamma_i(g,x)\ge\eta/2$ for all $x\in\Omega$ and $i$.
Hence $g\in\mathcal L^-$ with $\kappa=\eta/2$. Furthermore,
\[
p_m(g-f)
\le
p_m(g-f^\eta)+p_m(f^\eta-f)
<
\frac{\delta}{2}+\frac{\delta}{2}
=
\delta.
\]
Thus $g\in\mathcal V$, completing the proof.}
	\end{proof}

	\begin{proof}[Proof of Lemma \ref{surjective DF}]
		Let $\gamma(t)=S_f(t)x_0$, and let $P$ be its minimal period.
Since $f(x_0)\neq0$ and $S_f(T)x_0=x_0$, we have
$T=rP$ for some $r\in\mathbb N\setminus\{0\}$.For this proof, write \(S_\gamma(t,s):=S_{Df(\gamma)}(t,s)\),
so that the terminal time is written first.
Let $M=S_\gamma(P,0)$. By periodicity,
$S_\gamma(T,0)=M^r$.

Let $q_r(z)=1+z+\cdots+z^{r-1}$.
By Corollary~\ref{simple}, $1$ is the only eigenvalue of $M$
on the unit circle. Since the zeros of $q_r$ are the nontrivial
$r$-th roots of unity, $q_r(M)$ is invertible. Therefore, from
$M^r-I=(M-I)q_r(M)$,
\begin{equation}\label{ker*}
\ker\bigl((M^r)^*-I\bigr)=\ker(M^*-I).
\end{equation}
For
$h\in\mathcal C^\infty_{\mathrm{BF}}$, define
\[
J_T h
:=
\left.
\frac{d}{d\varepsilon}
S_{f+\varepsilon h}(T)x_0
\right|_{\varepsilon=0}
=
\int_0^T
S_\gamma(T,s)h(\gamma(s))\,ds,
\]
where the integral formula follows by variation of constants in the
inhomogeneous variational equation.

Set
$W=\ker((M^r)^*-I)$.
We first show that, for every $0\neq c\in W$, there exists
$h\in\mathcal C^\infty_{\mathrm{BF}}$
such that $\langle c,J_T h\rangle\neq0$.

By (\ref{ker*}), $M^*c=c$.
Define
$\varphi(s)= S_\gamma(P,s)^*c$ for $0\le s\le P$.
Let $A(s)=Df(\gamma(s))$. The identity $\partial_s S_\gamma(P,s)
=
-S_\gamma(P,s)A(s)$
gives
$\dot\varphi(s)=-A(s)^*\varphi(s)$.
Hence $\varphi$ solves the adjoint equation.
Since $\varphi(P)=c=\varphi(0)$ and $A(s)$ is $P$-periodic,
$\varphi$ extends to a nonzero $P$-periodic solution.
Moreover,
$ S_\gamma(T,s)^*c=\varphi(s)$ for $0\le s\le T$,
with $\varphi$ understood periodically.

Choose $j$ and $s_0\in(0,P)$ such that
$\varphi_j(s_0)\neq0$.
There exist $0<a<s_0<b<P$ such that $\varphi_j$
has a fixed strict sign on $I=(a,b)$.
Set
$q_j(s)=(\gamma_j(s),\gamma_{j+1}(s))$ for $0\le s\le P$,
with cyclic indices. By Corollary~\ref{periodic solution of system}, $q_j$ is injective on $(0,P]$,
while periodicity gives $q_j(0)=q_j(P)$. Since $s_0\in I=(a,b)$, $q_j(s_0)\notin q_j\bigl([0,a]\cup[b,P]\bigr)$.
Indeed, if $q_j(s)=q_j(s_0)$ for some
$s\in[0,a]\cup[b,P]$, then the injectivity of $q_j$ on $(0,P]$ is
contradicted; when $s=0$, use $q_j(0)=q_j(P)$.

The set $q_j([0,a]\cup[b,P])$ is compact, so
\[
d_0
:=
\operatorname{dist}
\left(
q_j(s_0),
q_j\bigl([0,a]\cup[b,P]\bigr)
\right)
>0.
\]
Set $\rho=d_0/2$. If $s\in[0,P]$ and
$q_j(s)\in B(q_j(s_0),\rho)$, then $s\in I$. Let $q_0=q_j(s_0)$. Choose
$b_0\in\mathcal C_c^\infty(\mathbb R^2)$ such that
$b_0\ge0$, $b_0(q_0)>0$, and
$\operatorname{supp}b_0\subset B(q_0,\rho)$.
For instance, one may take
\[
b_0(z)
=
\psi\!\left(
1-\frac{4|z-q_0|^2}{\rho^2}
\right),
\qquad
\psi(r)
=
\begin{cases}
e^{-1/r}, & r>0,\\
0, & r\le0.
\end{cases}
\]
Then $b_0(q_0)=e^{-1}>0$. Moreover, $b_0(z)\neq0$ only if
$|z-q_0|<\rho/2$, and hence
\[
\operatorname{supp}b_0
\subset
\overline{B(q_0,\rho/2)}
\subset
B(q_0,\rho).
\]
Define $h=(h_1,\ldots,h_n)$ by
$h_j(x)=b_0(x_j,x_{j+1})$ and
$h_i(x)=0$ for $i\neq j$.
Then
$h\in\mathcal C^\infty_{\mathrm{BF}}$.
Consequently, if $s\in[0,P]$ and
$b_0(q_j(s))\neq0$, then $s\in I$.
Hence
\begin{equation}\label{cJ}
    \langle c,J_T h\rangle
=
r\int_0^P
\varphi_j(s)b_0(q_j(s))\,ds
\neq0.
\end{equation}

For each such $h$, define
$\ell_h\in W^*$ by
$\ell_h(c)=\langle c,J_T h\rangle$.
By \eqref{cJ}, the functionals $\ell_h$ separate points of $W$. Since $W$ is finite-dimensional, these functionals span $W^*$.
Hence there exist
$h^1,\ldots,h^N\in
\mathcal C^\infty_{\mathrm{BF}}$
such that
\[
\bigcap_{\nu=1}^N
\left\{
c\in W:
\langle c,J_T h^\nu\rangle=0
\right\}
=
\{0\}.
\]
Since
$W=\operatorname{Ran}(M^r-I)^\perp$,
it follows that
\begin{equation}\label{RAN}
\operatorname{Ran}(M^r-I)
+
\operatorname{span}
\{J_T h^1,\ldots,J_T h^N\}
=
\mathbb R^n.
\end{equation}
Set
$f_a=f+\sum_{\nu=1}^N a_\nu h^\nu$.
The $C^1$ dependence of solutions on initial data and parameters implies
that $F(t,y,a)=S_{f_a}(t)y-y$ is $C^1$ near $(T,x_0,0)$.
Direct differentiation gives
\[
DF(T,x_0,0)(\tau,u,\eta)
=
\tau f(x_0)
+
(M^r-I)u
+
\sum_{\nu=1}^N
\eta_\nu J_T h^\nu.
\]
Thus (\ref{RAN}) implies that $DF(T,x_0,0)$ is surjective.

Finally, each $h^\nu$ is constructed from a compactly supported smooth
function of two neighboring variables. Hence its signed neighboring
derivatives are globally bounded. Thus there exists $C_\nu>0$ such that
$|\beta_i(h^\nu,x)|\le C_\nu$ and
$|\gamma_i(h^\nu,x)|\le C_\nu$
for all $x$ and $i$. Consequently,
\[
\min\{\beta_i(f_a,x),\ \gamma_i(f_a,x)\}
\ge
\kappa_0-\sum_{\nu=1}^N |a_\nu|C_\nu.
\]
The right-hand side is positive for sufficiently small $\|a\|$,
so $f_a\in\mathcal L^-$.
	\end{proof}

The perturbations in Lemma~\ref{surjective DF} are localized in an
adjacent-coordinate plane. Figure~\ref{fig:projected-perturbations}
illustrates this construction and its analogue for connecting orbits,
which will be used in Section~\ref{KSD}.

\begin{figure}[htbp]
\centering
\begin{tikzpicture}[x=1cm,y=1cm,font=\small,
  selectedarc/.style={draw=diagramblue,line width=1pt,
    postaction={decorate},decoration={markings,
    mark=at position .22 with {\arrow{Stealth}},
    mark=at position .78 with {\arrow{Stealth}}}}]
  \begin{scope}
    \node at (3.25,4.85) {(a) Periodic orbit: one minimal period};
    \draw[-{Stealth},black!45] (.35,.3)--(6.3,.3) node[below] {$x_j$};
    \draw[-{Stealth},black!45] (.35,.3)--(.35,4.4) node[above] {$x_{j+1}$};
    \filldraw[fill=diagramred!4,draw=diagramred,dashed]
      (4.9,2.4) circle[radius=.65];
    \filldraw[fill=diagramred!20,draw=diagramred!70]
      (4.9,2.4) circle[radius=.325];
    \draw[draw=black!55,line width=.8pt]
      (4.1,3.6) .. controls (2.6,4.35) and (.9,3.6) .. (1.1,2.4)
      .. controls (.9,1.2) and (2.6,.45) .. (4.1,1.2);
    \draw[selectedarc]
      (4.1,1.2) .. controls (5.166667,1.73) and (5.166667,3.07) .. (4.1,3.6);
    \fill (4.1,1.2) circle[radius=1.2pt] node[below left] {$q_j(a)$};
    \fill (4.1,3.6) circle[radius=1.2pt] node[above left] {$q_j(b)$};
    \fill (4.9,2.4) circle[radius=1.4pt];
    \node[anchor=east] at (4.38,2.38) {$q_j(s_0)$};
    \fill (1.1,2.4) circle[radius=1.2pt];
    \node[align=center] at (2.35,2.6) {$q_j(0)$\\$=q_j(P)$};
    \draw[thin,black!55] (1.17,2.4)--(1.65,2.55);
    \node[anchor=west,text=diagramred,font=\footnotesize] at (5.15,1.22) {$\operatorname{supp}b_0$};
    \draw[thin,diagramred] (5.16,1.47)--(5.01,2.16);
    \node[text=diagramred] at (5.55,3.35) {$B(q_j(s_0),\rho)$};
    \draw[thin,diagramred] (5.42,3.1)--(5.24,2.95);
    \node at (3.35,-.35) {$h_j(x)=b_0(x_j,x_{j+1}),\quad h_i=0\ (i\ne j)$};
  \end{scope}
  \begin{scope}[xshift=7.65cm]
    \node at (3.15,4.85) {(b) Connecting orbit};
    \draw[-{Stealth},black!45] (.35,.3)--(6.3,.3) node[below] {$x_j$};
    \draw[-{Stealth},black!45] (.35,.3)--(.35,4.4) node[above] {$x_{j+1}$};
    \filldraw[fill=diagramred!4,draw=diagramred,dashed]
      (3.45,2.4) circle[radius=.65];
    \filldraw[fill=diagramred!20,draw=diagramred!70]
      (3.45,2.4) circle[radius=.325];
    \draw[draw=black!55,line width=.8pt]
      (1.05,1.15) .. controls (1.6,.82) and (1.85,1.15) .. (2.25,1.55);
    \draw[selectedarc]
      (2.25,1.55) .. controls (2.7,2.05) and (2.9,2.28) .. (3.45,2.4)
      .. controls (3.95,2.51) and (4.25,2.45) .. (4.6,3.15);
    \draw[draw=black!55,line width=.8pt]
      (4.6,3.15) .. controls (4.95,3.9) and (5.45,3.8) .. (5.8,3.45);
    \filldraw[fill=white,draw=black] (1.05,1.15) circle[radius=2pt] node[above left] {$q_j^-$};
    \filldraw[fill=white,draw=black] (5.8,3.45) circle[radius=2pt] node[right] {$q_j^+$};
    \fill (2.25,1.55) circle[radius=1.2pt] node[below right] {$q_j(s_1)$};
    \fill (4.6,3.15) circle[radius=1.2pt] node[above left] {$q_j(s_2)$};
    \fill (3.45,2.4) circle[radius=1.4pt];
    \node at (2.62,3.15) {$q_j(s_0)$};
    \draw[thin,black!55] (2.95,2.95)--(3.37,2.47);
    \node[anchor=west,text=diagramred,font=\footnotesize] at (4.15,1.25) {$\operatorname{supp}b$};
    \draw[thin,diagramred] (4.2,1.48)--(3.63,2.2);
    \node[text=diagramred] at (2.8,4.12) {$B(q_j(s_0),d/2)$};
    \draw[thin,diagramred] (3.13,3.92)--(3.43,3.06);
    \node at (3.35,-.35) {$h_j(x)=b(x_j,x_{j+1}),\quad h_i=0\ (i\ne j)$};
  \end{scope}
\end{tikzpicture}
\caption{Localization of a perturbation in an adjacent-coordinate
projection. Blue arcs mark the chosen time windows, and the darker
disks mark the supports of the planar bump functions. Dashed circles
bound larger neighborhoods. In (a), the outer
disk avoids $q_j([0,a]\cup[b,P])$, where $P$ is the minimal period.
In (b), it avoids the two projected tails and their limiting points;
the limiting points are drawn separately for clarity, but need not
be distinct. Only times in the chosen window can meet the support
(within one period in (a)); no claim is made about the number of
visits inside that window. The lifted perturbations depend on two
coordinates and need not have compact support in $\Omega$.}
\label{fig:projected-perturbations}
\end{figure}
\FloatBarrier
To complete the proof of Theorem~\ref{critical elements is generic}, set
\[
\mathcal O:=\bigcap_{m=1}^{\infty}\mathcal O_m^h.
\]
By Propositions~\ref{Baire} and~\ref{G-H-O-M}, $\mathcal O$ is a dense
$G_\delta$ subset of $\mathcal L^-$. Every equilibrium belongs to some
$K_m$, and every periodic orbit is contained in some $K_m$ with its
minimal period at most $m$. Thus all critical elements of every
$f\in\mathcal O$ are hyperbolic.
	
	\section{Generic Kupka--Smale property}\label{GKS}
	{For any nonconstant connection between hyperbolic equilibria,
Proposition~\ref{Hyperbolic critical elements prop}(iii)--(iv) gives
\[
i(e^-)\ge2h^--1\ge2h^+-1\ge i(e^+).
\]
If the two indices differ, or if their common value is even, an even
integer lies between them, and Theorem~\ref{Automatic Transversality}
gives transversality. The proof of
Theorem~\ref{Generic non-existence of homoindexed connecting orbits}
therefore reduces to excluding, generically, connections between hyperbolic
equilibria $e^-$ and $e^+$ with the same odd Morse index:
\[
i(e^+)=i(e^-)=2h-1,
\qquad
h\in\left\{1,\ldots,\frac{\widetilde n+1}{2}\right\}.
\]

Throughout this section, connecting orbits are understood to be nonconstant.
If $q$ lies on such an orbit of $\dot x=f(x)$, then
$0\neq f(q)\in T_qW^u(e^-)\cap T_qW^s(e^+)$.
Since $\dim W^u(e^-)+\dim W^s(e^+)=n$
and the tangent spaces share the nonzero vector $f(q)$, the manifolds
cannot be transverse at $q$. For this pair of equilibria, transversality
is therefore equivalent to the absence of nonconstant connecting orbits,
or, equivalently, to
\[
W^u(e^-)\cap W^s(e^+)\subset \{e^-\}\cap\{e^+\}.
\]
For $e^-=e^+=e$, the equilibrium itself is not excluded, since
$T_eW^u(e)\oplus T_eW^s(e)=\mathbb R^n$.
We shall prove that this condition holds on a residual set of vector fields.}
	
	\subsection{Reduction to a local genericity statement}
	Let $(\mathcal L^-,\iota)$ be as defined in Section~\ref{GTC}. Recall from Proposition~\ref{Baire} that $\mathcal L^-$ is a
second-countable Baire space, and from Proposition~\ref{G-H-P} that
$\mathcal E^h_{m+1}$ is open and dense in $\mathcal L^-$. The compact
exhaustion satisfies $K_m\subset \operatorname{Int}K_{m+1}$ and
$\Omega=\bigcup_m \operatorname{Int}K_m$, after any initial empty terms
have been omitted.
	\begin{prop}\label{Reduction}
		Assume that, for every $m\in\mathbb N\setminus\{0\}$ and every
$f_0\in\mathcal E^h_{m+1}$, there exist an open neighborhood
$U(f_0,m)\subset\mathcal E^h_{m+1}$ of $f_0$ in $\mathcal L^-$ and a
residual subset $G(f_0,m)$ of $U(f_0,m)$ such that, for every
$f\in G(f_0,m)$, the system $\dot x=f(x)$ admits no nonconstant
connecting orbit contained in $K_m$ between two equilibria of the
same odd Morse index. Then Theorem~\ref{Generic non-existence of homoindexed connecting orbits} holds.
	\end{prop}
	
	\begin{proof}
Fix $m\ge1$. By second countability, the open cover
$\{U(f_0,m):f_0\in\mathcal E^h_{m+1}\}$ of $\mathcal E^h_{m+1}$
admits a countable subcover $\{U_j\}_{j\ge1}$. Let $G_j$ denote the
corresponding residual subset of $U_j$. Since $U_j$ is open in
$\mathcal E^h_{m+1}$, the set $U_j\setminus G_j$ is meagre in
$\mathcal E^h_{m+1}$. Hence
\[
H_m:=\mathcal E^h_{m+1}\setminus
\bigcup_{j\ge1}(U_j\setminus G_j)
\]
is residual in $\mathcal E^h_{m+1}$, and therefore also residual in
$\mathcal L^-$, since $\mathcal E^h_{m+1}$ is open and dense in
$\mathcal L^-$.

Set
\[
\widetilde{\mathcal O}
:=\mathcal O\cap\bigcap_{m\ge1}H_m,
\]
where $\mathcal O$ is the residual set given by Theorem~\ref{critical elements is generic}. Then
$\widetilde{\mathcal O}$ is residual in $\mathcal L^-$.

Let $f\in\widetilde{\mathcal O}$. The closure of every connecting
orbit between two equilibria is compact in $\Omega$, and hence is
contained in some $K_m$. By the definition of $H_m$, there is no
nonconstant connecting orbit between equilibria of the same odd
Morse index. Together with Theorem~\ref{Automatic Transversality}, this shows that $f$ is
Kupka--Smale.

To complete the proof of Theorem~\ref{Generic non-existence of homoindexed connecting orbits},
we exclude homoclinic orbits to periodic orbits; those to equilibria have
already been excluded. If a hyperbolic periodic orbit admitted a homoclinic orbit,
it would be transverse by Theorem~\ref{Automatic Transversality}. By the Birkhoff--Smale
homoclinic theorem (see \cite[Theorem B]{Smale-65} and
\cite[Theorem 26.0.6]{SW}), the associated Poincar\'e map would
contain a horseshoe, contradicting the Poincar\'e--Bendixson property
in Proposition~\ref{PB}. Thus no homoclinic orbit exists, and
Theorem~\ref{Generic non-existence of homoindexed connecting orbits} follows.
	\end{proof}
	
	The next lemma supplies the neighborhood $U(f_0,m)$ required by
Proposition~\ref{Reduction}, together with uniform estimates near the equilibria.

	\begin{lem}\label{lem:local-uniform-hyperbolicity}
Fix $m\geq 1$ and $f_0\in\mathcal E^h_{m+1}$, and let
$e_1,\ldots,e_k$ be all the equilibria of $f_0$ in $K_{m+1}$.
Then there exist bounded open sets $V_0,V_1$, numbers
$0<r_0<R_0<R_1$, and an open neighborhood
$U(f_0,m)\subset\mathcal E^h_{m+1}$ of $f_0$ in
$(\mathcal L^-,\iota)$ such that
\[
K_m\subset V_0\Subset V_1\Subset \operatorname{Int}K_{m+1},
\]
and the following properties hold.

\begin{enumerate}
\item [(i)]
For every $f\in U(f_0,m)$ and $1\leq j\leq k$, the system
$\dot x=f(x)$ has a unique equilibrium $e_j(f)$ in $B(e_j,R_1)$,
with $e_j(f)\in B(e_j,r_0)$ and $i(e_j(f))=i(e_j)$.
Moreover, if $x\in\overline{V_1}$ is an equilibrium of $f$, then
$x=e_j(f)$ for some $j\in\{1,\ldots,k\}$.

\item [(ii)]
The closed balls $\overline{B(e_j,R_1)}$, $1\leq j\leq k$, are
pairwise disjoint and contained in $\Omega$.

\item [(iii)]
For every $f\in U(f_0,m)$ and $1\leq j\leq k$, there exists an open
neighborhood $V_j(f)$ of $e_j(f)$ such that
\[
B(e_j,R_0)\Subset V_j(f)\Subset B(e_j,R_1).
\]
The local stable and unstable manifolds of $e_j(f)$ in $V_j(f)$,
denoted by $W^s_{\mathrm{loc}}(e_j(f))$ and
$W^u_{\mathrm{loc}}(e_j(f))$, are $C^1$ embedded submanifolds of
dimensions $n-i(e_j)$ and $i(e_j)$, respectively, and
\[
W^s_{\mathrm{loc}}(e_j(f))
\cap W^u_{\mathrm{loc}}(e_j(f))
=\{e_j(f)\}.
\]

\item [(iv)]
If a solution $u$ of $\dot x=f(x)$ satisfies
$u(t)\in B(e_j,R_0)$ for every $t\geq t_0$, then
$u(t)\in W^s_{\mathrm{loc}}(e_j(f))$ for every $t\geq t_0$.
If $u(t)\in B(e_j,R_0)$ for every $t\leq t_0$, then
$u(t)\in W^u_{\mathrm{loc}}(e_j(f))$ for every $t\leq t_0$.
{Moreover, there exist constants $C,\eta>0$, independent of
$f\in U(f_0,m)$, $j$, and $t_0$, such that
\[
\|u(t)-e_j(f)\|\leq Ce^{-\eta(t-t_0)},\qquad t\geq t_0,
\]
in the first case, and
\[
\|u(t)-e_j(f)\|\leq Ce^{-\eta(t_0-t)},\qquad t\leq t_0,
\]
in the second.}
\end{enumerate}
\end{lem}

The proof is given in Appendix~\ref{subsec:uniform-local-estimates}.

\begin{rk}\label{rem:local-neighborhood}
The neighborhood $U(f_0,m)$ in Lemma~\ref{lem:local-uniform-hyperbolicity}
is open in the topology $\iota$ on $\mathcal L^-$. The construction
uses only this topology. A Banach parameter space will enter later,
in the density argument.
\end{rk}

Fix $m\geq1$ and $f_0\in\mathcal E^h_{m+1}$, and retain the notation
of Lemma~\ref{lem:local-uniform-hyperbolicity}. Choose
\[
r_0<\rho_0<\rho_1<\rho_2<R_0<R_1.
\]
For $p\in\{0,\ldots,n\}$, define
\[
\mathcal Z_{m+1}^p(f_0)
:=
\{e_j:i(e_j)=p\},
\]
where $e_1,\ldots,e_k$ are the equilibria of $f_0$ in $K_{m+1}$.
For $r_0<\rho<R_0$, set
\[
D_\rho^p
:=
\bigcup_{e_j\in\mathcal Z_{m+1}^p(f_0)}
\overline{B(e_j,\rho)}.
\]
If $\mathcal Z_{m+1}^p(f_0)=\varnothing$, we set
$D_\rho^p=\varnothing$. For $l\geq1$, let $G_{l,p}(f_0,m)$ be the set of all
$f\in U(f_0,m)$ for which there is no nonconstant complete solution
$u$ of $\dot x=f(x)$ satisfying $u(\mathbb R)\subset K_m$ and $u(t)\in D_{\rho_0}^p$ for all $|t|\geq l$.

Any solution excluded by this definition is a connecting orbit. Indeed,
since $\rho_0<R_1$, Lemma~\ref{lem:local-uniform-hyperbolicity}(ii) implies that the closed balls occurring
in $D_{\rho_0}^p$ are pairwise disjoint. Each of the sets
$u((-\infty,-l])$ and $u([l,\infty))$ is connected, so there are
fixed indices $j^-$ and $j^+$ such that $u(t)\in\overline{B(e_{j^-},\rho_0)}\,(t\leq-l)$ and $u(t)\in\overline{B(e_{j^+},\rho_0)}\,(t\geq l)$.
Since $\rho_0<R_0$, Lemma~\ref{lem:local-uniform-hyperbolicity}(iv) applies to both tails and yields $u(t)\to e_{j^-}(f)$ as $t\to-\infty$ and $u(t)\to e_{j^+}(f)$ as $t\to+\infty$.
By Lemma~\ref{lem:local-uniform-hyperbolicity}(i), $i(e_{j^-}(f))=i(e_{j^+}(f))=p$. 
Thus the solutions excluded by $G_{l,p}(f_0,m)$ are precisely the
nonconstant connections in $K_m$ between equilibria of Morse index $p$
whose two tails lie in the prescribed $\rho_0$-neighborhoods for
$|t|\geq l$.

Define
\[
G(f_0,m)
:=
\bigcap_{l\geq1}\bigcap_{p=0}^n G_{l,p}(f_0,m).
\]
The set $G(f_0,m)$ has the property required in
Proposition~\ref{Reduction}. Indeed, suppose that $f\in G(f_0,m)$ admits a
nonconstant connecting orbit contained in $K_m$ between two
equilibria of the same Morse index $p$. By Lemma~\ref{lem:local-uniform-hyperbolicity}(i), its
endpoints are $e_{j^-}(f)$ and $e_{j^+}(f)$ for suitable indices
$j^-$ and $j^+$. Since
$e_{j^\pm}(f)\in B(e_{j^\pm},r_0)$ and $r_0<\rho_0$, convergence to
the two endpoints implies that, for some $l\geq1$, $u(t)\in D_{\rho_0}^p$ for all $|t|\geq l$.
This contradicts $f\in G_{l,p}(f_0,m)$.

It therefore suffices to prove that $G_{l,p}(f_0,m)$ is open and
dense in $U(f_0,m)$ for every $l\geq1$ and
$p\in\{0,\ldots,n\}$.

	\subsection{\texorpdfstring{$G_{l,p}(f_0,m)$ is open in $U(f_0,m)$}{G(l,p)(f0,m) is open in U(f0,m)}}\label{KSO}

Fix $l\geq1$ and $p\in\{0,\ldots,n\}$. The openness argument
rests on the following compactness lemma, which controls both tails
of a family of complete solutions.

\begin{lem}\label{lem:compactness-connections}
Let $K$ be a compact subset of $\Omega$. Let $f_\nu\to f_\infty$ in $U(f_0,m)$ with respect to the topology
$\iota$, and let $u_\nu$ be complete solutions of
$\dot x=f_\nu(x)$ such that $u_\nu(\mathbb R)\subset K$ for every $\nu$.
Suppose that there exist indices $j^-,j^+$, a number $t_0>0$, and a
radius $r_0<\rho<R_0$, all independent of $\nu$, such that
\[
u_\nu(t)\in\overline{B(e_{j^-},\rho)}
\quad (t\leq-t_0),
\qquad
u_\nu(t)\in\overline{B(e_{j^+},\rho)}
\quad (t\geq t_0).
\]

After passing to a subsequence, there exists a complete
solution $u_\infty$ of $\dot x=f_\infty(x)$, with
$u_\infty(\mathbb R)\subset K$, such that
\[
\sup_{t\in\mathbb R}
\|u_\nu(t)-u_\infty(t)\|
\longrightarrow0
\qquad(\nu\to\infty).
\]
Moreover,
$u_\infty(t)\to e_{j^-}(f_\infty)$ as $t\to-\infty$ and
$u_\infty(t)\to e_{j^+}(f_\infty)$ as $t\to+\infty$. If $j^-\neq j^+$, then $u_\infty$ is nonconstant. If
$j^-=j^+=j$ and, for every $\nu$, there exists
$t_\nu\in[-t_0,t_0]$ such that $\|u_\nu(t_\nu)-e_j\|\geq R_0$,
then $u_\infty$ is nonconstant.
\end{lem}

\begin{proof}
The inclusion $u_\nu(\mathbb R)\subset K$ for every $\nu$ and the
compactness of $K$ give a uniform bound for $\{u_\nu\}$ on
$\mathbb R$. Convergence $f_\nu\to f_\infty$ in $\iota$
also gives uniform convergence on $K$, so there exists $M>0$ such
that $\sup_\nu\sup_{x\in K}\|f_\nu(x)\|\leq M$.
Therefore, for $s<t$, $\|u_\nu(t)-u_\nu(s)\|\leq M|t-s|$, and hence $\{u_\nu\}$ is equicontinuous on $\mathbb R$.

Apply the Arzel\`a--Ascoli theorem on the intervals
$[-N,N]$, $N\in\mathbb N$, and take a diagonal subsequence.
This gives a subsequence, still denoted by $u_\nu$, and a continuous
function $u_\infty:\mathbb R\to K$ such that, for every $T>0$,
\begin{equation}
\sup_{|t|\leq T}
\|u_\nu(t)-u_\infty(t)\|
\longrightarrow0
\qquad(\nu\to\infty).
\label{eq:local-uniform-convergence}
\end{equation}
For any $s<t$, the integral identity
\[
u_\nu(t)
=
u_\nu(s)+\int_s^t f_\nu(u_\nu(r))\,dr
\]
holds. Since $f_\nu\to f_\infty$ uniformly on $K$,
\eqref{eq:local-uniform-convergence} allows passage to the limit.
Thus $u_\infty$ is a complete solution of
$\dot x=f_\infty(x)$. Since $K$ is closed,
$u_\infty(\mathbb R)\subset K$.

Because $\rho<R_0$, Lemma~\ref{lem:local-uniform-hyperbolicity}(iv) applies to both tails. Hence there
exist constants $C,\eta>0$, independent of $\nu$, such that
\begin{equation}
\|u_\nu(t)-e_{j^+}(f_\nu)\|
\leq Ce^{-\eta(t-t_0)}
\quad (t\geq t_0),
\qquad
\|u_\nu(t)-e_{j^-}(f_\nu)\|
\leq Ce^{-\eta(-t-t_0)}
\quad (t\leq-t_0).
\label{eq:uniform-tail-estimates}
\end{equation}
By Lemma~\ref{lem:local-uniform-hyperbolicity}(i),
$e_{j^\pm}(f_\nu)\to e_{j^\pm}(f_\infty)$.
Passing to the limit in \eqref{eq:uniform-tail-estimates} for each
fixed $t$, using \eqref{eq:local-uniform-convergence}, gives the
corresponding estimates for $u_\infty$. Consequently,
$u_\infty(t)\to e_{j^-}(f_\infty)$ as $t\to-\infty$ and
$u_\infty(t)\to e_{j^+}(f_\infty)$ as $t\to+\infty$.

The tail estimates strengthen \eqref{eq:local-uniform-convergence}
to uniform convergence on $\mathbb R$. Indeed, for $T>t_0$ and $t\geq T$,
\[
\begin{aligned}
\|u_\nu(t)-u_\infty(t)\|
\leq{}
\|u_\nu(t)-e_{j^+}(f_\nu)\| +\|e_{j^+}(f_\nu)-e_{j^+}(f_\infty)\| +\|e_{j^+}(f_\infty)-u_\infty(t)\|.
\end{aligned}
\]
The first and third terms are bounded by
$Ce^{-\eta(T-t_0)}$, uniformly for $t\geq T$, whereas the middle
term converges to zero as $\nu\to\infty$. The negative tail is
handled in the same way. Combining these estimates with
\eqref{eq:local-uniform-convergence} on $[-T,T]$ yields
\begin{equation}
\sup_{t\in\mathbb R}
\|u_\nu(t)-u_\infty(t)\|
\longrightarrow0.
\label{eq:global-uniform-convergence}
\end{equation}

If $j^-\neq j^+$, parts (i) and (ii) of Lemma~\ref{lem:local-uniform-hyperbolicity} give
$e_{j^-}(f_\infty)\neq e_{j^+}(f_\infty)$. Since these are the
negative and positive limits of $u_\infty$, respectively,
$u_\infty$ is nonconstant. Suppose now that $j^-=j^+=j$. Since
$t_\nu\in[-t_0,t_0]$, after passing to a further subsequence we may
assume that $t_\nu\to t_*\in[-t_0,t_0]$. By
\eqref{eq:local-uniform-convergence} and the continuity of
$u_\infty$, $\|u_\infty(t_*)-e_j\|\geq R_0$.
On the other hand, Lemma~\ref{lem:local-uniform-hyperbolicity}(i) gives
$e_j(f_\infty)\in B(e_j,r_0)$, with $r_0<R_0$. Hence
$u_\infty(t_*)\neq e_j(f_\infty)$, so $u_\infty$ is nonconstant.
This completes the proof.
\end{proof}

\begin{proof}[Proof of the openness of $G_{l,p}(f_0,m)$]
Let
$f_\nu\in U(f_0,m)\setminus G_{l,p}(f_0,m)$ and suppose that
$f_\nu\to f_\infty\in U(f_0,m)$ with respect to $\iota$. We show
that $f_\infty\notin G_{l,p}(f_0,m)$.

By the definition of $G_{l,p}(f_0,m)$, for every $\nu$ there exists
a nonconstant complete solution $u_\nu$ of
$\dot x=f_\nu(x)$ satisfying $u_\nu(\mathbb R)\subset K_m$ and $u_\nu(t)\in D_{\rho_0}^p$ for all $|t|\geq l$.
Recall that $D_{\rho_0}^p$ is the finite union of the closed balls
$\overline{B(e_j,\rho_0)}$, where
$e_j\in\mathcal Z_{m+1}^p(f_0)$. Since $\rho_0<R_1$,
Lemma~\ref{lem:local-uniform-hyperbolicity}(ii) implies that these balls are pairwise disjoint.

For each $\nu$, the sets $u_\nu((-\infty,-l])$ and
$u_\nu([l,\infty))$ are connected. Hence each of them is contained
in a single closed ball occurring in $D_{\rho_0}^p$. Thus there
exist indices $j_\nu^-$ and $j_\nu^+$ such that
\[
u_\nu(t)\in\overline{B(e_{j_\nu^-},\rho_0)}
\quad(t\leq-l),
\qquad
u_\nu(t)\in\overline{B(e_{j_\nu^+},\rho_0)}
\quad(t\geq l).
\]
Since only finitely many indices are possible, after passing to a
subsequence we may assume that
$j_\nu^-=j^-$ and $j_\nu^+=j^+$ for every $\nu$, where
$j^-$ and $j^+$ are independent of $\nu$.

Suppose first that $j^-\neq j^+$. Lemma~\ref{lem:compactness-connections} applies with
$K=K_m$, $t_0=l$, and $\rho=\rho_0$. Hence, after passing to a
subsequence, there exists a nonconstant complete solution
$u_\infty$ of $\dot x=f_\infty(x)$ such that
$u_\infty(\mathbb R)\subset K_m$ and
$u_\nu\to u_\infty$ uniformly on $\mathbb R$.

For every fixed $t\geq l$,
$u_\nu(t)\in\overline{B(e_{j^+},\rho_0)}$ for all $\nu$.
Since $u_\nu(t)\to u_\infty(t)$ and
$\overline{B(e_{j^+},\rho_0)}$ is closed,
$u_\infty(t)\in\overline{B(e_{j^+},\rho_0)}$.
Similarly,
$u_\infty(t)\in\overline{B(e_{j^-},\rho_0)}$ for every
$t\leq-l$. Therefore
$u_\infty(t)\in D_{\rho_0}^p$ whenever $|t|\geq l$.
Hence $f_\infty\notin G_{l,p}(f_0,m)$.

It remains to consider $j^-=j^+=j$. We claim that, for every $\nu$,
there exists $t_\nu\in(-l,l)$ such that
\begin{equation}
\|u_\nu(t_\nu)-e_j\|\geq R_0.
\label{eq:exit-condition}
\end{equation}
Suppose that \eqref{eq:exit-condition} fails for some $\nu$. Then
$\|u_\nu(t)-e_j\|<R_0$ for every $t\in(-l,l)$. On the other hand,
for $|t|\geq l$,
$u_\nu(t)\in\overline{B(e_j,\rho_0)}
\subset B(e_j,R_0)$ because $\rho_0<R_0$. Hence
$u_\nu(t)\in B(e_j,R_0)$ for every $t\in\mathbb R$.
Applying Lemma~\ref{lem:local-uniform-hyperbolicity}(iv) in the positive and negative time directions
gives
\[
u_\nu(0)\in
W^s_{\mathrm{loc}}(e_j(f_\nu))
\cap
W^u_{\mathrm{loc}}(e_j(f_\nu)).
\]
By Lemma~\ref{lem:local-uniform-hyperbolicity}(iii), this intersection equals
$\{e_j(f_\nu)\}$. Thus $u_\nu(0)=e_j(f_\nu)$, and uniqueness of
solutions implies that $u_\nu$ is constant, contradicting the
choice of $u_\nu$. Therefore \eqref{eq:exit-condition} holds.

Lemma~\ref{lem:compactness-connections} now applies with $K=K_m$, $t_0=l$, $\rho=\rho_0$, and the
sequence $\{t_\nu\}$. Hence, after passing to a subsequence, there
exists a nonconstant complete solution $u_\infty$ of
$\dot x=f_\infty(x)$, with $u_\infty(\mathbb R)\subset K_m$, such
that $u_\nu\to u_\infty$ uniformly on $\mathbb R$.

For every fixed $t\geq l$,
$u_\nu(t)\in\overline{B(e_j,\rho_0)}$ for all $\nu$, and hence
$u_\infty(t)\in\overline{B(e_j,\rho_0)}$. The same argument applies
for every fixed $t\leq-l$. Therefore
$u_\infty(t)\in D_{\rho_0}^p$ whenever $|t|\geq l$, and hence
$f_\infty\notin G_{l,p}(f_0,m)$.

Thus $U(f_0,m)\setminus G_{l,p}(f_0,m)$ is closed in $U(f_0,m)$,
and $G_{l,p}(f_0,m)$ is open in $U(f_0,m)$.
\end{proof}

\begin{rk}\label{rem:exit-condition}
{\rm The exit condition \eqref{eq:exit-condition} is needed only when
$j^-=j^+$. Without this condition, a sequence of nonconstant
homoclinic solutions may converge to the equilibrium
$e_j(f_\infty)$. Such constant limits cannot be discarded, since
they will also occur later in the properness argument for the full
zero set.

The passage from convergence on bounded time intervals to convergence
on $\mathbb R$ uses the uniform exponential estimates in
Lemma~\ref{lem:local-uniform-hyperbolicity}(iv). These estimates control
both tails and turn \eqref{eq:local-uniform-convergence} into
\eqref{eq:global-uniform-convergence}.}
\end{rk}

	\subsection{\texorpdfstring{$G_{l,p}(f_0,m)$ is dense in $U(f_0,m)$}{G(l,p)(f0,m) is dense in U(f0,m)}}\label{KSD}
Let $W$ be an arbitrary nonempty open subset of $U(f_0,m)$.
It suffices to show that $W\cap G_{l,p}(f_0,m)\neq\varnothing$. By Lemma~\ref{dense in L-}, there exist
$f_*\in W\cap\mathcal C^\infty_{\mathrm{BF}}$ and $\kappa>0$
such that
$\beta_i(f_*,x)\geq\kappa$ and $\gamma_i(f_*,x)\geq\kappa$
for every $x\in\Omega$ and $1\leq i\leq n$. We encode connecting
orbits as zeros of a map on a Banach sequence space. After verifying
the hypotheses of Theorem~\ref{SS-1}, we shall obtain the density of
$G_{l,p}(f_0,m)$ in $U(f_0,m)$.

Define
\[
H:=
\left\{
h\in\mathcal C^1_{\mathrm{BF}}:
\|h\|_b:=
\sup_{x\in\Omega}
\bigl(\|h(x)\|+\|Dh(x)\|\bigr)<\infty
\right\}.
\]
The norm $\|\cdot\|_b$ makes $H$ a Banach space. Indeed, if
$\{h_\nu\}$ is Cauchy in $H$, then $h_\nu$ and $Dh_\nu$ converge
uniformly on $\Omega$, and the limit belongs to
$\mathcal C^1_{\mathrm{BF}}$.

Since $W$ is open with respect to $\iota$ and $f_*\in W$, there
exist $m_0\geq1$ and $\delta>0$ such that $\{g\in\mathcal L^-:p_{m_0}(g-f_*)<\delta\}\subset W$.
For every $a\in H$,
\[
p_{m_0}(a)
=
\sup_{x\in K_{m_0}}
\bigl(\|a(x)\|+\|Da(x)\|\bigr)
\leq \|a\|_b.
\]
For $\varepsilon>0$, set
\[
\mathbb B_\varepsilon^H
:=
\{a\in H:\|a\|_b<\varepsilon\},
\qquad
f_a:=f_*+a.
\]
Choosing $\varepsilon>0$ sufficiently small, we have
$f_a\in W$ for every $a\in\mathbb B_\varepsilon^H$. By decreasing
$\varepsilon$ if necessary, we may further assume that
$
\beta_i(f_a,x)\geq{\kappa}/{2}$ and $
\gamma_i(f_a,x)\geq{\kappa}/{2},
$
for every $a\in\mathbb B_\varepsilon^H$, $x\in\Omega$, and
$1\leq i\leq n$. Hence $f_a\in\mathcal L^-$ for every
$a\in\mathbb B_\varepsilon^H$. The open ball
$\mathbb B_\varepsilon^H$ will serve as the parameter domain in the
Sard--Smale theorem.

To construct the discretization, choose a time step $\tau>0$ uniformly
for $a\in\mathbb B_\varepsilon^H$ and $j\in\{1,\ldots,k\}$. Since
$\overline{V_0}\Subset V_1$, set
\[
\delta_0
:=
\operatorname{dist}
\bigl(\overline{V_0},\mathbb R^n\setminus V_1\bigr)>0.
\]
Moreover, for every $j\in\{1,\ldots,k\}$,
\[
\operatorname{dist}
\bigl(
\overline{B(e_j,\rho_1)},
\mathbb R^n\setminus B(e_j,\rho_2)
\bigr)
=
\rho_2-\rho_1>0.
\]
The set $\overline{V_1}
\cup
\bigcup_{j=1}^k\overline{B(e_j,\rho_2)}$
is a compact subset of $\Omega$. Since $f_a=f_*+a$ and
$\|a\|_b<\varepsilon$, there exists $M>0$, independent of
$a\in\mathbb B_\varepsilon^H$, such that $\|f_a(x)\|\leq M$
for every $x\in
\overline{V_1}
\cup
\bigcup_{j=1}^k\overline{B(e_j,\rho_2)}$.
Choose $\tau>0$ such that
\[
M\tau<\min\{\delta_0,\rho_2-\rho_1\}.
\]
A first-exit argument, together with the local existence theorem,
gives the following properties for every
$a\in\mathbb B_\varepsilon^H$:
\begin{itemize}
\item every solution starting in $V_0$ exists on $[-\tau,\tau]$
and remains in $V_1$;
\item for every $j\in\{1,\ldots,k\}$, every solution starting in
$B(e_j,\rho_1)$ exists on $[-\tau,\tau]$ and remains in
$B(e_j,\rho_2)$.
\end{itemize}
This choice of $\tau$ will remain fixed for the rest of the proof.

Set
\[
X:=\ell^\infty(\mathbb Z,\mathbb R^n),
\qquad
\|v\|_\infty:=\sup_{q\in\mathbb Z}\|v_q\|.
\]
For $v=(v_q)_{q\in\mathbb Z}\in X$, define
\[
d_0(v)
:=
\inf_{q\in\mathbb Z}
\operatorname{dist}
\bigl(v_q,\mathbb R^n\setminus V_0\bigr),
\]
and
\[
d_1(v)
:=
\inf_{\substack{q\in\mathbb Z\\ |q\tau|\geq l}}
\operatorname{dist}
\bigl(v_q,\mathbb R^n\setminus D_{\rho_1}^p\bigr).
\]
The condition $d_0(v)>0$ keeps all components $v_q$ a uniform
positive distance from $\mathbb R^n\setminus V_0$. Likewise,
$d_1(v)>0$ keeps the components with indices $q$ satisfying
$|q\tau|\geq l$ a uniform positive distance from
$\mathbb R^n\setminus D_{\rho_1}^p$.
For every fixed $A\subset\mathbb R^n$, $\bigl|
\operatorname{dist}(x,A)-\operatorname{dist}(y,A)
\bigr|
\leq\|x-y\|$.
Consequently, for $v,w\in X$,
\[
|d_i(v)-d_i(w)|
\leq
\|v-w\|_\infty,
\qquad i=0,1.
\]
Hence
\[
\mathcal D_{m,l,p}
:=
\{v\in X:d_0(v)>0,\ d_1(v)>0\}
\]
is an open subset of $X$. If $\mathcal Z_{m+1}^p(f_0)=\varnothing$, then
$D_{\rho_0}^p=\varnothing$, so
$G_{l,p}(f_0,m)=U(f_0,m)$, and there is nothing to prove. We
therefore assume from now on that
$\mathcal Z_{m+1}^p(f_0)\neq\varnothing$.

Define
\begin{equation}\label{eq:discretization-map}
\Psi_{m,l,p}:
\mathcal D_{m,l,p}\times\mathbb B_\varepsilon^H
\to X,
\qquad
\Psi_{m,l,p}(v,a)(q)
=
v_{q+1}-S_{f_a}(\tau)v_q.
\end{equation}
For $v_q\in V_0$, as required by $v\in\mathcal D_{m,l,p}$, the
choice of $\tau$ ensures that $S_{f_a}(\tau)v_q$ exists. Hence
$\Psi_{m,l,p}$ is well defined.

\begin{lem}\label{lem:zero-trajectory-correspondence}
Let $(v,a)\in
\mathcal D_{m,l,p}\times\mathbb B_\varepsilon^H$, where
$v=(v_q)_{q\in\mathbb Z}$.

\begin{enumerate}
\item [(i)]
If $\Psi_{m,l,p}(v,a)=0$, then there exists a unique complete
solution $u:\mathbb R\to\Omega$ of $\dot x=f_a(x)$ satisfying the
following properties:
\begin{enumerate}
\item [(a)] $u(\mathbb R)\subset V_1$ and
$u(q\tau)=v_q$ for every $q\in\mathbb Z$;

\item [(b)] There exist indices $j^-,j^+$, with
$i(e_{j^-})=i(e_{j^+})=p$, such that
$u(t)\in B(e_{j^-},\rho_2)$ for $t\leq-l-\tau$ and
$u(t)\in B(e_{j^+},\rho_2)$ for $t\geq l+\tau$;

\item [(c)] The corresponding limits satisfy
$\lim_{t\to-\infty}u(t)=e_{j^-}(f_a)$ and
$\lim_{t\to+\infty}u(t)=e_{j^+}(f_a)$.
Consequently, $u$ is either a constant equilibrium solution
$e_j(f_a)$, with $i(e_j)=p$, or a nonconstant connecting orbit
between two equilibria of Morse index $p$.
\end{enumerate}

\item [(ii)]
Conversely, suppose that $u$ is a complete solution of
$\dot x=f_a(x)$ satisfying $u(\mathbb R)\subset K_m$ and $u(t)\in D_{\rho_0}^p$ for any $|t|\geq l$.
Then the sequence $v=(v_q)_{q\in\mathbb Z}$, defined by
$v_q=u(q\tau)$, belongs to $\mathcal D_{m,l,p}$ and satisfies
$\Psi_{m,l,p}(v,a)=0$.
\end{enumerate}
\end{lem}

\begin{proof}
Assume first that $\Psi_{m,l,p}(v,a)=0$. Then
$v_{q+1}=S_{f_a}(\tau)v_q$ for every $q\in\mathbb Z$. For
$t\in[q\tau,(q+1)\tau]$, define
$u(t):=S_{f_a}(t-q\tau)v_q$. Since
$S_{f_a}(\tau)v_q=v_{q+1}=S_{f_a}(0)v_{q+1}$, the definitions on
consecutive intervals agree at their common endpoints. Thus $u$ is
a well-defined complete solution of $\dot x=f_a(x)$ satisfying
$u(q\tau)=v_q$ for every $q\in\mathbb Z$. Uniqueness follows from
uniqueness of the initial-value problem. Moreover, $d_0(v)>0$
implies $v_q\in V_0$ for every $q$, and the choice of $\tau$ gives
$u([q\tau,(q+1)\tau])\subset V_1$. Hence
$u(\mathbb R)\subset V_1$, proving (a).

For (b), let $q_+:=\min\{q\in\mathbb Z:q\tau\geq l\}$. Since
$d_1(v)>0$, there exists $j^+$ with $i(e_{j^+})=p$ such that
$v_{q_+}\in B(e_{j^+},\rho_1)$. By the choice of $\tau$,
$v_{q_++1}\in B(e_{j^+},\rho_2)$. Since
$(q_++1)\tau\geq l$, there is also some $j'$ with
$i(e_{j'})=p$ such that $v_{q_++1}\in B(e_{j'},\rho_1)$.
Because $\rho_1,\rho_2<R_1$ and the balls $B(e_j,R_1)$ are
pairwise disjoint, necessarily $j'=j^+$. Induction gives
$v_q\in B(e_{j^+},\rho_1)$ for all $q\geq q_+$, and hence
$u(t)\in B(e_{j^+},\rho_2)$ for all $t\geq q_+\tau$. Since
$q_+\tau<l+\tau$, we obtain
$u(t)\in B(e_{j^+},\rho_2)$ for $t\geq l+\tau$.

Similarly, let $q_-:=\max\{q\in\mathbb Z:q\tau\leq-l\}$. Applying
the same argument backward in time gives an index $j^-$ with
$i(e_{j^-})=p$ such that $v_q\in B(e_{j^-},\rho_1)$ for all
$q\leq q_-$. Hence $u(t)\in B(e_{j^-},\rho_2)$ for all
$t\leq q_-\tau$. Since $q_-\tau>-l-\tau$, it follows that
$u(t)\in B(e_{j^-},\rho_2)$ for $t\leq-l-\tau$. This proves (b).

Since $\rho_2<R_0$, Lemma~\ref{lem:local-uniform-hyperbolicity}(iv) applies to both tails and yields
$\lim_{t\to-\infty}u(t)=e_{j^-}(f_a)$ and
$\lim_{t\to+\infty}u(t)=e_{j^+}(f_a)$. By Lemma~\ref{lem:local-uniform-hyperbolicity}(i),
$i(e_{j^\pm}(f_a))=i(e_{j^\pm})=p$. This proves (c).

For (ii), define $v_q:=u(q\tau)$. Since
$u(\mathbb R)\subset K_m\Subset V_0$, we have
$d_0(v)\geq\operatorname{dist}(K_m,\mathbb R^n\setminus V_0)>0$.
If $|q\tau|\geq l$, then $v_q\in D_{\rho_0}^p$. Since
$\rho_0<\rho_1$ and the corresponding $R_1$-balls are pairwise
disjoint,
\[
\operatorname{dist}
\bigl(D_{\rho_0}^p,\mathbb R^n\setminus D_{\rho_1}^p\bigr)
\geq \rho_1-\rho_0>0.
\]
Thus $d_1(v)>0$, so $v\in\mathcal D_{m,l,p}$. Finally,
$v_{q+1}=u((q+1)\tau)=S_{f_a}(\tau)u(q\tau)
=S_{f_a}(\tau)v_q$ for every $q\in\mathbb Z$. Therefore
$\Psi_{m,l,p}(v,a)=0$.
\end{proof}

The linearization of $\Psi_{m,l,p}$ at a zero records the
transversality of the corresponding connecting orbit.

\begin{lem}\label{lem:psi-fredholm}
The map $\Psi_{m,l,p}$ is of class $C^1$. Moreover:
\begin{enumerate}
\item [(i)]
for every $(v,a)\in\Psi_{m,l,p}^{-1}(0)$, the operator $D_v\Psi_{m,l,p}(v,a):X\to X$
is Fredholm of index zero;

\item [(ii)]
if $0$ is a regular value of $\Psi_{m,l,p}(\cdot,a)$, then
$f_a\in G_{l,p}(f_0,m)$.
\end{enumerate}
\end{lem}

\begin{proof}
Let $G(x,a)=S_{f_a}(\tau)x$. Our choice of $\tau$ keeps the relevant
solution segments in a fixed compact subset of $\Omega$. For each
fixed parameter $a$, uniform continuity of $Df_a$ on this set,
together with the bounds on $h$ and $Dh$ supplied by $\|h\|_b$, gives a
Taylor remainder for $G$ that is uniform in the sequence index. Taking the
supremum over all indices proves the Fr\'echet differentiability of
$\Psi_{m,l,p}$, with
\[
\bigl(D\Psi_{m,l,p}(v,a)(w,h)\bigr)_q
=
w_{q+1}
-
D_xG(v_q,a)w_q
-
D_aG(v_q,a)h.
\]
The corresponding estimates for the variational equations also give
continuity of the derivative in the operator norm. Hence $\Psi_{m,l,p}$ is of class
$C^1$.

Let $(v,a)\in\Psi_{m,l,p}^{-1}(0)$, and let $u$ be the complete
solution associated with $(v,a)$ by
Lemma~\ref{lem:zero-trajectory-correspondence}. Let $S_u(t,s):=S_{Df_a(u)}(t,s)$ denote the evolution operator
of the variational equation along $u$. It is
characterized by
\[
\frac{\partial}{\partial t} S_u(t,s)
=
Df_a(u(t)) S_u(t,s),
\qquad
 S_u(s,s)=I.
\]
Set $T_q:= S_u((q+1)\tau,q\tau)$. For $w\in X$ and $h\in H$,
define
\[
(L_{v,a}w)_q:=w_{q+1}-T_qw_q,
\qquad
(J_{v,a}h)_q:=
\left.
\frac{d}{d\epsilon}
S_{f_a+\epsilon h}(\tau)v_q
\right|_{\epsilon=0}.
\]
Then
\begin{equation}
D\Psi_{m,l,p}(v,a)(w,h)
=
L_{v,a}w-J_{v,a}h,
\qquad
D_v\Psi_{m,l,p}(v,a)=L_{v,a}.
\label{eq:psi-derivative}
\end{equation}
By Lemma~\ref{lem:zero-trajectory-correspondence}, the complete
trajectory $u$ has limiting equilibria $e^-$ and $e^+$ satisfying $i(e^-)=i(e^+)=p$.
Corollary~\ref{cor:B-functional-transversality} therefore yields $\operatorname{ind}L_{v,a}
=
i(e^-)-i(e^+)
=
0$.
Since $D_v\Psi_{m,l,p}(v,a)=L_{v,a}$ by
\eqref{eq:psi-derivative}, the operator
$D_v\Psi_{m,l,p}(v,a):X\to X$ is Fredholm of index zero.

For (ii), suppose that $0$ is a regular value of
$\Psi_{m,l,p}(\cdot,a)$. Assume, to the contrary, that
$f_a\notin G_{l,p}(f_0,m)$. Then there exists a nonconstant complete
solution $u$ of $\dot x=f_a(x)$ such that $u(\mathbb R)\subset K_m$ and $
u(t)\in D_{\rho_0}^p$ for $|t|\geq l$.
Define $v_q:=u(q\tau)$. By
Lemma~\ref{lem:zero-trajectory-correspondence}(ii),
$v\in\mathcal D_{m,l,p}$ and
$\Psi_{m,l,p}(v,a)=0$. Since $0$ is a regular value of $\Psi_{m,l,p}(\cdot,a)$,
$D_v\Psi_{m,l,p}(v,a)=L_{v,a}$ is surjective. Hence
Corollary~\ref{cor:B-functional-transversality} implies that $u$ is
transverse. By Lemma~\ref{lem:zero-trajectory-correspondence}(i), $u$ connects two
equilibria $e^-$ and $e^+$ with $i(e^-)=i(e^+)=p$. 
Thus the connection cannot be transverse, a contradiction. Therefore
$f_a\in G_{l,p}(f_0,m)$.
\end{proof}

\begin{rk}\label{rem:constant-zero}
{\rm If $(v,a)$ corresponds to a constant solution $u(t)\equiv e$, then
$L_{v,a}:X\to X$ is an isomorphism by
Corollary~\ref{cor:B-functional-transversality}. Consequently,
$D\Psi_{m,l,p}(v,a):X\times H\to X$ is surjective.}
\end{rk}

We next express surjectivity of the full derivative in terms of
the adjoint equation. For $(v,a)\in\Psi_{m,l,p}^{-1}(0)$,
let $u$ be the complete trajectory associated with $(v,a)$ by
Lemma~\ref{lem:zero-trajectory-correspondence}. The adjoint equation
along $u$ is
\begin{equation}
\dot\varphi(t)
=
-Df_a(u(t))^*\varphi(t),
\qquad t\in\mathbb R.
\label{eq:7-adjoint}
\end{equation}

\begin{lem}\label{lem:range-full-derivative}
Let $(v,a)\in\Psi_{m,l,p}^{-1}(0)$ and $Y\in X$. The following
statements are equivalent:
\begin{enumerate}
\item [(i)]
$Y\in\operatorname{Ran}D\Psi_{m,l,p}(v,a)$;

\item [(ii)]
there exists $h\in H$ such that
$Y+J_{v,a}h\in\operatorname{Ran}L_{v,a}$;

\item [(iii)]
there exists $h\in H$ such that
\[
\sum_{q\in\mathbb Z}
\left\langle
\varphi((q+1)\tau),(J_{v,a}h)_q
\right\rangle
=
-
\sum_{q\in\mathbb Z}
\left\langle
\varphi((q+1)\tau),Y_q
\right\rangle
\]
for every bounded solution
$\varphi\in C_b(\mathbb R,\mathbb R^n)$ of
\eqref{eq:7-adjoint}.
\end{enumerate}
\end{lem}

\begin{proof}
By \eqref{eq:psi-derivative}, (i) is equivalent to (ii).
The range characterization in
Corollary~\ref{cor:B-functional-transversality} shows that (ii) is
equivalent to (iii). By Lemma~\ref{lem:B-adjoint-decay},
$(\varphi(q\tau))_{q\in\mathbb Z}\in
\ell^1(\mathbb Z,\mathbb R^n)$. Since
$Y,J_{v,a}h\in X=\ell^\infty(\mathbb Z,\mathbb R^n)$,
both series in (iii) converge absolutely.
\end{proof}

\begin{lem}\label{lem:finite-adjoint-surjectivity}
Let $(v,a)\in\Psi_{m,l,p}^{-1}(0)$. Suppose that the space of bounded
solutions of \eqref{eq:7-adjoint} has dimension $d>0$, and let
$\varphi^1,\ldots,\varphi^d$ be a basis. Then
$D\Psi_{m,l,p}(v,a)$ is surjective if and only if the linear map
\begin{equation}
h\longmapsto
\left(
\sum_{q\in\mathbb Z}
\left\langle
\varphi^i((q+1)\tau),(J_{v,a}h)_q
\right\rangle
\right)_{i=1}^d
\in\mathbb R^d
\label{eq:adjoint-parameter-map}
\end{equation}
from $H$ to $\mathbb R^d$ is surjective.
\end{lem}

\begin{proof}
Suppose first that $D\Psi_{m,l,p}(v,a)$ is surjective. If the map
in \eqref{eq:adjoint-parameter-map} were not surjective, there would
exist a nonzero vector $(c_1,\ldots,c_d)\in\mathbb R^d$ orthogonal to
its range. Then
$\varphi:=\sum_{i=1}^d c_i\varphi^i$ is a nonzero bounded solution of
\eqref{eq:7-adjoint} satisfying
\[
\sum_{q\in\mathbb Z}
\left\langle
\varphi((q+1)\tau),(J_{v,a}h)_q
\right\rangle
=0
\qquad\text{for every }h\in H.
\]
Taking $Y_q=\varphi((q+1)\tau)$ in
Lemma~\ref{lem:range-full-derivative} gives
\[
0
=
-\sum_{q\in\mathbb Z}
\|\varphi((q+1)\tau)\|^2,
\]
a contradiction. The series converges by
Lemma~\ref{lem:B-adjoint-decay}, and it is strictly positive because
a nonzero solution of \eqref{eq:7-adjoint} cannot vanish at a finite
time.

Conversely, suppose that the map in
\eqref{eq:adjoint-parameter-map} is surjective. Given $Y\in X$,
choose $h\in H$ such that, for $i=1,\ldots,d$,
\[
\sum_{q\in\mathbb Z}
\left\langle
\varphi^i((q+1)\tau),(J_{v,a}h)_q
\right\rangle
=
-
\sum_{q\in\mathbb Z}
\left\langle
\varphi^i((q+1)\tau),Y_q
\right\rangle.
\]
By linearity, the same equality holds for every bounded solution of
\eqref{eq:7-adjoint}. Lemma~\ref{lem:range-full-derivative} therefore
gives $Y\in\operatorname{Ran}D\Psi_{m,l,p}(v,a)$.
\end{proof}

\begin{rk}\label{rem:no-bounded-adjoint}
If \eqref{eq:7-adjoint} has no nonzero bounded solution, then
Corollary~\ref{cor:B-functional-transversality} gives
$\operatorname{Ran}L_{v,a}=X$, so the full derivative is surjective.
\end{rk}

\begin{lem}\label{lem:integral-surjectivity}
Let $(v,a)\in\Psi_{m,l,p}^{-1}(0)$, and let $u$ be the complete
solution associated with $(v,a)$ by
Lemma~\ref{lem:zero-trajectory-correspondence}. Then
$D\Psi_{m,l,p}(v,a)$ is surjective if and only if, for every nonzero
bounded solution $\varphi\in C_b(\mathbb R,\mathbb R^n)$ of
\eqref{eq:7-adjoint}, there exists $h\in H$ such that
\begin{equation}
\int_{\mathbb R}
\langle\varphi(t),h(u(t))\rangle\,dt
\neq0.
\label{eq:adjoint-detection}
\end{equation}
When these conditions hold,
$\ker D\Psi_{m,l,p}(v,a)$ has a closed complement in $X\times H$.
\end{lem}

\begin{proof}
If $u(t)\equiv e$, then $L_{v,a}:X\to X$ is an isomorphism by
Remark~\ref{rem:constant-zero}, so $D\Psi_{m,l,p}(v,a)$ is surjective.
Hyperbolicity of $e$ gives $U_0\oplus S_0=\mathbb R^n$.
Lemma~\ref{lem:B-adjoint-decay} then excludes nonzero bounded
solutions of the adjoint equation, and the integral condition is
vacuously satisfied.

Suppose that $u$ is nonconstant. By
Lemma~\ref{lem:finite-adjoint-surjectivity} and finite-dimensional
linear algebra, $D\Psi_{m,l,p}(v,a)$ is surjective if and only if,
for every nonzero bounded solution $\varphi$ of
\eqref{eq:7-adjoint}, there exists $h\in H$ such that
\[
\sum_{q\in\mathbb Z}
\left\langle
\varphi((q+1)\tau),(J_{v,a}h)_q
\right\rangle
\neq0.
\]
Differentiating the flow with respect to the parameter $a$ and using
variation of constants gives
\begin{equation}
(J_{v,a}h)_q
=
\int_{q\tau}^{(q+1)\tau} S_u((q+1)\tau,s)h(u(s))\,ds.
\label{eq:J-variation}
\end{equation}
Since solutions of \eqref{eq:7-adjoint} satisfy
\[ S_u((q+1)\tau,s)^*
\varphi((q+1)\tau)
=
\varphi(s),
\]
we obtain
\[
\begin{aligned}
\sum_{q\in\mathbb Z}
\left\langle
\varphi((q+1)\tau),(J_{v,a}h)_q
\right\rangle
&=
\sum_{q\in\mathbb Z}
\int_{q\tau}^{(q+1)\tau}
\langle\varphi(s),h(u(s))\rangle\,ds =
\int_{\mathbb R}
\langle\varphi(s),h(u(s))\rangle\,ds .
\end{aligned}
\]
Exponential decay from Lemma~\ref{lem:B-adjoint-decay}, together
with the boundedness of $h$, ensures absolute convergence throughout.
This proves the equivalence.

Finally, assume that $D\Psi_{m,l,p}(v,a)$ is surjective. Write
$L=L_{v,a}$ and $J=J_{v,a}$. Since $L$ is Fredholm, there exists a
closed subspace $X_1\subset X$ such that
$X=\ker L\oplus X_1$
and the restriction
$L|_{X_1}:X_1\to\operatorname{Ran}L$ is an isomorphism. Let $\pi:X\to X/\operatorname{Ran}L$
be the quotient map. Since $D\Psi_{m,l,p}(v,a)(w,h)=Lw-Jh$
is surjective, $\pi J:H\to X/\operatorname{Ran}L$ is surjective.
As $X/\operatorname{Ran}L$ is finite-dimensional, there exists a
finite-dimensional subspace $H_0\subset H$ such that
$\pi J|_{H_0}$ is an isomorphism.

Set $Z:=X_1\times H_0$. We claim that
\[
D\Psi_{m,l,p}(v,a)|_Z:Z\to X
\]
is an isomorphism. Indeed, if $Lw-Jh=0$ with $(w,h)\in Z$, then
applying $\pi$ gives $\pi Jh=0$. Hence $h=0$, and consequently
$Lw=0$. Since $w\in X_1$, it follows that $w=0$. Thus the restriction
is injective. Conversely, given $y\in X$, choose $h\in H_0$ such
that $\pi Jh=-\pi y$. Then $y+Jh\in\operatorname{Ran}L$, so there
exists $w\in X_1$ satisfying $Lw=y+Jh$. Hence
$D\Psi_{m,l,p}(v,a)(w,h)=y$, proving surjectivity.

Therefore,
\[
X\times H
=
\ker D\Psi_{m,l,p}(v,a)\oplus Z.
\]
The subspace $Z$ is closed because $X_1$ is closed and $H_0$ is
finite-dimensional. It therefore provides the required closed
complement to $\ker D\Psi_{m,l,p}(v,a)$ in $X\times H$.
\end{proof}

\begin{proof}[Proof of the density of $G_{l,p}(f_0,m)$ in $U(f_0,m)$]
We apply Theorem~\ref{SS-1} to
$\Psi_{m,l,p}:\mathcal D_{m,l,p}\times\mathbb B_\varepsilon^H\to X$
at the value $0$. By Lemma~\ref{lem:psi-fredholm},
$\Psi_{m,l,p}$ is of class $C^1$, and
$D_v\Psi_{m,l,p}(v,a):X\to X$ is Fredholm of index $0<1$ at every
$(v,a)\in\Psi_{m,l,p}^{-1}(0)$. Hence hypothesis~(i) of
Theorem~\ref{SS-1} is satisfied.

For hypothesis~(ii), fix
$(v,a)\in\Psi_{m,l,p}^{-1}(0)$, and let $u$ be the complete
trajectory associated with $(v,a)$ by
Lemma~\ref{lem:zero-trajectory-correspondence}. If $u$ is constant,
Corollary~\ref{cor:B-functional-transversality} implies that
$L_{v,a}$ is an isomorphism. Since
$D_v\Psi_{m,l,p}(v,a)=L_{v,a}$, the full derivative
$D\Psi_{m,l,p}(v,a)$ is surjective.

Suppose that $u$ is nonconstant. By
Lemma~\ref{lem:zero-trajectory-correspondence}, $u$ connects two
hyperbolic equilibria $e^-$ and $e^+$ with
$i(e^-)=i(e^+)=p$. Since $u$ is nonconstant, $p\neq0$; otherwise
$e^-$ is a sink and $W^u(e^-)=\{e^-\}$. If $p>0$ were even,
Theorem~\ref{Automatic Transversality} would imply that the connection is transverse, whereas a
nonconstant connection between equilibria of the same Morse index
cannot be transverse. Hence $p$ is odd.

Let $\varphi\neq0$ be an arbitrary bounded solution of
\eqref{eq:7-adjoint}; if no such solution exists, the condition in
Lemma~\ref{lem:integral-surjectivity} is vacuous. Choose
$t_0\in\mathbb R$ and $j\in\{1,\ldots,n\}$ such that
$\varphi_j(t_0)\neq0$. By continuity, there is an open interval
$I\ni t_0$ on which $\varphi_j$ has a fixed strict sign. With cyclic
indices, set $q_j(t):=(u_j(t),u_{j+1}(t))$ and
$q_j^\pm:=(e_j^\pm,e_{j+1}^\pm)$. Since $p$ is odd,
Proposition~\ref{one to one for connect} implies that $q_j:\mathbb R\to\mathbb R^2$ is
injective. Therefore
$q_j^{-1}(\{q_j^-,q_j^+\})$ contains at most two points. We may thus
choose
$s_0\in I\setminus q_j^{-1}(\{q_j^-,q_j^+\})$.

Choose $s_1<s_0<s_2$ with $[s_1,s_2]\subset I$ and set
\[
F:=
q_j\bigl((-\infty,s_1]\cup[s_2,\infty)\bigr)
\cup\{q_j^-,q_j^+\}.
\]
Since $q_j(t)\to q_j^\pm$ as $t\to\pm\infty$, the set $F$ is
compact. Moreover, injectivity of $q_j$ and the choice of $s_0$
give $q_j(s_0)\notin F$. Hence
$d:=\operatorname{dist}(q_j(s_0),F)>0$.
This separation is illustrated in Figure~\ref{fig:projected-perturbations}(b).

Put $z_0:=q_j(s_0)$ and define $b:\mathbb R^2\to[0,\infty)$ by
\[
b(z):=
\begin{cases}
\displaystyle
\exp\!\left(
-\frac{1}{1-16\|z-z_0\|^2/d^2}
\right),
& \|z-z_0\|<d/4,\\[1.2ex]
0,
& \|z-z_0\|\geq d/4.
\end{cases}
\]
Then $b\in C_c^\infty(B(z_0,d/2))$, $b\geq0$, and
$b(z_0)=e^{-1}>0$. Define $h=(h_1,\ldots,h_n)\in H$ by
$h_j(x):=b(x_j,x_{j+1})$ and $h_i(x):=0$ for $i\neq j$. Since $b$
and $Db$ are bounded, $h\in H$.

For $t\notin(s_1,s_2)$, one has $q_j(t)\in F$, and therefore
$b(q_j(t))=0$. Since $\varphi_j$ has a fixed strict sign on $I$ and
$b(q_j(s_0))>0$,
\[
\int_{\mathbb R}
\langle\varphi(t),h(u(t))\rangle\,dt
=
\int_{s_1}^{s_2}
\varphi_j(t)b(q_j(t))\,dt
\neq0.
\]
Lemma~\ref{lem:integral-surjectivity} now implies that
$D\Psi_{m,l,p}(v,a)$ is surjective; moreover, its kernel has a closed
complement in $X\times H$. Thus hypothesis~(ii) of Theorem~\ref{SS-1} is
satisfied.

We verify the remaining hypothesis, (iii)(b), by proving that the
projection of the zero set onto the parameter space is $\sigma$-proper.
For $r\geq1$, define
\[
M_r:=
\left\{
(v,a)\in\Psi_{m,l,p}^{-1}(0):
d_0(v)\geq\frac1r,\quad d_1(v)\geq\frac1r
\right\}.
\]
Since $d_0(v)>0$ and $d_1(v)>0$ for every
$v\in\mathcal D_{m,l,p}$, one has
$\Psi_{m,l,p}^{-1}(0)=\bigcup_{r\geq1}M_r$.

Fix $r\geq1$ and let $(v^\nu,a_\nu)\in M_r$ with
$a_\nu\to a\in\mathbb B_\varepsilon^H$. Let $u_\nu$ be the complete
trajectory corresponding to $(v^\nu,a_\nu)$. By
Lemma~\ref{lem:zero-trajectory-correspondence},
$u_\nu(\mathbb R)\subset V_1$. For each $\nu$, the same lemma yields
indices $j_\nu^-,j_\nu^+$ with
$i(e_{j_\nu^-})=i(e_{j_\nu^+})=p$ such that
$u_\nu(t)\in B(e_{j_\nu^-},\rho_2)$ for $t\leq-l-\tau$ and
$u_\nu(t)\in B(e_{j_\nu^+},\rho_2)$ for $t\geq l+\tau$.
There are only finitely many possible endpoints among $e_1,\ldots,e_k$.
After passing to a subsequence, we can therefore fix indices $j^-,j^+$ such that
$j_\nu^-=j^-$ and $j_\nu^+=j^+$ for every $\nu$. Consequently,
$\lim_{t\to-\infty}u_\nu(t)=e_{j^-}(f_{a_\nu})$ and
$\lim_{t\to+\infty}u_\nu(t)=e_{j^+}(f_{a_\nu})$.

Since $a_\nu\to a$ in $H$, for every $m'\geq1$,
$p_{m'}(f_{a_\nu}-f_a)
=p_{m'}(a_\nu-a)\leq\|a_\nu-a\|_b\to0$.
Thus $f_{a_\nu}\to f_a$ with respect to $\iota$.
Applying Lemma~\ref{lem:compactness-connections} with $K=\overline{V_1}$,
$t_0=l+\tau$, and $\rho=\rho_2$, we obtain, after passing to a
subsequence, a complete solution $u$ of $\dot x=f_a(x)$ such that
\[
\sup_{t\in\mathbb R}\|u_\nu(t)-u(t)\|\to 0.
\]
Let $v_q:=u(q\tau)$. Then
$\|v^\nu-v\|_\infty\to0$. Since $\Psi_{m,l,p}$, $d_0$, and $d_1$
are continuous,
$\Psi_{m,l,p}(v,a)=0$, $d_0(v)\geq1/r$, and $d_1(v)\geq1/r$.
Hence $(v,a)\in M_r$. Therefore the projection
$M_r\to\mathbb B_\varepsilon^H$, $(v,a)\mapsto a$, is proper, and
hypothesis~(iii)(b) of Theorem~\ref{SS-1} holds with
$Y=V=\mathbb B_\varepsilon^H$.

Theorem~\ref{SS-1} now yields a dense residual subset
\[
\Theta:=
\left\{
a\in\mathbb B_\varepsilon^H:
0\ \text{is a regular value of }\Psi_{m,l,p}(\cdot,a)
\right\}.
\]
By Lemma~\ref{lem:psi-fredholm}(ii),
$f_a\in G_{l,p}(f_0,m)$ for every $a\in\Theta$. Since
$f_a\in W$ for every $a\in\mathbb B_\varepsilon^H$ and
$\Theta\neq\varnothing$, it follows that
$W\cap G_{l,p}(f_0,m)\neq\varnothing$.

As $W$ was an arbitrary nonempty open subset of $U(f_0,m)$,
$G_{l,p}(f_0,m)$ is dense in $U(f_0,m)$.
\end{proof}

	\section{\texorpdfstring{{Generic Morse--Smale property}}{Generic Morse--Smale property}}\label{Generic Morse--Smale property111}
We prove Theorem~\ref{GenericityofMorse–Smaleproperty} in two steps.
First, Proposition~\ref{Relative generic Kupka--Smale property} constructs
a generic subset $\mathcal O_{\mathrm{KS}}$ of
$(\mathcal L_d^-,\iota_d)$ on which all critical elements are hyperbolic
and the relevant invariant manifolds intersect transversely. We then
establish the remaining global properties in
Sections~\ref{limit set finite} and~\ref{npf}.
Proposition~\ref{MSP} completes the proof by showing that every
$f\in\mathcal O_{\mathrm{KS}}$ has only finitely many critical elements
and that every nonwandering point belongs to one of them.

\subsection{Generic Kupka--Smale property in the dissipative class}

For $a\in\Omega$ and $R>0$ with $\overline{B(a,R)}\Subset\Omega$, set
$$
    \mathcal D^-_{a,R}
:=
\left\{
f\in\mathcal L^-:
\langle f(x),x-a\rangle<0
\text{ for every }x\in\Omega
\text{ with }\|x-a\|\ge R
\right\}.
$$
Thus
\[
\mathcal L_d^-=\bigcup_{a,R}\mathcal D^-_{a,R},
\]
where the union is taken over all $a\in\Omega$ and $R>0$ such that
$\overline{B(a,R)}\Subset\Omega$. All spaces in this subsection carry
the relative topology induced by $\iota$.

\begin{lem}[Admissible approximation]\label{Aa}
Fix a nonempty class $\mathcal D^-_{a,R}$. For every
$f\in\mathcal D^-_{a,R}$ and every relatively open neighborhood $\mathcal V$ of $f$ in
$\mathcal D^-_{a,R}$, there exist $f_*\in
\mathcal V\cap\mathcal D^-_{a,R}\cap
\mathcal C^\infty_{\mathrm{BF}}$
and $\kappa,\eta>0$ such that $\min\{\beta_i(f_*,x),\gamma_i(f_*,x)\}\ge\kappa$
for all $x\in\Omega$ and $1\le i\le n$, and
\begin{equation}\label{deta}
\langle f_*(x),x-a\rangle
\le -\eta\|x-a\|^2,
\qquad
\|x-a\|\ge R.
\end{equation}
Moreover, for the Banach space $H$ introduced in Section~\ref{KSD},
there exists $\delta_*>0$ such that
\[
f_*+\mathbb B^H_{\delta_*}
\subset
\mathcal V\cap\mathcal D^-_{a,R}.
\]
\end{lem}

\begin{proof}
Choose an open set $\mathcal U\subset\mathcal L^-$ such that
$f\in\mathcal U$ and
$\mathcal U\cap\mathcal D^-_{a,R}\subset\mathcal V$.
Let $B$ be a constant matrix whose allowed neighboring entries have the
signs prescribed in~(\ref{transformation}) and unit absolute value, with all remaining
off-diagonal entries zero. Let $c=\|B\|+2$ and, for $\mu>0$, define $f_\mu(x)
=
f(x)+\mu(B-cI)(x-a)$.
Since $-cI$ changes only the diagonal derivatives, $\min\{\beta_i(f_\mu,x),\gamma_i(f_\mu,x)\}\ge\mu$
for all $x\in\Omega$ and $1\le i\le n$. Moreover, for
$\|x-a\|\ge R$, writing $v=x-a$ and using
$\langle Bv,v\rangle\le\|B\|\|v\|^2$, we obtain
\begin{equation}\label{dmu}
\langle f_\mu(x),x-a\rangle
\le
-2\mu\|x-a\|^2.
\end{equation}

Since $f_\mu\to f$ with respect to $\iota$ as $\mu\to 0$, choose
$\mu>0$ sufficiently small so that $f_\mu\in\mathcal U$. Applying the
coordinatewise smoothing construction leading to~(\ref{supsup}), we may choose
$f_*\in\mathcal C^\infty_{\mathrm{BF}}\cap\mathcal U$ such that
\[
\sup_{x\in\Omega}
\|f_*(x)-f_\mu(x)\|
<
\frac{\mu R}{2},
\]
and
\[
\max\{|\beta_i(f_*,x)-\beta_i(f_\mu,x)|,\,
|\gamma_i(f_*,x)-\gamma_i(f_\mu,x)|\}
<
\frac{\mu}{2}
\]
for all $x\in\Omega$ and $1\le i\le n$. Hence
\[
\min\{\beta_i(f_*,x),\gamma_i(f_*,x)\}
\ge \frac{\mu}{2}.
\]
By~(\ref{dmu}), we also have
\[
\langle f_*(x),x-a\rangle
\le
-\frac{3\mu}{2}\|x-a\|^2,
\qquad
\|x-a\|\ge R.
\]
Thus~(\ref{deta}) holds, for instance, with
$\kappa=\mu/2$ and $\eta=\mu$, so
$f_*\in\mathcal V\cap\mathcal D^-_{a,R}$.

Finally, since $\mathcal V$ is relatively open and $f_*\in\mathcal V$,
there exist $m\ge1$ and $\rho>0$ such that
\[
g\in\mathcal D^-_{a,R},
\qquad
p_m(g-f_*)<\rho
\quad\Longrightarrow\quad
g\in\mathcal V.
\]
For $h\in H$, we have $p_m(h)\le \|h\|_b$. Moreover, for every
$x\in\Omega$ and $1\le i\le n$,
\[
\min\{\beta_i(f_*+h,x),\gamma_i(f_*+h,x)\}
\ge \kappa-\|h\|_b,
\]
while, for $\|x-a\|\ge R$,
\[
\langle (f_*+h)(x),x-a\rangle
\le
-\left(\eta-\frac{\|h\|_b}{R}\right)\|x-a\|^2.
\]
Therefore, if
\[
0<\delta_*<
\min\left\{\rho,\frac{\kappa}{2},\frac{\eta R}{2}\right\},
\]
then every $h\in H$ with $\|h\|_b<\delta_*$ satisfies
\[
p_m(h)<\rho,\qquad
\min\{\beta_i(f_*+h,x),\gamma_i(f_*+h,x)\}\ge\frac{\kappa}{2}
\]
for all $x\in\Omega$ and $1\le i\le n$, and
\[
\langle (f_*+h)(x),x-a\rangle
\le
-\frac{\eta}{2}\|x-a\|^2,
\qquad \|x-a\|\ge R.
\]
Hence $f_*+\mathbb B_{\delta_*}^H
\subset
\mathcal V\cap\mathcal D^-_{a,R}.$
\end{proof}

\begin{prop}[Relative generic Kupka--Smale property]\label{Relative generic Kupka--Smale property}
There exists a dense $G_\delta$ subset $\mathcal O_{\mathrm{KS}}$ of
$(\mathcal L_d^-,\iota_d)$ such that, for every
$f\in\mathcal O_{\mathrm{KS}}$, all critical elements are hyperbolic
and there is no nonconstant connecting orbit between equilibria of the
same Morse index. In particular, $f$ is Kupka--Smale, and homoclinic
orbits are excluded.
\end{prop}

\begin{proof}
We begin with a fixed class $\mathcal D^-_{a,R}$. With $K_m$ as in
Section~\ref{GTC}, set
\[
F_m
=
K_m\cap\{x\in\Omega:\|x-a\|\ge R\}.
\]
Without loss of generality, assume that $F_m\neq\varnothing$ for every
$m$. Then
\[
\mathcal D^-_{a,R}
=
\bigcap_{m\ge1}
\left\{
f\in\mathcal L^-:
\max_{x\in F_m}
\langle f(x),x-a\rangle<0
\right\}.
\]
Each set on the right is open in $\mathcal L^-$. Hence
$\mathcal D^-_{a,R}$ is a $G_\delta$ subset of the completely
metrizable space $\mathcal L^-$, and is therefore completely
metrizable and a Baire space.

We next show that the genericity arguments of Sections~\ref{GTC} and~\ref{GKS} remain
valid in $\mathcal D^-_{a,R}$. For every $m$, $\mathcal E_m^h\cap\mathcal D^-_{a,R}$ and $\mathcal O_m^h\cap\mathcal D^-_{a,R}$
are relatively open in $\mathcal D^-_{a,R}$. To prove density,
we use Lemma~\ref{Aa} to choose, in any prescribed nonempty relatively
open subset of $\mathcal D^-_{a,R}$, a smooth vector field with
uniformly positive neighboring derivatives and a uniform dissipative
margin. All bounded perturbations used in the proofs of
Propositions~\ref{G-H-P} and~\ref{G-H-O-M} can then be made small enough
that the perturbed vector field remains in the prescribed set.

The affine perturbation $\alpha(x-e)$ in Proposition~\ref{G-H-P}
requires a separate estimate. For $\|x-a\|\ge R$,
\[
|\langle\alpha(x-e),x-a\rangle|
\le
|\alpha|
\left(
1+\frac{\|e-a\|}{R}
\right)
\|x-a\|^2.
\]
Thus, for sufficiently small $|\alpha|$, it also preserves the
dissipative margin and the prescribed $\iota$-neighborhood. Since it
changes only the diagonal derivatives, the sign conditions defining
$\mathcal L^-$ are unaffected. Consequently, $\mathcal E_m^h\cap\mathcal D^-_{a,R}$ and $\mathcal O_m^h\cap\mathcal D^-_{a,R}$
are open and dense in $\mathcal D^-_{a,R}$.

Similarly, let $f_0\in\mathcal E_{m+1}^h$, and let $U(f_0,m)$ be the
open neighborhood given by Lemma~\ref{lem:local-uniform-hyperbolicity}. For every $l\ge1$ and
$0\le p\le n$, $G_{l,p}(f_0,m)\cap\mathcal D^-_{a,R}$
is open and dense in $U(f_0,m)\cap\mathcal D^-_{a,R}$,
whenever the latter is nonempty. Openness follows from Section~\ref{KSO}.
For density, Lemma~\ref{Aa} allows us to choose the Banach parameter ball used in
Section~\ref{KSD} so that all perturbed fields lie in any prescribed
nonempty relatively open subset of
$U(f_0,m)\cap\mathcal D^-_{a,R}$. The rest of the argument in
Section~\ref{KSD} applies without change.

We now combine these relative genericity results into a single dense
$G_\delta$ subset of $\mathcal L_d^-$. Fix $m\ge1$. For every $f_0\in\mathcal E_{m+1}^h\cap\mathcal L_d^-$,
let $U(f_0,m)$ be the neighborhood given by Lemma~\ref{lem:local-uniform-hyperbolicity}. Since
$\mathcal L_d^-$ is metrizable, we may choose a relatively open
neighborhood $V(f_0,m)$ of $f_0$ such that $\overline{V(f_0,m)}^{\,\mathcal L_d^-}
\subset
U(f_0,m)\cap\mathcal L_d^-$.
These sets $V(f_0,m)$ cover
$\mathcal E_{m+1}^h\cap\mathcal L_d^-$. Since $\mathcal L_d^-$ is
second countable, they admit a countable subcover, denoted by
$\{V_{m,j}\}_{j\ge1}$. Let $f_{m,j}$ be the corresponding reference
vector fields. Thus
$$
\overline{V_{m,j}}^{\,\mathcal L_d^-}
\subset
U(f_{m,j},m)\cap\mathcal L_d^-.
$$
For $l\ge1$ and $0\le p\le n$, set
\[
A_{m,j,l,p}
:=
\bigl(
G_{l,p}(f_{m,j},m)\cap\mathcal L_d^-
\bigr)
\cup
\left(
\mathcal L_d^-
\setminus
\overline{V_{m,j}}^{\,\mathcal L_d^-}
\right).
\]
Each $A_{m,j,l,p}$ is relatively open in $\mathcal L_d^-$. We claim
that $A_{m,j,l,p}\cap\mathcal D^-_{a,R}$
is dense in every fixed class $\mathcal D^-_{a,R}$.

Let $\mathcal W$ be a nonempty relatively open subset of
$\mathcal D^-_{a,R}$. If
\[
\mathcal W\cap
\left(
\mathcal L_d^-
\setminus
\overline{V_{m,j}}^{\,\mathcal L_d^-}
\right)
\neq\varnothing,
\]
then $\mathcal W\cap A_{m,j,l,p}\neq\varnothing$. Suppose instead
that this intersection is empty. Since
$\mathcal W\subset\mathcal D^-_{a,R}\subset\mathcal L_d^-$,
\[
\mathcal W
\subset
\overline{V_{m,j}}^{\,\mathcal L_d^-}
\subset
U(f_{m,j},m)\cap\mathcal L_d^-.
\]
Together with $\mathcal W\subset\mathcal D^-_{a,R}$, this gives $\mathcal W
\subset
U(f_{m,j},m)\cap\mathcal D^-_{a,R}$.
The relative density of
$G_{l,p}(f_{m,j},m)\cap\mathcal D^-_{a,R}$ in
$U(f_{m,j},m)\cap\mathcal D^-_{a,R}$ therefore implies $\mathcal W\cap A_{m,j,l,p}\neq\emptyset$.
Thus $A_{m,j,l,p}\cap\mathcal D^-_{a,R}$ is open and dense in
$\mathcal D^-_{a,R}$.

Define
\begin{equation}\label{OKS}
\mathcal O_{\mathrm{KS}}
:=
\bigcap_{m\ge1}
\bigl(
\mathcal O_m^h\cap\mathcal L_d^-
\bigr)
\cap
\bigcap_{m,j,l\ge1}
\bigcap_{p=0}^{n}
A_{m,j,l,p}.
\end{equation}
For every fixed $\mathcal D^-_{a,R}$, all factors in~(\ref{OKS}), after
intersection with $\mathcal D^-_{a,R}$, are open and dense in
$\mathcal D^-_{a,R}$. Since $\mathcal D^-_{a,R}$ is a Baire space, $\mathcal O_{\mathrm{KS}}\cap\mathcal D^-_{a,R}$
is dense in $\mathcal D^-_{a,R}$.

To prove density in $\mathcal L_d^-$, let $\mathcal V_d$ be a
nonempty relatively open subset of $\mathcal L_d^-$, and choose
$f\in\mathcal V_d$. By the definition of $\mathcal L_d^-$,
$f\in\mathcal D^-_{a,R}$ for some $a,R$. Hence $\mathcal V_d\cap\mathcal D^-_{a,R}$
is a nonempty relatively open subset of $\mathcal D^-_{a,R}$. Since
$\mathcal O_{\mathrm{KS}}\cap\mathcal D^-_{a,R}$ is dense in
$\mathcal D^-_{a,R}$, $\mathcal V_d\cap\mathcal O_{\mathrm{KS}}\neq\emptyset$.
Therefore $\mathcal O_{\mathrm{KS}}$ is dense in $\mathcal L_d^-$. Since
(\ref{OKS}) is a countable intersection of relatively open subsets of
$\mathcal L_d^-$, it is a dense $G_\delta$ subset of
$(\mathcal L_d^-,\iota_d)$.

Finally, let $f\in\mathcal O_{\mathrm{KS}}$. Since $f\in\bigcap_{m\ge1}\mathcal O_m^h$,
all critical elements of $f$ are hyperbolic. The second intersection in~(\ref{OKS}) rules out nonconstant connections
between equilibria of the same Morse index. Indeed, suppose
that such a connection of index $p$ exists. Its closure is contained
in some $K_m$. Since
$f\in\mathcal E_{m+1}^h\cap\mathcal L_d^-$, choose $j$ such that
$f\in V_{m,j}$. For some $l\ge1$, the two tails of the connection
satisfy the condition excluded in the definition of
$G_{l,p}(f_{m,j},m)$. Hence $f\notin G_{l,p}(f_{m,j},m)$.
However, $f\in A_{m,j,l,p}$ by~(\ref{OKS}), while
$f\in V_{m,j}\subset\overline{V_{m,j}}^{\,\mathcal L_d^-}$. The
second set in the definition of $A_{m,j,l,p}$ does not contain $f$, so $f\in G_{l,p}(f_{m,j},m),$
a contradiction.

Thus there is no nonconstant connecting orbit between equilibria of
the same Morse index. Together with Theorem~\ref{Automatic Transversality}, this yields the
Kupka--Smale property. In particular, homoclinic orbits to equilibria
are excluded. If a hyperbolic periodic orbit admitted a homoclinic
orbit, Theorem~\ref{Automatic Transversality} would make it transverse; the Birkhoff--Smale
homoclinic theorem would then yield a horseshoe, contradicting the
equilibrium-free Poincar\'e--Bendixson property in Proposition~\ref{PB}.
Hence homoclinic orbits are excluded.
\end{proof}

\subsection{Compact localization and the limit-set alternative}\label{limit set finite}

For the remainder of this section, fix
$f\in\mathcal O_{\mathrm{KS}}$, where $\mathcal O_{\mathrm{KS}}$ is the
generic set obtained in Proposition~\ref{Relative generic Kupka--Smale property}, and choose $a\in\Omega$ and
$R>0$ as in condition~(D). We retain the notation
$\phi_f(t,x)=S_f(t)x$ for the maximal solution through $x$.

\begin{lem}\label{Compact localization and limit-set alternative}
The following statements hold.
\begin{enumerate}
\item [(i)] The closed ball $\overline{B(a,R)}$ is positively invariant, and
every solution starting in $\overline{B(a,R)}$ exists for all $t\ge0$.

\item [(ii)] All nonwandering points and all critical elements in $\Omega$
lie in $B(a,R)$. Moreover, the system has only finitely many
equilibria.

\item [(iii)] For every $x\in\overline{B(a,R)}$, the set $\omega(x)$ is a
single equilibrium, a nonconstant periodic orbit, or a union of
equilibria and nonconstant complete orbits whose $\alpha$- and
$\omega$-limit sets are individual equilibria in $\omega(x)$.
\end{enumerate}
\end{lem}

\begin{proof}
We first prove (i). Suppose that a solution starting in
$\overline{B(a,R)}$ leaves this ball, and let $\tau>0$ be its first
exit time. Then $\|\phi_f(\tau,x)-a\|=R$
and
\[
\left.
\frac{d}{dt}\frac12\|\phi_f(t,x)-a\|^2
\right|_{t=\tau}
\ge0.
\]
On the other hand, condition~(D) gives
\[
\left.
\frac{d}{dt}\frac12\|\phi_f(t,x)-a\|^2
\right|_{t=\tau}
=
\langle f(\phi_f(\tau,x)),\phi_f(\tau,x)-a\rangle<0,
\]
a contradiction. Hence $\overline{B(a,R)}$ is positively invariant.
Every such solution remains in the compact set
$\overline{B(a,R)}\Subset\Omega$, so the standard continuation theorem
for ordinary differential equations gives existence for all positive
times.

We next prove (ii). By (D) and continuity, choose $R_0\in(0,R)$ such
that
\begin{equation}\label{R0}
\langle f(x),x-a\rangle<0,
\qquad
x\in\Omega,\quad \|x-a\|\ge R_0.
\end{equation}
Fix $x\in\Omega$ with $\|x-a\|>R_0$. Since
\[
\left.
\frac{d}{dt}\|\phi_f(t,x)-a\|^2
\right|_{t=0}
=
2\langle f(x),x-a\rangle<0,
\]
there exists $\tau>0$ such that $\|\phi_f(\tau,x)-a\|<\|x-a\|$.
By continuous dependence on the initial value, we may choose an open
neighborhood $U$ of $x$ such that every solution starting in $U$ is
defined on $[0,\tau]$ and
\begin{equation}\label{supU}
\sup_{y\in U}\|\phi_f(\tau,y)-a\|
<
\inf_{y\in U}\|y-a\|,
\qquad
\inf_{y\in U}\|y-a\|>R_0.
\end{equation}
Indeed, the strict inequality
$\|\phi_f(\tau,x)-a\|<\|x-a\|$ and continuity of the maps
$y\mapsto\phi_f(\tau,y)$ and $y\mapsto\|y-a\|$ yield~(\ref{supU})
once $U$ is sufficiently small.

By~(\ref{R0}), after time $\tau$ the distance from $a$ continues to
decrease as long as it is at least $R_0$; once a trajectory enters
$\overline{B(a,R_0)}$, the first-exit argument shows that it cannot
leave. Hence $\phi_f(t,U)\cap U=\emptyset$, $t\ge\tau$,
whenever the solution is defined. Thus $x$ is wandering.

Therefore every nonwandering point lies in
$\overline{B(a,R_0)}\subset B(a,R)$. Since equilibria and periodic points are
nonwandering, all critical elements lie in $B(a,R)$. By
Proposition~\ref{Relative generic Kupka--Smale property}, every equilibrium is hyperbolic and hence isolated.
The equilibrium set $f^{-1}(0)$ is closed in $\Omega$ and is contained
in the compact set $\overline{B(a,R_0)}\Subset\Omega$. It is therefore
compact. Since every equilibrium is isolated, this set is finite.

We now prove (iii). Fix $x\in\overline{B(a,R)}$. By (i), its
positive orbit is precompact; hence $\omega(x)$ is nonempty,
compact, and connected. If $\omega(x)$ contains no equilibrium,
Proposition~\ref{PB} implies that it is a nonconstant periodic
orbit. If $\omega(x)$ consists only of equilibria, its connectedness
and the finiteness of the equilibria imply that it is a single
equilibrium.

It remains to consider the case in which $\omega(x)$ contains
both equilibria and non-equilibrium points. Choose $R_1>R$ such
that $\overline{B(a,R_1)}\Subset\Omega$. The same argument as in
(i) shows that the semiflow on $B(a,R_1)$ is forward complete.
Lemma~\ref{Terescak hypotheses} verifies the hypotheses of
\cite[Corollaries~2.3--2.4]{Tl} on this ball, with
$K_0=\omega(x)$ and $\Pi=\{v:v_1=v_2=0\}$.
Let $V$ be the neighborhood supplied by that result. Since the
positive orbit is precompact and approaches $\omega(x)$, its
sufficiently late tail has compact closure contained in $V$.
After shifting the initial time, we may therefore apply
\cite[Lemma~6.2]{Tl}. Its first two alternatives,
a single periodic orbit and a set consisting only of equilibria,
have already been excluded. Hence, for some $T>0$, the fixed
coordinate projection
\[
P:\mathbb R^n\to\mathbb R^2,\qquad P(v)=(v_1,v_2),
\]
is bi-Lipschitz on
\[
K:=\overline{\{\phi_f(t,x):t\ge T\}}.
\]
In the terminology of \cite[p.~51]{Tl}, a Lipschitz embedding
is required to have a Lipschitz inverse on its image.
The set $K$ is compact and positively invariant,
and $\omega(x)\subset K$.

Write $Q=(P|_K)^{-1}:P(K)\to K$ and define
\[
F(z):=P f(Q(z)),\qquad z\in P(K).
\]
Since $f$ is Lipschitz on $\overline{B(a,R_1)}$ and $Q$ is
Lipschitz, $F$ is Lipschitz on $P(K)$. Extend each component of
$F$ to a Lipschitz function on $\mathbb R^2$, and multiply the
resulting vector field by a smooth compactly supported function
that equals one on a neighborhood of $P(K)$. This gives a
globally Lipschitz vector field
$\widetilde F:\mathbb R^2\to\mathbb R^2$ with compact support
such that $\widetilde F=F$ on $P(K)$. Let $\psi$ denote the
complete flow generated by $\dot z=\widetilde F(z)$.

For every $y\in K$, positive invariance of $K$ gives
\[
\frac{d}{dt}P\phi_f(t,y)
=
P f(\phi_f(t,y))
=
\widetilde F(P\phi_f(t,y)),
\qquad t\ge0.
\]
Uniqueness of solutions therefore yields
\[
\psi(t,Py)=P\phi_f(t,y),
\qquad y\in K,\quad t\ge0.
\]
In particular, with $z_0=P\phi_f(T,x)$, compactness of $K$
implies $\omega_\psi(z_0)=P\omega(x)$,
where $\omega_\psi$ denotes the $\omega$-limit set for the planar
flow. Since $\omega(x)$ is compact and invariant, the identity between
the two flows extends to every $y\in\omega(x)$ and all $t\in\mathbb R$.

Equilibria also correspond under $P$. Indeed, if $y\in K$ and
$\widetilde F(Py)=0$, then $\psi(t,Py)=Py$ for all $t\ge0$.
The preceding identity and the injectivity of $P|_K$ imply
$\phi_f(t,y)=y$ for all $t\ge0$, and hence $f(y)=0$.
The converse follows immediately from the definition of $F$.
Thus $P\omega(x)$ contains finitely many equilibria and also
contains non-equilibrium points. In particular, it is not a
periodic orbit, and $z_0$ is neither an equilibrium nor a
periodic point.

Applying \cite[Theorem~7.3]{Ciesielski2012} to the planar flow
$\psi$, we conclude that, for every non-equilibrium point
$z\in P\omega(x)$, both $\alpha_\psi(z)$ and $\omega_\psi(z)$
are nonempty compact connected sets consisting only of
equilibria. These limit sets lie in $P\omega(x)$, which contains
only finitely many equilibria; hence each is a single
equilibrium. The conjugacy between the flows on $\omega(x)$
and $P\omega(x)$ transfers this conclusion to the original
system. Therefore every non-equilibrium point of $\omega(x)$
lies on a nonconstant complete orbit whose $\alpha$- and
$\omega$-limit sets are individual equilibria in $\omega(x)$.
This proves (iii).
\end{proof}

Figure~\ref{fig:limit-sets} illustrates the alternatives in
Lemma~\ref{Compact localization and limit-set alternative}(iii).
The third panel illustrates the alternative involving equilibria
and their connections. We will use the strict decrease established
below to rule out this alternative for the vector fields under consideration.

\begin{figure}[htbp]
\centering
\begin{tikzpicture}[font=\small]
\begin{scope}
  \fill[diagramblue] (0,0) circle (2pt);
  \node[above right] at (0,0) {$e$};
  \draw[orbit] (-1.3,.65).. controls (-.75,.85) and (-.4,.3)..(-.13,.08);
  \draw[orbit] (1.3,-.65).. controls (.75,-.85) and (.4,-.3)..(.13,-.08);
  \draw[orbit] (-.8,-.85).. controls (-.3,-.9) and (-.2,-.4)..(-.05,-.12);
  \node at (0,-1.47) {(a) An equilibrium};
\end{scope}
\begin{scope}[xshift=4.55cm]
  \draw[diagramblue,line width=.9pt] (0,0) ellipse (1.2 and .82);
  \draw[orbit] (1.2,0) arc[start angle=0,end angle=90,x radius=1.2,y radius=.82];
  \draw[orbit] (-1.2,0) arc[start angle=180,end angle=270,x radius=1.2,y radius=.82];
  \node at (0,0) {$\gamma$};
  \node at (0,-1.47) {(b) A periodic orbit};
\end{scope}
\begin{scope}[xshift=9.2cm]
  \coordinate (e1) at (-1.2,0);
  \coordinate (e2) at (1.2,0);
  \draw[orbit] (e1).. controls (-.6,1.12) and (.6,1.12)..(e2);
  \draw[orbit] (e2).. controls (.6,-1.12) and (-.6,-1.12)..(e1);
  \fill[diagramblue] (e1) circle (2pt);
  \fill[diagramblue] (e2) circle (2pt);
  \node[left] at (e1) {$e_1$};
  \node[right] at (e2) {$e_2$};
  \node[align=center] at (0,-1.47) {(c) Equilibria and connections};
\end{scope}
\end{tikzpicture}
\caption{Schematic planar limit sets when the number of equilibria
is finite. Panel (c) shows one possible connection configuration in
the planar classification, not a configuration asserted to occur
in the generic class. Lemma~\ref{Strict decrease along connections}
rules out such directed cycles, and Lemma~\ref{exactly one} excludes
the mixed alternative altogether.}
\label{fig:limit-sets}
\end{figure}
\FloatBarrier

\subsection{Connections, nonwandering points, and finiteness}\label{npf}

We next show that an integer-valued rank decreases along connections
between critical elements. This will allow us to characterize the
nonwandering set and prove that only finitely many periodic orbits exist.

\begin{lem}\label{Exit compactness}
Let $\sigma\subset B(a,R)$ be a hyperbolic critical element. There
exists an open neighborhood $U$ of $\sigma$, with
$\sigma\subset U\Subset B(a,R)$, such that if the negative orbit through $y\in\overline U$ is defined
for all $t\le0$ and $\phi_f(t,y)\in\overline U$ for all $t\le0$,
then
\begin{equation}\label{dist0}
\operatorname{dist}(\phi_f(t,y),\sigma)\to 0
\qquad (t\to-\infty).
\end{equation}

Suppose that, for each $k\in\mathbb N$, there exist $s_k<\tau_k$ and a solution
$u_k\in C^1([s_k,\tau_k],\Omega)$ of $\dot u=f(u)$ such that
$u_k([s_k,\tau_k])\subset\overline{B(a,R)}$,
$\operatorname{dist}(u_k(s_k),\sigma)\to0$ as $k\to\infty$,
$u_k([s_k,\tau_k))\subset U$, and $u_k(\tau_k)\in\partial U$.
Then, after passing to a subsequence, there exists $y\in\partial U$
with the following properties:
\begin{enumerate}
\item [(i)] $u_k(\tau_k)\to y,$ as $k\to \infty;$
\item [(ii)] for every $T>0$,
\begin{equation}\label{limsupuk}
\lim_{k\to\infty}
\sup_{-T\le t\le0}
\bigl\|
u_k(\tau_k+t)-\phi_f(t,y)
\bigr\|
=0,
\end{equation}
and $\alpha(y)=\sigma$;

\item [(iii)] the orbit through $y$ is defined for all $t\in\mathbb R$ and
is contained in $\overline{B(a,R)}$. Moreover, if $u_k([s_k,\tau_k])\subset K$ for all $k\in\mathbb N$,
where $K$ is a closed invariant set, then $y\in K$.
\end{enumerate}
\end{lem}

\begin{proof}
The local stable and unstable manifold theorem allows $U$ to be chosen
with property~(\ref{dist0}); see \cite[Theorem~4.1]{CC} and
\cite[Appendix~C]{MX}.

We first show that $\tau_k-s_k\to\infty$.
Otherwise, after passing to a subsequence,
$\tau_k-s_k\to T<\infty$. Choose $z_k\in\sigma$ such that $\|u_k(s_k)-z_k\|
=
\operatorname{dist}(u_k(s_k),\sigma)$.
Since $\sigma$ is compact, after passing to a further subsequence we
may assume that $z_k\to z\in\sigma$. Hence $u_k(s_k)\to z$, and
continuous dependence on the initial value and time gives
\[
u_k(\tau_k)
=
\phi_f(\tau_k-s_k,u_k(s_k))
\to
\phi_f(T,z)\in\sigma,
\]
contradicting $u_k(\tau_k)\in\partial U$. Thus
$\tau_k-s_k\to\infty$.

Compactness of $\partial U$ gives a subsequence such that
$u_k(\tau_k)\to y\in\partial U$, proving~(i).

Fix $N\in\mathbb N$. For all sufficiently large $k$,
$[\tau_k-N,\tau_k]\subset[s_k,\tau_k]$. Set $v_k(t):=u_k(\tau_k+t)$ for all $-N\le t\le0$.
Then $v_k([-N,0])\subset\overline U$. Since $f$ is bounded on the
compact set $\overline U$, the family $\{v_k\}$ is uniformly bounded
and equicontinuous. The Arzel\`a--Ascoli theorem and a diagonal argument over
$N=1,2,\ldots$ give a further subsequence converging uniformly on
every compact subinterval of $(-\infty,0]$ to a function $v$. Passing to the integral equations satisfied by
$v_k$ shows that $v$ solves $\dot v=f(v)$ with $v(0)=y$.
Hence, by uniqueness, $v(t)=\phi_f(t,y)$ for all $t\le 0$,
and~(\ref{limsupuk}) follows.
Moreover, $\phi_f(t,y)\in\overline U$ for all $t\le0$.
By~(\ref{dist0}),
\[
\operatorname{dist}(\phi_f(t,y),\sigma)\to 0
\qquad (t\to-\infty),
\]
and therefore $\alpha(y)=\sigma$. This proves~(ii).

Finally, since $y\in\overline U\subset B(a,R)$,
Lemma~\ref{Compact localization and limit-set alternative}(i) gives
forward existence for all $t\ge0$ and positive invariance of
$\overline{B(a,R)}$. Thus the orbit through $y$ is complete and
contained in $\overline{B(a,R)}$. If
$u_k([s_k,\tau_k])\subset K$ for every $k$, then
$u_k(\tau_k)\in K$; closedness of $K$ then gives $y\in K$. This proves~(iii).
\end{proof}

\begin{lem}[Strict decrease along connections]\label{Strict decrease along connections}
Let $f\in\mathcal O_{\mathrm{KS}}$, and let $\sigma^-$ and $\sigma^+$
be distinct critical elements of $\dot x=f(x)$. Define
\begin{equation}
\rho(\sigma)
=
\begin{cases}
2i(\sigma), & \sigma \text{ is an equilibrium},\\
2i(\sigma)+1, & \sigma \text{ is a periodic orbit}.
\end{cases}
\end{equation}
If there exists a nonconstant connection from $\sigma^-$ to $\sigma^+$, then
\begin{equation}\label{decrease}
\rho(\sigma^-)>\rho(\sigma^+).
\end{equation}
Consequently, every sequence of successive connections contains at
most $2n$ connections, and no cycle of connecting orbits can occur.
\end{lem}

\begin{proof}
By Proposition~\ref{Relative generic Kupka--Smale property}, all critical elements are hyperbolic and the
connection is transverse. If $q$ lies on the connecting orbit, then
$0\neq f(q)\in T_qW^u(\sigma^-)\cap T_qW^s(\sigma^+)$ and $T_qW^u(\sigma^-)+T_qW^s(\sigma^+)=\mathbb R^n$.
Hence
\begin{equation}\label{RN}
    1\le
\dim W^u(\sigma^-)+\dim W^s(\sigma^+)-n.
\end{equation}

If $\sigma^-$ is an equilibrium, (\ref{RN}) gives
$i(\sigma^-)>i(\sigma^+)$, whether $\sigma^+$ is an equilibrium or a
periodic orbit, and thus~(\ref{decrease}) follows. If $\sigma^-$ is periodic and
$\sigma^+$ is an equilibrium, then (\ref{RN}) gives
$i(\sigma^-)\ge i(\sigma^+)$, so
$\rho(\sigma^-)=2i(\sigma^-)+1>2i(\sigma^+)=\rho(\sigma^+)$.

It remains to consider the case where both critical elements are
periodic orbits. Let $h^-\ge h^+$ be as in Proposition~\ref{Hyperbolic critical elements prop}. By
Theorem~\ref{Floquet Theory} and Proposition~\ref{Hyperbolic critical elements prop},
$2h^\pm-2\le i(\sigma^\pm)\le2h^\pm-1$. If $h^->h^+$, then
$i(\sigma^-)>i(\sigma^+)$, and~(\ref{decrease}) follows.

Suppose $h^-=h^+=h$. If
$h<(\widetilde n+1)/2$, Proposition~\ref{Hyperbolic critical elements prop}(i)(ii) gives
$\dim W^u(\sigma^-)=2h$ and
$\dim W^s(\sigma^+)=n-2h+2$. Hence
$i(\sigma^-)=2h-1$ and $i(\sigma^+)=2h-2$, so~(\ref{decrease}) holds.

Now let $h=(\widetilde n+1)/2$. If $n$ is even, then
$\widetilde n=n-1$ and $h=n/2$. Proposition~\ref{Hyperbolic critical elements prop} gives
$\dim W^u(\sigma^-)=n$ and $\dim W^s(\sigma^+)=2$, whence
$i(\sigma^-)=n-1$ and $i(\sigma^+)=n-2$. Thus~(\ref{decrease}) again holds.
If $n$ is odd, then $\widetilde n=n$ and $h=(n+1)/2$.
Proposition~\ref{Hyperbolic critical elements prop}(ii) gives $\dim W^s(\sigma^+)=1$. Since the local
stable manifold of $\sigma^+$ is then the periodic orbit itself, no
nonconstant incoming connection to $\sigma^+$ can occur.

Thus~(\ref{decrease}) holds for every nonconstant connection. Finally,
$\rho(\sigma)$ is integer-valued and $0\le\rho(\sigma)\le2n$.
Therefore $\rho$ can strictly decrease along at most $2n$ successive
connections, and no cycle of connecting orbits can occur.
\end{proof}

Assigning different ranks to equilibria and periodic orbits is useful
even when their Morse indices agree: a connection from a periodic
orbit to an equilibrium may preserve the Morse index, but it
still lowers $\rho$ by one. Figure~\ref{fig:rank-order} illustrates
this ordering and the obstruction to a directed cycle.

\begin{figure}[htbp]
\centering
\begin{tikzpicture}[font=\small,>=Stealth]
  \draw[->,black!60] (-1.8,-.4)--(-1.8,2.65) node[above] {$\rho$};
  \foreach \y/\lab in {0/{2k},1.1/{2k+1},2.2/{2k+2}}{
    \draw[black!50] (-1.88,\y)--(-1.72,\y);
    \node[left] at (-1.93,\y) {$\lab$};
  }
  \coordinate (a) at (0,2.2);
  \coordinate (b) at (3.1,1.1);
  \coordinate (c) at (6.2,0);
  \draw[->,diagramblue,line width=.9pt,shorten >=5pt,shorten <=5pt] (a)--(b);
  \draw[->,diagramblue,line width=.9pt,shorten >=5pt,shorten <=5pt] (b)--(c);
  \fill[diagramblue] (a) circle (2.4pt);
  \draw[diagramblue,line width=.9pt] (b) circle (.13);
  \fill[diagramblue] (c) circle (2.4pt);
  \node[above right,align=left] at (a) {$e^-$\\$i(e^-)=k+1$};
  \node[above right,align=left] at (b) {$\gamma$\\$i(\gamma)=k$};
  \node[right,align=left] at (c) {$e^+$\\$i(e^+)=k$};
  \draw[->,diagramred,dashed] (6.05,-.27)
    .. controls (3.8,-1.35) and (.05,-1.1)..(-.25,1.96);
  \node[fill=white,text=diagramred,inner sep=2pt] at (2.2,-.83)
    {a return would increase $\rho$};
\end{tikzpicture}
\caption{The rank $\rho(e)=2i(e)$ for equilibria and
$\rho(\gamma)=2i(\gamma)+1$ for periodic orbits strictly decreases
along each nonconstant connection. The displayed indices illustrate
the ordering, without asserting that this particular chain is
realized. The dashed return is incompatible with strict decrease.}
\label{fig:rank-order}
\end{figure}
\FloatBarrier

\begin{lem}\label{exactly one}
For every $x\in\overline{B(a,R)}$, the set $\omega(x)$ consists of
exactly one critical element.
\end{lem}

\begin{proof}
If $\omega(x)$ contains no equilibrium, Lemma~\ref{Compact localization and limit-set alternative}(iii) implies that
$\omega(x)$ is a nonconstant periodic orbit.

Suppose that $\omega(x)$ contains equilibria. Lemma~\ref{Compact localization and limit-set alternative}(ii) allows us
to choose an equilibrium $e\in\omega(x)$ with minimal Morse index
among the equilibria in $\omega(x)$. We claim that $\omega(x)=\{e\}$.

Assume, to the contrary, that $\omega(x)\neq\{e\}$, and choose
$z\in\omega(x)\setminus\{e\}$. Since $e$ is hyperbolic and hence
isolated, we may choose the neighborhood $U$ in Lemma~\ref{Exit compactness} sufficiently
small so that $e$ is the only equilibrium in $\overline U$ and
$z\notin\overline U$. In particular, $\omega(x)\not\subset\overline U$.
Since $e\in\omega(x)$, there exists $s_k\to\infty$ such that
$\phi_f(s_k,x)\to e$. For all sufficiently large $k$,
$\phi_f(s_k,x)\in U$. Since $\omega(x)\not\subset\overline U$, the positive orbit cannot remain in
$U$ for all $t\ge s_k$; let $\tau_k>s_k$ be its first exit time from
$U$. Then $\phi_f(t,x)\in U\,(s_k\le t<\tau_k)$ and $\phi_f(\tau_k,x)\in\partial U$.

Applying Lemma~\ref{Exit compactness} to
$u_k(t)=\phi_f(t,x)$, $t\in[s_k,\tau_k]$, and passing to a subsequence,
we obtain $\phi_f(\tau_k,x)\to y\in\partial U$ and $\alpha(y)=e$.
Since $\tau_k\ge s_k\to\infty$, it follows that $y\in\omega(x)$.
Moreover, $e$ is the only equilibrium in $\overline U$, while
$y\in\partial U$; hence $y$ is not an equilibrium.

By Lemma~\ref{Compact localization and limit-set alternative}(iii), the complete orbit through $y$ is therefore a
nonconstant connecting orbit from $e$ to some equilibrium
$e'\in\omega(x)$. Proposition~\ref{Relative generic Kupka--Smale property}
excludes homoclinic orbits, so $e'\ne e$.
Lemma~\ref{Strict decrease along connections} gives $\rho(e)>\rho(e')$. Since both $e$ and $e'$ are
equilibria, $\rho(e)=2i(e)$ and $\rho(e')=2i(e')$, and hence $i(e)>i(e')$,
contradicting the minimality of $i(e)$. Therefore
$\omega(x)=\{e\}$.

Thus, in all cases, $\omega(x)$ consists of exactly one critical
element.
\end{proof}

\begin{lem}[Nonwandering points]\label{NP}
Every nonwandering point of $f$ belongs to a critical element.
\end{lem}

\begin{proof}
Let $x$ be a nonwandering point. By Lemma~\ref{Compact localization and limit-set alternative}(ii), $x\in B(a,R)$.
By Lemma~\ref{exactly one}, there exists a critical element $\sigma_0$ such that
$\omega(x)=\sigma_0$.

Suppose that $x$ belongs to no critical element. Orbit segments that
return near $x$ will allow us to construct a chain exceeding the length
bound in Lemma~\ref{Strict decrease along connections}.
Since $\sigma_0=\omega(x)$, for every $k\ge1$ we may choose $s_k>k$
such that $\operatorname{dist}\bigl(\phi_f(s_k,x),\sigma_0\bigr)
<1/{2k}$.
For this fixed $s_k$, continuous dependence on the initial value gives a
neighborhood $W_k$ of $x$ such that
$W_k\subset B(x,1/k)\cap B(a,R)$ and
\[
\|\phi_f(s_k,z)-\phi_f(s_k,x)\|
<\frac1{2k},
\qquad z\in W_k.
\]
Hence
$\operatorname{dist}(\phi_f(s_k,z),\sigma_0)<1/k$ for every
$z\in W_k$.

Since $x$ is nonwandering, there exist $x_k\in W_k$ and $T_k>2s_k$
such that $\phi_f(T_k,x_k)\in W_k$. Set
\[
L_k:=T_k-s_k,
\qquad
v_k(t):=\phi_f(s_k+t,x_k),
\quad 0\le t\le L_k.
\]
By Lemma~\ref{Compact localization and limit-set alternative}(i),
$v_k([0,L_k])\subset\overline{B(a,R)}$. Moreover,
\begin{equation}\label{Lk}
\operatorname{dist}\bigl(v_k(0),\sigma_0\bigr)\to0,
\qquad
L_k\to\infty,
\qquad
v_k(L_k)\to x.
\end{equation}

For a critical element $\sigma$, we say that a sequence of solution
segments $v_k:[0,L_k]\to\overline{B(a,R)}$ has property $(R_\sigma)$
if the three conditions in~(\ref{Lk}) hold with $\sigma$ in place of
$\sigma_0$. Thus the family constructed above has property
$(R_{\sigma_0})$.

We now use these segments to construct successive connections.
Suppose that
\begin{equation}\label{chain}
\sigma_0\longrightarrow\sigma_1\longrightarrow\cdots
\longrightarrow\sigma_j
\end{equation}
is a finite chain of nonconstant connections and that a sequence of
solution segments $v_k:[0,L_k]\to\overline{B(a,R)}$ has property
$(R_{\sigma_j})$.

We claim that these segments either produce a cycle of nonconstant
connections or extend the chain to a critical element $\sigma_{j+1}$.
In the latter case, a subsequence of suitable terminal subsegments
has property $(R_{\sigma_{j+1}})$.

Since $x$ belongs to no critical element, $x\notin\sigma_j$.
Choose a neighborhood $U_j$ of $\sigma_j$ as in Lemma~\ref{Exit compactness},
sufficiently small that $x\notin\overline U_j$. By property
$(R_{\sigma_j})$, for all sufficiently large $k$, $v_k(0)\in U_j$ and $v_k(L_k)\notin\overline U_j$.
Let $\tau_k\in(0,L_k)$ be the first exit time from $U_j$. Then
$v_k([0,\tau_k))\subset U_j$ and $v_k(\tau_k)\in\partial U_j$.
By Lemma~\ref{Exit compactness}, after passing to a subsequence,
\begin{equation}\label{Uj}
v_k(\tau_k)\to y_j\in\partial U_j,
\qquad
\alpha(y_j)=\sigma_j.
\end{equation}

Set $d_k:=L_k-\tau_k$. If $\{d_k\}$ has a bounded subsequence, then,
passing to this subsequence, we may assume $d_k\to d\ge0$. Since $v_k(L_k)=\phi_f(d_k,v_k(\tau_k))$,
while $v_k(L_k)\to x$ and $v_k(\tau_k)\to y_j$, continuous dependence
on time and initial data yields $\phi_f(d,y_j)=x$. Hence $\omega(y_j)=\omega(x)=\sigma_0$.
Together with $\alpha(y_j)=\sigma_j$, this gives a nonconstant
connection $\sigma_j\longrightarrow\sigma_0$. If $\sigma_j=\sigma_0$
(in particular, when $j=0$), this connection is homoclinic, contrary
to Proposition~\ref{Relative generic Kupka--Smale property}. Otherwise,
combining it with~(\ref{chain}) produces a cycle, contradicting
Lemma~\ref{Strict decrease along connections}.
Thus $\{d_k\}$ has no bounded subsequence. Since $d_k\ge0$, $d_k\to\infty$.

By Lemma~\ref{exactly one}, $\omega(y_j)=\sigma_{j+1}$ for some critical element
$\sigma_{j+1}$. Since $y_j\in\partial U_j$ whereas
$\sigma_j\subset U_j$, the orbit through $y_j$ is nonconstant. Hence
(\ref{Uj}) gives a nonconstant connection
$\sigma_j\longrightarrow\sigma_{j+1}$.
Proposition~\ref{Relative generic Kupka--Smale property} excludes
homoclinic orbits, so $\sigma_{j+1}\ne\sigma_j$. If $\sigma_{j+1}$
coincides with an earlier element of~(\ref{chain}), the resulting
cycle contradicts Lemma~\ref{Strict decrease along connections}.

It remains to recover property $(R_{\sigma_{j+1}})$. Since
$\omega(y_j)=\sigma_{j+1}$, for every $m\ge1$ choose $r_m>m$ such that
\begin{equation}\label{distrm}
\operatorname{dist}\bigl(\phi_f(r_m,y_j),\sigma_{j+1}\bigr)
<\frac1{2m}.
\end{equation}
For each fixed $m$, (\ref{Uj}) and continuous dependence imply $\phi_f(r_m,v_k(\tau_k))
\to
\phi_f(r_m,y_j)$ as $k\to\infty$.
Using $d_k\to\infty$, choose a strictly increasing sequence $k_m$ such
that $d_{k_m}>r_m+m$ and
\begin{equation}\label{1/2m}
\left\|
\phi_f(r_m,v_{k_m}(\tau_{k_m}))
-\phi_f(r_m,y_j)
\right\|
<\frac1{2m}.
\end{equation}
Define $\widetilde L_m:=d_{k_m}-r_m$ and
\[
\widetilde v_m(t)
:=
v_{k_m}(\tau_{k_m}+r_m+t),
\quad
0\le t\le\widetilde L_m.
\]
Then (\ref{distrm}) and (\ref{1/2m}) give $\operatorname{dist}\bigl(\widetilde v_m(0),\sigma_{j+1}\bigr)
<1/m$,
while $d_{k_m}>r_m+m$ gives $\widetilde L_m>m$. Moreover,
\[
\widetilde v_m(\widetilde L_m)
=
v_{k_m}(L_{k_m})\to x\quad\text{ as }k_m\to \infty.
\]
Thus the three conditions in~(\ref{Lk}) hold for
$\{\widetilde v_m\}$ with $\sigma_{j+1}$ in place of $\sigma_0$.
Hence $\{\widetilde v_m\}$ has property $(R_{\sigma_{j+1}})$,
proving the claim.

The initial family has property $(R_{\sigma_0})$, so we can start
this construction with $j=0$. Lemma~\ref{Strict decrease along connections}
excludes cycles. Each repetition therefore extends the chain by a
nonconstant connection and provides a new family with the corresponding
return property. After $2n+1$ repetitions, we obtain a chain of
$2n+1$ successive nonconstant connections. This contradicts Lemma~\ref{Strict decrease along connections}, which states
that every chain of successive nonconstant connections contains at
most $2n$ connections.

This contradiction proves that every nonwandering point belongs to a critical element.
\end{proof}

\begin{prop}[Morse--Smale property]\label{MSP}
For every $f\in\mathcal O_{\mathrm{KS}}$, the system has only finitely many
critical elements, and its nonwandering set in $\Omega$ consists precisely
of points belonging to critical elements. Consequently, $f$ is
Morse--Smale in the sense of Theorem~\ref{GenericityofMorse–Smaleproperty}.
\end{prop}

\begin{proof}
We first prove that there are only finitely many periodic orbits.

Suppose, to the contrary, that $\{\gamma_k\}_{k\ge1}$ is a sequence of
distinct periodic orbits. Choose $x_k\in\gamma_k$. By Lemma~\ref{Compact localization and limit-set alternative}(ii), all
periodic orbits lie in $B(a,R)$. Since
$\overline{B(a,R)}\Subset\Omega$, after passing to a subsequence we may
assume that $x_k\to x\in\overline{B(a,R)}$. Each $x_k$ is nonwandering. Since the nonwandering set is closed in
$\Omega$, $x$ is also nonwandering. By Lemma~\ref{NP}, $x$ belongs to a
critical element $\sigma$. Choose a neighborhood $U$ of $\sigma$ as in Lemma~\ref{Exit compactness}. Since
$x_k\to x\in\sigma\subset U$, we have $x_k\in U$ for all sufficiently
large $k$. Since the periodic orbits $\gamma_k$ are distinct, we may discard
at most one of them and assume that $\gamma_k\neq\sigma$ for every $k$.

No $\gamma_k$ can be contained in $U$. Indeed, if
$\gamma_k\subset U$, then for any $z\in\gamma_k$ the complete negative
orbit through $z$ remains in $\overline U$. By the defining property of
$U$ in Lemma~\ref{Exit compactness}, $\operatorname{dist}(\phi_f(t,z),\sigma)\to0$ as $t\to-\infty$.
Hence $\alpha(z)=\sigma$. On the other hand, periodicity gives
$\alpha(z)=\gamma_k$, and therefore $\gamma_k=\sigma$, a contradiction.

Thus the positive orbit through $x_k$ must leave $U$. Let $\tau_k>0$
be its first exit time from $U$. Applying Lemma~\ref{Exit compactness} to $u_k(t)=\phi_f(t,x_k)\,(0\le t\le\tau_k)$
and using $x_k\to x\in\sigma$, we obtain, after passing to a subsequence,
\begin{equation}\label{fyU}
\phi_f(\tau_k,x_k)\to y\in\partial U,
\qquad
\alpha(y)=\sigma.
\end{equation}
For every $k$, the point $\phi_f(\tau_k,x_k)$ lies on the periodic
orbit $\gamma_k$, and hence is nonwandering. Since the nonwandering set
is closed, $y$ is nonwandering. Lemma~\ref{NP} therefore implies that $y$
belongs to some critical element $\sigma'$. Since $y\in\sigma'$, we have $\alpha(y)=\sigma'$. Comparing this with
(\ref{fyU}) gives $\sigma'=\sigma$. Hence $y\in\sigma\subset U$, contradicting
$y\in\partial U$. Therefore the number of periodic orbits is finite.

There are finitely many equilibria by Lemma~\ref{Compact localization and limit-set alternative}(ii), and finitely many periodic orbits by the preceding argument. Hence $f$ has only finitely many
critical elements. Every point belonging to a critical element is nonwandering, while
Lemma~\ref{NP} gives the converse. Thus the nonwandering set is exactly the union of the critical elements.

Finally, Proposition~\ref{Relative generic Kupka--Smale property} shows that every critical element is hyperbolic
and that the stable and unstable manifolds of any two critical elements
intersect transversely. Thus every $f\in\mathcal O_{\mathrm{KS}}$
satisfies all the Morse--Smale conditions in Theorem~\ref{GenericityofMorse–Smaleproperty}. Since
$\mathcal O_{\mathrm{KS}}$ is a dense $G_\delta$ subset of
$(\mathcal L_d^-,\iota_d)$ by Proposition~\ref{Relative generic Kupka--Smale property},
setting $\mathcal O_M:=\mathcal O_{\mathrm{KS}}$ proves
Theorem~\ref{GenericityofMorse–Smaleproperty}.
\end{proof}
	
\refstepcounter{appsec}
\section*{Appendix A. Sard--Smale theorem}
\label{app:sard-smale}

\setcounter{defn}{0}
\renewcommand{\thedefn}{A.\arabic{defn}}

\setcounter{thm}{0}
\renewcommand{\thethm}{A.\arabic{thm}}

\setcounter{rk}{0}
\renewcommand{\therk}{A.\arabic{rk}}

\def\theHdefn{appendix.A.\arabic{defn}}
\def\theHthm{appendix.A.\arabic{thm}}
\def\theHrk{appendix.A.\arabic{rk}}
	
	Let $\mathcal X$ and $\mathcal Y$ be Banach spaces. A bounded linear
operator $L:\mathcal X\to\mathcal Y$ is Fredholm if its range is closed
and both $\dim\ker L$ and $\operatorname{codim}\operatorname{Ran}L$ are
finite. Its index is
\[
\operatorname{ind}L=\dim\ker L-\operatorname{codim}\operatorname{Ran}L.
\]
For a $C^1$ map $f:M\to M'$ between Banach manifolds, a point $x\in M$
is regular if $Df(x)$ is surjective and its kernel has a closed
complement in $T_xM$. A point $y\in M'$ is a regular value if every
point of $f^{-1}(y)$ is regular; see \cite[Definition~5.7]{H}.

We use the following version of the Sard--Smale theorem from \cite{H}.
For related versions and applications, see \cite{Q,ST,BP,BG}.

	\begin{thm}\label{SS-1}
		Let $k\ge1$ be an integer, let $X,Y,Z$ be $C^k$ Banach manifolds,
and let $U\subset X$ and $V\subset Y$ be open. Let $f:U\times V\to Z$
be of class $C^k$, and fix $\zeta\in Z$. Assume that the following conditions hold:
		\begin{enumerate}
			\item[{\rm(i)}] $\forall \left(x,y\right)\in f^{-1}\left(\zeta\right)$, $D_xf\left(x,y\right):T_xX\to T_\zeta Z$ is Fredholm with index strictly 
			less than $k$;
			
			\item[{\rm(ii)}] $\forall \left(x,y\right)\in f^{-1}\left(\zeta\right)$, 
			$Df\left(x,y\right)=\left(D_xf\left(x,y\right),D_yf\left(x,y\right)\right):T_xX \times T_yY \to T_\zeta Z$ is surjective;
			
			\item[{\rm(iii)}] one of the following conditions holds:
			\begin{enumerate}
				\item[{\rm(a)}] $X$ and $Y$ are separable metric spaces;
				
				\item[{\rm(b)}] the projection
\[
\pi:f^{-1}(\zeta)\to Y,\qquad \pi(x,y)=y,
\]
is $\sigma$-proper: there exist sets $M_i$ with
$f^{-1}(\zeta)=\bigcup_{i=1}^\infty M_i$ such that each restriction
$\pi|_{M_i}:M_i\to Y$ is proper. Equivalently, whenever
$(x_j,y_j)\in M_i$ and $y_j$ converges to a point of $Y$,
the sequence $(x_j,y_j)$ has a subsequence converging in $M_i$.

			\end{enumerate}
		\end{enumerate}
		Then the set $\Theta=\{y \in V \mid \zeta$ is a regular value of $f(\cdot,y)\}$ contains a countable intersection of open dense subsets of $V$ and is therefore dense in $V$.
	\end{thm}
	
	\begin{rk}
		\rm{Under the hypotheses of Theorem~\ref{SS-1}, condition (iii)(b)
holds if $f^{-1}(\zeta)$ is Lindel\"of, and in particular if it is a
separable metric space. This condition also holds when $X$ and $Y$ are
separable metric spaces. Every finite-dimensional kernel has a closed
complement; see \cite[Lemma~5.8]{H}. Thus the splitting condition is
automatic when the derivative has finite-dimensional domain.}
\end{rk} 

    {When both the state and target manifolds are finite-dimensional, the Fredholm condition is automatic,
and the index is $\operatorname{dim}X-\operatorname{dim}Z$. The following version is used in Section~\ref{GTC}; see \cite[p.~48]{AR}.}
	
	\begin{thm} \label{SS-2}
		Let $k\ge1$ be an integer, let $X,Y,Z$ be $C^k$ Banach manifolds,
and let $U\subset X$ and $V\subset Y$ be open. Let $f:U\times V\to Z$
be of class $C^k$, and fix $\zeta\in Z$. Assume that the following conditions hold:
		\begin{enumerate}
			\item[{\rm(i)}] $\dim X=n<\infty$, $\dim Z=q<\infty$;
			
			\item[{\rm(ii)}] $k>\operatorname{max}\left\{n-q,0\right\}$;
			
			\item[{\rm(iii)}] $\forall \left(x,y\right)\in f^{-1}\left(\zeta\right)$, $Df\left(x,y\right):T_xX \times T_yY \to T_\zeta Z$ is surjective;
			
			\item[{\rm(iv)}] $X$ and $Y$ are separable metric spaces.
		\end{enumerate}
		Then the set $\Theta=\{y \in V \mid \zeta$ is a regular value of $f(\cdot,y)\}$ is generic, and hence dense, in $V$.
	\end{thm}

	\setcounter{defn}{0} 
	\renewcommand{\thedefn}{B.\arabic{defn}}
	\setcounter{lem}{0} 
	\renewcommand{\thelem}{B.\arabic{lem}}
	\setcounter{thm}{0} 
	\renewcommand{\thethm}{B.\arabic{thm}}
	\setcounter{prop}{0} 
	\renewcommand{\theprop}{B.\arabic{prop}}
	\setcounter{equation}{0} 
	\renewcommand{\theequation}{B.\arabic{equation}}
	\setcounter{cor}{0} 
	\renewcommand{\thecor}{B.\arabic{cor}}
	\setcounter{rk}{0} 
	\renewcommand{\therk}{B.\arabic{rk}}
	\section*{Appendix B. Exponential dichotomies and applications}
	\setcounter{subsection}{0}
\renewcommand{\thesubsection}{B.\arabic{subsection}}
\renewcommand{\theHsubsection}{appendix.B.\arabic{subsection}}
\subsection{Variational equations and half-line dichotomies}
\label{subsec:half-line-dichotomies}
Let $f\in C^1(\Omega,\mathbb R^n)$ and let $\widetilde u$ be a
complete trajectory with
$\lim_{t\to\pm\infty}\widetilde u(t)=e^\pm$, where $e^\pm$ are
hyperbolic equilibria. A constant trajectory at a hyperbolic
equilibrium is allowed. The closure of $\widetilde u(\mathbb R)$ is
compact in $\Omega$. Set \(A(t)=Df(\widetilde u(t))\), and let \(S_A(t,s)\) denote the corresponding variational evolution operator, with the convention introduced in Section \ref{Spectral theory, invariant spaces and nested invariant cones}. This operator is invertible for all \(s,t\in\mathbb R\).

Fix $\tau>0$, set $T_q= S_A((q+1)\tau,q\tau)$, and write
$T(q,m)= S_A(q\tau,m\tau)$. Since $A$ is bounded, there exists $C_\tau>0$ such that $\|S_A(t,s)\|\leq C_\tau$ whenever $|t-s|\leq\tau$.
In particular, both $T_q$ and $T_q^{-1}$ are uniformly bounded. Define the
bounded linear difference operator
\begin{equation}
L:\ell^\infty(\mathbb Z,\mathbb R^n)
\longrightarrow
\ell^\infty(\mathbb Z,\mathbb R^n),
\qquad
(LY)(q)=Y(q+1)-T_qY(q).
\label{eq:B-difference-operator}
\end{equation}

The linearized equation along $\widetilde u$ is
\begin{equation}
\dot v(r)
=
Df(\widetilde u(r))v(r)
\equiv A(r)v(r),
\qquad r\in\mathbb R.
\label{eq:B-variational}
\end{equation}
For $s\in\mathbb R$ and $v_s\in\mathbb R^n$, its solution satisfying
$v(s)=v_s$ is given by $v(r)= S_A(r,s)v_s$.

The adjoint linearized equation is
\begin{equation}
\dot\varphi(r)
=
-A(r)^*\varphi(r),
\qquad r\in\mathbb R.
\label{eq:B-adjoint}
\end{equation}
For any solutions $v$ of \eqref{eq:B-variational} and $\varphi$ of
\eqref{eq:B-adjoint},
\[
\frac{d}{dr}\langle\varphi(r),v(r)\rangle
=
\langle-A(r)^*\varphi(r),v(r)\rangle
+
\langle\varphi(r),A(r)v(r)\rangle
=
0.
\]
Hence $\langle\varphi(r),v(r)\rangle$ is independent of $r$.
Thus, for any $s,t\in\mathbb R$, $\varphi(s)= S_A(t,s)^*\varphi(t).$
In particular,
\[
\varphi(q\tau)
=
T_q^*\varphi((q+1)\tau),
\qquad q\in\mathbb Z.
\]

\begin{lem}
\label{lem:B-half-line-dichotomy}
There exist invariant projections $P^-(q)$ on
$\mathbb Z^-=\{q\leq0\}$ and $P^+(q)$ on
$\mathbb Z^+=\{q\geq0\}$, of ranks $i(e^-)$ and $i(e^+)$,
respectively, and constants $C,\eta>0$ such that the following estimates hold on each half-line:
\begin{align*}
\|T(q,m)(I-P^\pm(m))\|
&\leq Ce^{-\eta(q-m)\tau},\quad q\geq m,\\
\|T(q,m)P^\pm(m)\|
&\leq Ce^{-\eta(m-q)\tau},\quad q\leq m.
\end{align*}
Here invariance means
$T(q,m)P^\pm(m)=P^\pm(q)T(q,m)$. The projections are onto unstable
subspaces.
\end{lem}

\begin{proof}
Set $A_\pm=Df(e^\pm)$. Since $A(t)\to A_\pm$ as
$t\to\pm\infty$, the variation-of-constants formula and the uniform bound for the evolution
over intervals of length $\tau$ give
\[
\|T_q-e^{\tau A_\pm}\|
\leq
C\tau
\sup_{s\in[q\tau,(q+1)\tau]}
\|A(s)-A_\pm\|
\to 0
\qquad (q\to\pm\infty).
\]
The constant matrices $e^{\tau A_\pm}$ are hyperbolic. Their stable
and unstable spectral splittings give discrete exponential
dichotomies. The roughness of exponential dichotomies yields dichotomies of the same
ranks for $T_q$ at sufficiently large positive and negative indices; see
\cite[Corollary~1.9]{H2}. Since every $T_q$ is invertible, transporting
the projections across the finitely many remaining steps extends
the dichotomies to the two half-lines, after increasing $C$ if necessary. The
factor $\tau$ in the estimates converts the discrete index difference
to elapsed time.
\end{proof}

Define
\[
U_0
:=
\left\{
\xi\in\mathbb R^n:
\sup_{t\leq0}\|S_A(t,0)\xi\|<\infty
\right\},
\qquad
S_0
:=
\left\{
\xi\in\mathbb R^n:
\sup_{t\geq0}\|S_A(t,0)\xi\|<\infty
\right\}.
\]
By the uniform bound above on $S_A(t,s)$ for $|t-s|\leq\tau$,
\[
\xi\in U_0
\quad\Longleftrightarrow\quad
\sup_{q\leq0}\|T(q,0)\xi\|<\infty,
\]
and
\[
\xi\in S_0
\quad\Longleftrightarrow\quad
\sup_{q\geq0}\|T(q,0)\xi\|<\infty.
\]

\begin{prop}
\label{prop:B-stable-unstable-spaces}
The following identities hold:
\[
U_0
=
\operatorname{Ran}P^-(0)
=
T_{\widetilde u(0)}W^u(e^-),
\qquad
S_0
=
\operatorname{Ran}(I-P^+(0))
=
T_{\widetilde u(0)}W^s(e^+).
\]
\end{prop}

\begin{proof}
The half-line dichotomies give the first equalities. For large
positive or negative time, $\widetilde u$ belongs to the
corresponding local invariant manifold, whose tangent vectors are
precisely the variational solutions bounded on that half-line.
Transporting these tangent spaces to time zero by the invertible variational evolution gives the
claimed identities; see \cite[Theorem~4.1]{CC} and \cite[Appendix~C]{MX}.
\end{proof}

\begin{lem}
\label{lem:B-adjoint-decay}
Let $\mathcal A$ denote the vector space of bounded solutions of
\eqref{eq:B-adjoint}. The evaluation map
\[
E_0:\mathcal A\to (U_0+S_0)^\perp,
\qquad
E_0\varphi:=\varphi(0),
\]
is a linear isomorphism. Moreover, there exist $C,\eta>0$ such that every
$\varphi\in\mathcal A$ satisfies
\[
\|\varphi(t)\|
\leq
Ce^{-\eta|t|}\|\varphi(0)\|.
\]
Consequently, $\varphi\in L^1(\mathbb R,\mathbb R^n)$ and
$(\varphi(q\tau))_{q\in\mathbb Z}\in
\ell^1(\mathbb Z,\mathbb R^n)$.
\end{lem}

\begin{proof}
If $\varphi$ is bounded, its constant pairing with any solution
that decays as $t\to-\infty$ must vanish. Hence
$\varphi(0)\perp U_0$. Likewise, $\varphi(0)\perp S_0$.
Conversely, $\varphi(0)\perp S_0$ means
$\varphi(0)=P^+(0)^*\varphi(0)$, while
$\varphi(0)\perp U_0$ means
$\varphi(0)=(I-P^-(0))^*\varphi(0)$. Invariance and
Lemma~\ref{lem:B-half-line-dichotomy} yield
\[
\varphi(q\tau)
=
[T(0,q)P^+(q)]^*\varphi(0)
\qquad (q\geq0),
\]
and
\[
\varphi(q\tau)
=
[T(0,q)(I-P^-(q))]^*\varphi(0)
\qquad (q\leq0).
\]
These expressions decay exponentially. The uniform bounds for the
adjoint evolution between successive sampling times then yield the
continuous-time estimate and the stated summability.
\end{proof}

\subsection{Functional characterization of transversality}
\label{subsec:functional-transversality}

A connecting orbit is transverse precisely when
$U_0+S_0=\mathbb R^n$. This condition is independent of the point
chosen on the orbit, since the variational evolution over any finite time interval is an
isomorphism.

\begin{cor}
\label{cor:B-transversality-adjoint}
The trajectory $\widetilde u(t)$ is transverse if and only if there
does not exist a nonzero bounded solution of the adjoint equation
\eqref{eq:B-adjoint}.
\end{cor}

\begin{proof}
This follows immediately from
Proposition~\ref{prop:B-stable-unstable-spaces} and
Lemma~\ref{lem:B-adjoint-decay}.
\end{proof}

\begin{cor}[Functional characterization of transversality]
\label{cor:B-functional-transversality}
The operator
\[
L:\ell^\infty(\mathbb Z,\mathbb R^n)
\longrightarrow
\ell^\infty(\mathbb Z,\mathbb R^n)
\]
defined by \eqref{eq:B-difference-operator} is Fredholm of index
$i(e^-)-i(e^+)$. More precisely, evaluation at zero gives
$\ker L\simeq U_0\cap S_0$, and
$$\operatorname{codim}\operatorname{Ran}L
=
\operatorname{codim}(U_0+S_0),
\qquad
\operatorname{ind}L
=
\dim U_0-\operatorname{codim}S_0
=
i(e^-)-i(e^+).$$
Moreover, a sequence
$Y\in\ell^\infty(\mathbb Z,\mathbb R^n)$ belongs to
$\operatorname{Ran}L$ if and only if
$$\sum_{q=-\infty}^{+\infty}
\left\langle
\varphi((q+1)\tau),Y(q)
\right\rangle
=
0$$
for every bounded solution $\varphi$ of
\eqref{eq:B-adjoint}. Finally, the connecting orbit
$\widetilde u(t)$ is transverse if and only if $L$ is surjective.
\end{cor}

\begin{proof}
By Lemma~\ref{lem:B-half-line-dichotomy}, the sequence $\{T_q\}$
admits finite-rank exponential dichotomies on both half-lines.
Moreover, since $\sup_{q\in\mathbb Z}\|T_q\|<\infty$, the domain
$V(L)$ in \cite[Theorem~1.15]{H2} coincides with
$\ell^\infty(\mathbb Z,\mathbb R^n)$. Indeed, for every
$Y\in\ell^\infty(\mathbb Z,\mathbb R^n)$,
\[
\sup_{q\in\mathbb Z}\|Y_{q+1}-T_qY_q\|
\leq
\left(1+\sup_{q\in\mathbb Z}\|T_q\|\right)\|Y\|_\infty<\infty.
\]
The Fredholm property and the formulas for the kernel, codimension,
and index therefore follow directly from \cite[Theorem~1.15]{H2}.

The same theorem shows that $Y\in\operatorname{Ran}L$ if and only if
it is orthogonal to every bounded solution of the discrete adjoint
recurrence
$\varphi(q\tau)=T_q^*\varphi((q+1)\tau)$. Since these are precisely
the sampled sequences of bounded solutions of
\eqref{eq:B-adjoint}, the stated range characterization follows.

Finally,
$\operatorname{Ran}L=\ell^\infty(\mathbb Z,\mathbb R^n)$ if and only
if $U_0+S_0=\mathbb R^n$, which is equivalent to transversality by
Proposition~\ref{prop:B-stable-unstable-spaces}. In the constant
hyperbolic case, $U_0\oplus S_0=\mathbb R^n$ and
$U_0\cap S_0=\{0\}$, so $L$ is an isomorphism.
\end{proof}

	\subsection{Uniform local estimates near equilibria}
\label{subsec:uniform-local-estimates}
We prove the uniform estimates used to control the two tails
of connecting trajectories in Section~\ref{GKS}.

\begin{proof}[Proof of Lemma~\ref{lem:local-uniform-hyperbolicity}]
Since $K_m\subset\operatorname{Int}K_{m+1}$, choose bounded open
sets $V_0,V_1$ such that
$K_m\subset V_0\Subset V_1\Subset\operatorname{Int}K_{m+1}$. 
Since $f_0\in\mathcal E^h_{m+1}$, every equilibrium of $f_0$ in
$K_{m+1}$ is hyperbolic and hence isolated. If $f_0$ has no equilibrium in
$K_{m+1}$, then $f_0$ has no zeros in $\overline{V_1}$. By compactness,
the same holds for every vector field in a sufficiently small neighborhood
of $f_0$. Choose any $0<r_0<R_0<R_1$; the remaining assertions are
vacuous. There are then no connecting orbits under consideration,
and all the sets $G_{l,p}(f_0,m)$ equal this neighborhood.
For the remainder of the proof, suppose that $f_0$ has at least one equilibrium in $K_{m+1}$. Its zero set in
$K_{m+1}$ is compact and consists of isolated points, so it is finite. Denote these equilibria by
$e_1,\ldots,e_k$, where $k\ge 1$.

Choose $R_1>0$ sufficiently small so that the closed balls
$\overline{B(e_j,R_1)}$ are pairwise disjoint and contained in
$\Omega$, and $e_j$ is the unique equilibrium of $f_0$ in
$\overline{B(e_j,R_1)}$. This gives the geometric requirement in
(ii).

We next choose the constants and radii for the uniform local estimates.
Set $A_j=Df_0(e_j)$, and let $P_j^s$ and $P_j^u$ denote the spectral
projections onto the stable and unstable generalized eigenspaces of
$A_j$, respectively. Since the finitely many matrices $A_j$ are
hyperbolic, there exist constants $M_0\geq1$ and $\alpha>0$,
independent of $j$, such that
\begin{equation}
\|e^{A_jt}P_j^s\|\leq M_0e^{-\alpha t}\quad (t\geq0),
\qquad
\|e^{A_jt}P_j^u\|\leq M_0e^{\alpha t}\quad (t\leq0).
\label{eq:uniform-hyperbolic-estimates}
\end{equation}
These are the standard exponential estimates associated with the
stable--unstable splitting of a hyperbolic matrix; see, for example,
\cite[Theorem~2.34 and Corollary~2.35]{CC}.

Fix $0<\eta<\alpha$. Choose $\varepsilon>0$ sufficiently small so
that
\begin{equation}
q_0:=\frac{2M_0\varepsilon}{\alpha}<1,
\qquad
q_\eta:=
M_0\varepsilon
\left(
\frac{1}{\alpha-\eta}
+
\frac{1}{\alpha+\eta}
\right)<1.
\label{eq:contraction-constants}
\end{equation}
By continuity of $Df_0$, we may choose $\rho>0$ sufficiently small
that $2\rho<R_1$ and
$\|Df_0(x)-A_j\|<\varepsilon/2$ whenever
$x\in B(e_j,2\rho)$, for every $j$.

Let
$P_*:=\max_{1\leq j\leq k}
\max\{\|P_j^s\|,\|P_j^u\|\}$.
Choose $0<r_0<R_0$ sufficiently small so that
$R_0+r_0<\rho$ and
\begin{equation}
M_0P_*(R_0+r_0)< (1-q_\eta)\rho.
\label{eq:radius-choice}
\end{equation}
All sets and radii introduced above are now fixed. In the remainder of the proof,
we will shrink only the neighborhood $U(f_0,m)$.

By the implicit function theorem, after shrinking $U(f_0,m)$
in the topology $\iota$ if necessary, each $e_j$ has a unique
continuation $e_j(f)$ near $e_j$, with
$e_j(f)\in B(e_j,r_0)$. Since convergence with respect to $\iota$
is $C^1$-uniform on compact subsets of $\Omega$, hyperbolicity and
the Morse index are preserved, and hence
$i(e_j(f))=i(e_j)$.

In fact, this continuation is the unique equilibrium of $f$ in
$B(e_j,R_1)$. Indeed, the implicit function theorem gives uniqueness
in a fixed smaller neighborhood of $e_j$, while the complement of
this neighborhood in $\overline{B(e_j,R_1)}$ is compact and contains
no zero of $f_0$. The positive minimum of $\|f_0\|$ on this compact
set, together with uniform convergence of $f$ to $f_0$ there,
excludes additional equilibria of $f$. The same argument on $\overline{V_1}\setminus
\bigcup_{j=1}^k B(e_j,R_1)$
shows that every equilibrium of $f$ in $\overline{V_1}$ must equal
$e_j(f)$ for some $j$. Thus (i) and (ii) hold.

We now consider the local dynamics near $e_j(f)$. For
$x=e_j(f)+z$, define
\begin{equation}
R_{j,f}(z)
:=
f(e_j(f)+z)-f(e_j(f))-A_jz.
\label{eq:nonlinear-remainder}
\end{equation}
Since $f(e_j(f))=0$, the equation $\dot x=f(x)$ becomes
$\dot z=A_jz+R_{j,f}(z)$. Moreover,
$DR_{j,f}(z)=Df(e_j(f)+z)-A_j$.
After shrinking $U(f_0,m)$ if necessary, the choice of $\rho$ and
the $C^1$-closeness of $f$ to $f_0$ imply
\begin{equation}
\|R_{j,f}(z_1)-R_{j,f}(z_2)\|
\leq
\varepsilon\|z_1-z_2\|,
\qquad
\|z_1\|,\|z_2\|\leq\rho,
\label{eq:uniform-remainder-Lipschitz}
\end{equation}
uniformly in $f\in U(f_0,m)$ and $1\leq j\leq k$. Indeed,
$e_j(f)+B(0,\rho)\subset B(e_j,2\rho)$, and
\eqref{eq:uniform-remainder-Lipschitz} follows from the mean value
formula.

The $C^1$ stable manifold theorem applied to
$\dot z=A_jz+R_{j,f}(z)$, together with
\eqref{eq:uniform-hyperbolic-estimates} and
\eqref{eq:uniform-remainder-Lipschitz}, gives local stable and
unstable manifolds for every $f\in U(f_0,m)$ and every $j$; see
\cite[Theorem~4.1]{CC} and \cite[Appendix~C]{MX}.
The constants in this construction are uniform in $f$ and
$j$, so the local neighborhoods can be chosen with a common
size. In particular, since $R_0+r_0<\rho$ and $2\rho<R_1$, they may
be chosen so that
$B(e_j,R_0)\Subset V_j(f)\Subset B(e_j,R_1)$.
Their dimensions follow from $i(e_j(f))=i(e_j)$, and the local
stable and unstable manifolds intersect only at $e_j(f)$.
This proves (iii).

It remains to prove (iv). We treat the positive-time case; the
negative-time case is analogous. Suppose that
$u(t)\in B(e_j,R_0)$ for every $t\geq t_0$, and set
$z(t)=u(t)-e_j(f)$. Then
$\|z(t)\|\leq R_0+r_0<\rho$ for all $t\geq t_0$, so
\eqref{eq:uniform-remainder-Lipschitz} is valid along the entire
trajectory. Set $\xi=P_j^sz(t_0)$.

Define
\[
C_b([t_0,\infty),\mathbb R^n)
:=
\left\{
w\in C([t_0,\infty),\mathbb R^n):
\|w\|_\infty:=
\sup_{t\geq t_0}\|w(t)\|<\infty
\right\},
\]
and define the closed ball
\[
\mathbb B_0(\rho)
:=
\left\{
w\in C_b([t_0,\infty),\mathbb R^n):
\|w\|_\infty\leq\rho
\right\}.
\]
Projecting the variation-of-constants formula onto the stable and
unstable subspaces and using the boundedness of $z$ gives
\begin{equation}
\begin{aligned}
z(t)
={}&e^{A_j(t-t_0)}\xi
+\int_{t_0}^{t}
e^{A_j(t-s)}P_j^sR_{j,f}(z(s))\,ds -\int_t^\infty
e^{A_j(t-s)}P_j^uR_{j,f}(z(s))\,ds .
\end{aligned}
\label{eq:lyapunov-perron}
\end{equation}
The last integral in \eqref{eq:lyapunov-perron} is obtained by
applying the variation-of-constants formula to the unstable component on
$[t,T]$ and then letting $T\to\infty$. The corresponding terminal
term converges to zero by
\eqref{eq:uniform-hyperbolic-estimates} and the boundedness of $z$.

For the fixed vector $\xi$, define the Lyapunov--Perron operator
$\mathcal T_\xi$ by the right-hand side of
\eqref{eq:lyapunov-perron}. By
\eqref{eq:uniform-hyperbolic-estimates},
\eqref{eq:uniform-remainder-Lipschitz}, and
\eqref{eq:contraction-constants},
\[
\|\mathcal T_\xi w_1-\mathcal T_\xi w_2\|_\infty
\leq
q_0\|w_1-w_2\|_\infty
\]
for all $w_1,w_2\in\mathbb B_0(\rho)$. Moreover,
\eqref{eq:radius-choice} implies $\mathcal T_\xi(\mathbb B_0(\rho))
\subset \mathbb B_0(\rho)$.
Hence $\mathcal T_\xi$ has a unique fixed point in
$\mathbb B_0(\rho)$. Since $z$ satisfies
\eqref{eq:lyapunov-perron} and $\|z\|_\infty<\rho$, the function
$z$ is this unique fixed point.

Next define the weighted Banach space
\[
X_\eta^+
:=
\left\{
w\in C([t_0,\infty),\mathbb R^n):
\|w\|_\eta:=
\sup_{t\geq t_0}
e^{\eta(t-t_0)}\|w(t)\|<\infty
\right\},
\]
and define $\mathbb B_\eta(\rho)
:=
\left\{
w\in X_\eta^+:
\|w\|_\eta\leq\rho
\right\}$.
Using \eqref{eq:uniform-hyperbolic-estimates} and
\eqref{eq:uniform-remainder-Lipschitz} in the weighted norm yields
\[
\|\mathcal T_\xi w_1-\mathcal T_\xi w_2\|_\eta
\leq
q_\eta\|w_1-w_2\|_\eta
\]
for all $w_1,w_2\in\mathbb B_\eta(\rho)$. Furthermore,
\eqref{eq:radius-choice} gives $\mathcal T_\xi(\mathbb B_\eta(\rho))
\subset\mathbb B_\eta(\rho)$.
Therefore, $\mathcal T_\xi$ has a unique fixed point in
$\mathbb B_\eta(\rho)$.

Since
$\mathbb B_\eta(\rho)\subset\mathbb B_0(\rho)$, the unique fixed
point of $\mathcal T_\xi$ in $\mathbb B_\eta(\rho)$ is also a fixed
point of $\mathcal T_\xi$ in $\mathbb B_0(\rho)$. By uniqueness in
$\mathbb B_0(\rho)$, this fixed point coincides with $z$. Hence
$z\in X_\eta^+$. The Lyapunov--Perron estimate gives
\begin{equation}
\|z\|_\eta
\leq
\frac{M_0}{1-q_\eta}\|\xi\|
\leq
\frac{M_0P_*(R_0+r_0)}{1-q_\eta}.
\label{eq:uniform-exponential-bound}
\end{equation}
Consequently,
\[
\|u(t)-e_j(f)\|
\leq
Ce^{-\eta(t-t_0)},\qquad t\geq t_0,
\]
where
$C=M_0P_*(R_0+r_0)/(1-q_\eta)$ is independent of
$f$, $j$, and $t_0$.

In particular, $u(t)\to e_j(f)$ as $t\to\infty$, and
$u(t)\in B(e_j,R_0)\subset V_j(f)$ for every $t\geq t_0$.
The local characterization of
$W^s_{\mathrm{loc}}(e_j(f))$ therefore gives $u(t)\in W^s_{\mathrm{loc}}(e_j(f))$, $t\geq t_0$.

For the negative-time case, interchange the stable and unstable
subspaces and repeat the preceding argument on $(-\infty,t_0]$.
This yields
$u(t)\in W^u_{\mathrm{loc}}(e_j(f))$ for every $t\leq t_0$ and
\[
\|u(t)-e_j(f)\|
\leq
Ce^{-\eta(t_0-t)},\qquad t\leq t_0.
\]
\end{proof}

	\section*{Acknowledgments}
The authors thank Professor Yi Wang of the University of Science and
Technology of China for many insightful and constructive discussions.

\end{document}